\documentclass[11pt]{amsart}

\usepackage{amsmath, amsthm, amssymb}

\usepackage{amsmath, amsthm, amssymb, dsfont}
\usepackage{epsfig}
\usepackage{array}

\usepackage[alphabetic]{amsrefs}

\usepackage{a4wide}

\usepackage{amsmath}
\usepackage{amsfonts}
\usepackage{amsmath}%
\usepackage{amsfonts}%

\usepackage{amssymb}%
\usepackage{graphicx}%

\usepackage[colorlinks=true, allcolors=blue]{hyperref}

\usepackage{amsfonts}
\usepackage{amssymb}
\usepackage{amsmath,amsxtra,amssymb}

\usepackage{amsfonts}

\usepackage{mathbbol}
\usepackage{amssymb} 
\usepackage{comment}
\DeclareSymbolFontAlphabet{\amsmathbb}{AMSb}

\newtheorem{thm}{Theorem}[section]

\newtheorem {lem} [thm]{Lemma}

\newtheorem{cor}[thm]{Corollary}

\newtheorem{defn}[thm]{Definition}

\newtheorem{prop}[thm]{Proposition}

\newtheorem{lemma}[thm]{Lemma}
\newtheorem{remark}[thm]{Remark}

\newtheorem{quest}[thm]{Question}

\newcommand{\act}{\curvearrowright}
\newcommand{\ts}{\textstyle}

\newcommand{\br}{\Bbb R}
\newcommand{\Linf}{L^{\infty}}

\newcommand{\Lone}{L^{1}}

\newcommand {\Tr} {{\textrm{Tr}}}

\newcommand {\Ind} {\mathcal{I}}
\newcommand {\id} {{\textup{id}}}

\newcommand{\Hil}{\mathsf{H}}

\newcommand{\Kil}{\mathsf{K}}
\newcommand{\sK}{\mathsf{K}}

\newcommand{\wot}{\overline{\ot}}

\newcommand{\Com}{\Delta}

\newcommand{\QG}{\mathbb{G}}
\newcommand{\QH}{\mathbb{H}}
\newcommand{\tens}{\otimes}
\newcommand{\comp}{\!\circ\!}

\newcommand{\I}{\mathds{1}}
\newcommand{\hQG}{\widehat{\mathbb{G}}}
\newcommand{\hQH}{\widehat{\mathbb{H}}}

\numberwithin{equation}{section}

\newcommand{\noop}[1]{}
\newcommand{\field}[1]{\mathbb{#1}}
\newcommand{\CC}{{\field{C}}}
\newcommand{\im}{{\operatorname{Im}}}

\renewcommand{\im}{\operatorname{im}}
\newcommand{\ot}{{\otimes}}
\newcommand{\om}{{\omega}}

\newcommand{\HH}{{\mathbb H}}

\newcommand{\ad}{{\mathrm{ad}}}

\newcommand{\ep}{{\varepsilon}}
\newcommand{\G}{{\mathbb{\Gamma}}}

\newcommand{\p}{{\mathcal{P}}}
\newcommand{\cP}{{\mathcal{P}}}
\newcommand{\fb}{{\partial_F(\G)}}

\newcommand{\cF}{\mathcal {F}}

\newcommand{\HS}{{\mathrm{HS}}}

\allowdisplaybreaks
\newcommand{\cst}{\ifmmode\mathrm{C}^*\else{$\mathrm{C}^*$}\fi}

\newcommand{\hG}{\hat{\G}}
\newcommand{\RR}{\mathbb{R}}
\newcommand{\KK}{\mathbb{K}}
\newcommand{\GG}{\mathbb{G}}

\newcommand{\vtens}{\,\bar{\otimes}\,}

\newcommand{\sM}{\mathsf{M}}
\newcommand{\sN}{\mathsf{N}}

\newcommand{\sB}{\mathsf{B}}
\newcommand{\sC}{\mathsf{C}}
\newcommand{\sV}{\mathsf{V}}
\newcommand{\sW}{\mathsf{W}}

\newcommand{\sA}{\mathsf{A}}
\newcommand{\sX}{\mathsf{X}}
\newcommand{\sY}{\mathsf{Y}}
\newcommand{\sI}{\mathsf{I}}
\newcommand{\sJ}{\mathsf{J}}

\newcommand{\hh}[1]{\widehat{#1}}

\newcommand{\ww}{\mathrm{W}}
\newcommand{\WW}{{\mathds{V}\!\!\text{\reflectbox{$\mathds{V}$}}}}
\newcommand{\Ww}{\mathds{W}}
\newcommand{\wW}{\text{\reflectbox{$\Ww$}}\:\!} 

\newcommand{\GGamma}{\mathbb{\Gamma}}
\newcommand{\LLambda}{\mathbb{\Lambda}}

\DeclareMathOperator{\Mor}{Mor}
\DeclareMathOperator{\M}{M}

\DeclareMathOperator{\Aut}{Aut}

\DeclareMathOperator{\Ltwo}{\mathnormal{L}^2\;\!\!}

\DeclareMathOperator{\c0}{c_0}
\DeclareMathOperator{\cu}{c}
\DeclareMathOperator{\linf}{\ell^\infty\;\!\!}
\DeclareMathOperator{\Irr}{Irr}
\DeclareMathOperator{\C}{C}
\numberwithin{equation}{section}

\newcommand{\NN}{\mathbb{N}} 

\newcommand{\FO}{\mathbb{F}O} 
 
\newcommand{\Cst}{$C^*$-\relax}
\newcommand{\qdim}{\operatorname{dim_{\text{$q$}}}}

\newcommand{\qTr}{\operatorname{qTr}}
\newcommand{\qtr}{\operatorname{qtr}}

\DeclareMathOperator{\Pol}{Pol}

\title{Noncommutative Furstenberg boundary}

\author[M. Kalantar]{Mehrdad Kalantar}
\address{Mehrdad Kalantar, Department of Mathematics, University of Houston, USA}
\email{kalantar@math.uh.edu}

\author[P. Kasprzak]{Pawe{\l} Kasprzak}
\address{Pawe{\l} Kasprzak, Department of Mathematical Methods in Physics, Faculty of Physics, University of Warsaw, Poland} 
\address{ Laboratory of NMR Spectroscopy, Centre of New Technologies, University of Warsaw, Poland}
\email{pawel.kasprzak@fuw.edu.pl}

\author[A. Skalski]{Adam Skalski}
\address{Adam Skalski, Institute of Mathematics of the Polish Academy of Sciences, ul.\ \'Sniadeckich 8, 00-656 Warszawa, Poland}
\email{a.skalski@impan.pl}

\author[R. Vergnioux]{Roland Vergnioux}
\address{Roland Vergnioux, Normandie Univ, UNICAEN, CNRS, LMNO, 14000 Caen, France} 
\email{roland.vergnioux@unicaen.fr}

\begin{document}

\begin{abstract}
	We introduce and study the notions of boundary actions and of the Furstenberg boundary of a discrete quantum group. As for classical groups, properties of boundary actions turn out to encode significant properties of the operator algebras associated with the discrete quantum group in question; for example we prove that if the action on the Furstenberg boundary is faithful, the quantum group $\C^*$-algebra admits at most one KMS-state for the scaling automorphism group. To obtain these results we develop a version of Hamana's theory of injective envelopes for quantum group actions, and  establish several facts on relative amenability for quantum subgroups. We then show that the Gromov boundary actions of free orthogonal quantum groups, as studied by Vaes and Vergnioux, are also boundary actions in our sense; we obtain this by proving that these actions admit unique stationary states. Moreover, we prove these actions are faithful, hence conclude a new unique KMS-state property in the general case, and a new proof of unique trace property when restricted to the unimodular case. We prove equivalence of simplicity of the crossed products of all boundary actions of a given discrete quantum group, and use it to obtain a new simplicity result for the crossed product of the Gromov boundary actions of free orthogonal quantum groups.
\end{abstract}



\maketitle


\section{Introduction}

The concept of boundary actions in topological dynamics was introduced by Furstenberg in 1950s. It was also Furstenberg who noted the existence of a universal boundary action for any locally compact group $G$; nowadays this action, as well as the relevant space, is called the Furstenberg boundary of $G$. In the last few years, beginning from the article \cite{KalKen} this notion found  unexpected and groundbreaking  applications in the study of operator algebras associated with discrete groups, in particular leading to a resolution of an old open problem regarding the relationship between the simplicity of the group $\C^*$-algebra and the uniqueness of its trace (see \cite{KalKen}, \cite{BKKO}, \cite{LB} and also the survey \cite{Raum}). In particular it was shown in these works that for a discrete group $\Gamma$ its reduced group $\C^*$-algebra $\C^*_r(\Gamma)$ is simple if and only if the action of $\Gamma$ on its Furstenberg boundary is free, and that $\C^*_r(\Gamma)$ admits unique trace if and only if this action is faithful, if and only if the amenable radical of $\Gamma$ is trivial.

As has been known for over twenty years now, another source of operator algebras sharing many properties with these related to discrete groups is provided by the theory of compact (equivalently, discrete) quantum groups, as initiated by Woronowicz (\cite{Wor}), with the quantum theory encompassing its classical counterpart. Of particular interest is the class of universal quantum groups of Van Daele and Wang  (\cite{VDW}), which leads to operator algebras behaving in many ways as these associated with the classical free groups (\cite{VaesVergnioux}), but recently shown by Brannan and the last named author of this paper to be non-isomorphic to these  at the von Neumann algebra level (\cite{BrannanVergnioux}).

In view of the statements in the beginning of this introduction it is natural to investigate a theory of boundary actions of discrete quantum groups and aim to apply it to operator algebraic questions. This is the subject of this paper, in which we introduce the notion of quantum boundary actions, establish the existence of the Furstenberg boundary for each discrete quantum group, record several consequences of properties of the action on the Furstenberg boundary for the features of the related operator algebras, and finally show that a geometric boundary action of the free orthogonal quantum group, as constructed in \cite{VaesVergnioux}, is a  faithful boundary action in our sense, which has several strong consequences.

\bigskip

Development of the theory in the quantum context offers several technical and conceptual difficulties. The notion of the Furstenberg boundary fits naturally with Hamana's work on injective envelopes -- which was a very important observation made in \cite{KalKen} -- so that it is suitable for noncommutative generalizations. This allows us  to establish the existence and the uniqueness of the noncommutative Furstenberg boundary, and use it for example to characterise the amenability of the quantum group in question. One should note that the relationship between actions of a discrete quantum group $\G$ on a $\C^*$-algebra $\sA$ and the $\ell^1(\G)$-structures on $\sA$ is subtler in comparison to the classical case. 

More importantly, quantum group actions in general do not admit `kernels' viewed as quantum subgroups, but rather `cokernels', identified with quantum subgroups of the dual quantum group. This is a key conceptual difference, and as its consequence in the main theorems regarding $\C^*$-simplicity  or the uniqueness of the trace (in the unimodular case) we were only able to obtain one-sided implications; however these are the ones that can be used to deduce something about the operator algebras from the knowledge of boundary actions, which is our main aim. 

The difference mentioned above forces us also to study various notions of relative amenability/coamenability in the quantum world, which itself is of independent interest. In particular we show 
that certain relations between amenability of a subgroup $\LLambda$ of a discrete quantum group $\G$ and existence of ucp liftings with particular properties, classically rather straightforward, persist also in the quantum world, but require far more advanced proofs, for example exploiting a natural equivalence relation on irreducible representations, introduced in \cite{Vergnioux_Amalgamated} and later studied in \cite{Kenny}.  

Another phenomenon, this time visible only in the quantum world, is that the faithfulness of the action on the Furstenberg boundary turns out to be related to the uniqueness of the KMS-state for the scaling automorphism group -- which in the unimodular case is the same as the uniqueness of the trace, and hence can be viewed as a natural generalization of the unique trace property. 

Finally, in comparison with the classical case, it is much more difficult to produce non-trivial examples of boundary actions, or rather to show that certain geometric boundaries, as constructed in several examples, starting from \cite{VaesVergnioux} (see also \cite{VaesVDW}), are boundary actions in our sense. To address this question, we produce a criterion based on unique stationarity (the concept which was recently very successfully applied in the classical framework by Hartman and the first-named author in \cite{YairMehrdad}), and then prove that the unique stationarity indeed holds for the `Gromov boundary' action of the free orthogonal quantum group, studied in \cite{VaesVergnioux}. With this result in hand we can exploit several general theorems obtained earlier for example to provide the `Ozawa-type' embeddings of the exact group $\C^*$-algebra $\C(O_N^+)$ inside a nuclear $\C^*$-algebra contained in the injective envelope of $\C(O_N^+)$.

\bigskip

The main results of the paper are Theorem \ref{thm:rel_amen}, concerning relative amenability for discrete quantum subgroups; Theorem \ref{thm:universalF}, establishing the existence of the noncommutative Furstenberg boundary for arbitrary discrete quantum group; Theorem \ref{thm:faith-Fur-bnd-->unq-trc} on faithfulness of the boundary action implying the existence of at most one KMS-state for the scaling automorphism group; Theorems \ref{embed:crossed} and \ref{thm:c-star-simplicity} on the Ozawa type embedding of the crossed product by boundary actions and connections between simplicity of the group $\C^*$-algebra of a discrete quantum group and boundary crossed products; and Theorems \ref{thm:unique stn} and \ref{thm_faithful} on the unique stationarity and faithfulness of the Gromov boundary actions of free orthogonal quantum groups, together with their corollaries.

A detailed  plan of the paper is as follows: In Section 2 we recall preliminary facts on locally compact quantum groups, establishing notation and basic conventions concerning quantum group actions, associated crossed products and quantum subgroups. We introduce there also the notions of Poisson transforms and the co-kernel for a given discrete quantum group action. Section 3 is devoted to the study of relative amenability and coamenability for the pairs quantum group-subgroup; among other things we characterise amenability of a discrete quantum subgroup via the existence of certain ucp lifts and establish the existence of the amenable radical in the quantum context. In Section 4 we define boundary actions of discrete quantum groups, prove the existence of the unique universal boundary, called the non-commutative Furstenberg boundary, and deduce its key properties. Then in Section 5 we characterise the co-kernel of the Furstenberg boundary action and prove that faithfulness of any boundary action leads to the unique trace property for a unimodular discrete quantum group, and for a non-unimodular one implies that there is no KMS-invariant state for the scaling automorphism group.
In Section 6 we discuss the connections between simplicity of the reduced $\C^*$-algebra of a discrete quantum group and that of the associated boundary crossed product.
In Section 7, the most concrete and technical part of the paper,  we study in detail the action of a free orthogonal quantum group on its Gromov boundary, prove that it is uniquely stationary and faithful, and deduce several consequences for the relevant operator algebras. Finally in Section 8 we list several open questions naturally arising from our work.

\section{Preliminaries}
In this section we recall some definitions and basic results from the theory of locally compact quantum groups, and gather some preliminary results that we will use in the rest of the paper.

All scalar products  will be linear on the right. The symbol $\ot$ will denote the tensor product of maps and minimal spatial tensor product of $\cst$-algebras, $\wot$ will denote the ultraweak tensor product of von Neumann algebras. Given two \cst-algebras $\sA$ and $\sB$, a \emph{morphism} from $\sA$ to $\sB$ is a $*$-homomorphism $\Phi$ from $\sA$ into the \emph{multiplier algebra} $\M(\sB)$ of $\sB$, which is \emph{non-degenerate}, i.e.~the set $\Phi(\sA)\sB$ of linear combinations of products of the form $\Phi(a)b$ ($a\in\sA$, $b\in\sB$) is dense in $\sB$. The set of all morphisms from $\sA$ to $\sB$ will be denoted by $\Mor(\sA,\sB)$. The non-degeneracy of morphisms ensures that each $\Phi\in\Mor(\sA,\sB)$ extends uniquely to a unital $*$-homomorphism $\M(\sA)\to\M(\sB)$ which we will usually denote by the same symbol and use  implicitly when composing the morphisms. On the multiplier \cst-algebras we will occasionally use apart from the norm topology also the \emph{strict topology}.
Unital completely positive will be often abbreviated to \emph{ucp}. 

For operators acting on tensor products of spaces we will use the familiar leg notation: so for example if $V$ is a vector space and $T\in L( V^{\ot 2})$ then, depending on which legs of the triple tensor product we want to act with $T$, we have the natural operators $T_{12}, T_{13}, T_{23} \in L(V^{\ot 3})$ (the notation will be also applied in a formally more complicated case of completed tensor products). Tensor flip between spaces and algebras will be denoted by $\sigma$.

\subsection{Locally compact quantum groups -- basic facts}
Throughout the paper symbols $\GG$, $\HH$ will denote \emph{locally compact quantum groups} in the sense of Kustermans and Vaes (\cite{KV}) -- we refer the reader to the latter paper, as well as \cite{KusNotes} and \cite{DKSS} for detailed definitions of the objects to be introduced below. 
A locally compact quantum group (often simply called a quantum group in what follows) $\GG$ is defined in terms of a von Neumann algebra $\Linf(\GG)$ equipped with a unital normal, coassociative *-homomorphism $\Delta:\Linf(\GG)\to \Linf(\GG) \wot \Linf(\GG)$, called the \emph{coproduct} or \emph{comultiplication}. The symbols $\varphi$ and $\psi$ will denote respectively \emph{left} and \emph{right invariant Haar weights} of $\GG$, which are unique up to a positive scalar multiple, and $\Ltwo(\GG)$ will denote the GNS Hilbert space of the left Haar weight $\varphi$ (on which $\Linf(\GG)$ acts). The Tomita-Takesaki anti-unitary conjugation associated with $\varphi$ will be denoted by $J:\Ltwo(\GG)\to\Ltwo(\GG)$.  We will also frequently use
the corresponding \cst-algebra of ``continuous functions on $\GG$ vanishing at infinity'', $\C_0(\GG)\subset \Linf(\GG)$. The comultiplication $\Delta$ restricts to a (still coassociative) morphism $\Delta\in\Mor\bigl(\C_0(\GG),\C_0(\GG)\tens\C_0(\GG)\bigr)$. Finally we have the universal object related to $\GG$, i.e.\ a \cst-algebra which we will denote by $\C_0^u(\GG)$, endowed with a comultiplication $\Delta^u\in\Mor\bigl(\C_0^u(\GG),\C_0^u(\GG)\tens\C_0^u(\GG)\bigr)$,  introduced and studied in \cite{Johanuniv}. There is  a canonical surjective \emph{reducing morphism} $\Lambda \in \Mor (\C_0^u(\GG), \C_0(\GG))$, intertwining the respective coproducts. If $\Lambda$ is injective, we say that $\GG$ is \emph{coamenable}; further we say that a locally compact quantum group $\QG$ is \emph{amenable}, if there exists a state $m \in \Linf(\QG)^*$ such that for every $\omega \in \Lone(\QG) := \Linf(\QG)_*$ we have
\[ m \circ ((\omega \ot \id)\circ \Com) =  m \circ ((\id \ot \omega)\circ \Com) = \omega(\cdot)\I. \]

A fundamental object in the study of $\GG$ turns out to be the \emph{left multiplicative unitary} $\ww\in B(\Ltwo(\GG) \ot \Ltwo(\GG))$, which satisfies the pentagonal equation
$\ww_{12} \ww_{13} \ww_{23} = \ww_{23} \ww_{12}$ (see \cite[Proposition 3.18]{KV}). In fact $\ww$ determines $\GG$ completely, as we have on one hand the equality: $\Linf(\GG)=\{(\id \ot \omega)(\ww): \omega \in B(\Ltwo(\GG))_*\}''$, and on the other $\ww$ implements the coproduct:
\[\Delta(x) =\ww^*(\I \ot x) \ww, \;\;\; x \in \Linf(\GG).\]
We also have the equality $\C_0(\GG)=\overline{\{(\id \ot \omega)(\ww): \omega \in B(\Ltwo(\GG))_*\}}^{\|\cdot\|}$.

As already mentioned above, the predual of $\Linf(\GG)$ is denoted, by the analogy with the classical case, $\Lone(\GG)$. It is a Banach algebra with respect to the convolution product given by the pre-adjoint of the comultiplication.

The multiplicative unitary $\ww$ allows for a straightforward description of the \emph{dual locally compact quantum group} of $\GG$, which we will denote by $\hh{\GG}$. We have $\Linf(\hh\GG)\subset B(\Ltwo(\GG))$   and  $\Linf(\hh\GG)=\{(\omega \ot \id )(\ww): \omega \in B(\Ltwo(\GG))_*\}''$. It turns out that the representation of $\Linf(\hh\GG)$   on $\Ltwo(\GG)$ can be promoted to the standard form of the von Neumann algebra $\Linf(\hh\GG)$, so that we have a canonical identification $\Ltwo(\GG) = \Ltwo(\hh\GG)$. In what follows, when we consider more than one quantum group, we will adorn the respective symbols with the upper index describing which group we refer to: so for example another (equivalent) way of defining $\hh\GG$ would be via the equality $\ww^{\hh\GG}= \sigma \left( (\ww^{\GG})^* \right)$. Since  $\Linf(\hh\GG)$ is naturally represented on $\Ltwo (\GG)$, the same holds for $\C_0(\hh\GG)$, and it can be proved that $\ww^{\GG} \in \M(\C_0(\GG) \ot \C_0(\hh\GG)) \subset \Linf(\GG) \wot \Linf(\hh\GG)$. In fact $\ww$ admits a universal version, $\WW \in \M(\C_0^u(\GG) \ot \C_0^u(\hh\GG))$, such that $\ww= (\Lambda_{\GG} \ot \Lambda_{\hh\GG})(\WW)$.

Occasionally we will also need the \emph{right multiplicative unitary}
$V \in \Linf (\hh\GG)' \wot \Linf(\GG)$, given by the formula
\[V = (\hh{J}\ot \hh{J})\ww^{\hh\GG}(\hh{J}\ot \hh{J})\] where $\hh{J}:\Ltwo(\GG)\to\Ltwo(\GG))$ is the antiunitary conjugation associated with the Haar weight $\hh\varphi$ of the dual quantum group $\hh\GG$. 
The right multiplicative unitary  implements the coproduct of $\GG$ via the formula
\[ \Delta(x) = V (x \ot \I)V^*, \;\;\; x \in \Linf(\GG).\] 
We will also sometimes denote the objects related to $\hh\GG$ simply by using hats, so for example $\hh{\varphi}$ and $\hh{\psi}$ denote the left and right Haar weights of $\hh\GG$, respectively. Occasionally we will use the notation $\C^*_r(\GG)$ for $\C_0(\hh\GG)$.

Recall that a locally compact quantum group is said to be \emph{compact} if the $\C^*$-algebra $\C_0(\GG)$ or, equivalently, $\C_0^u(\GG)$, is unital, and then these $\C^*$-algebras are denoted simply by $\C(\GG)$ and $\C^u(\GG)$. It is said to be  \emph{discrete} if $\hh\GG$ is compact. A compact quantum group $\QG$ can be fully described by its associated Hopf$^*$-algebra $\Pol(\QG)$, which is densely contained in both  $\C(\GG)$ and $\C^u(\GG)$. A discrete quantum group $\GG$ is said to be unimodular if its left and right Haar weights coincide, which is not automatic in the quantum case. This is equivalent to the fact that the dual compact quantum group $\hh\GG$ is \emph{of Kac type} -- the Haar state of $\hh\GG$ is tracial.

A \emph{(unitary) representation of $\GG$} on a Hilbert space $\Hil$ is a unitary $\mathcal{U} \in \M (\C_0(\GG) \ot \mathcal{K}(\Hil))$ such that $(\Delta \ot \id)(\mathcal{U})= \mathcal{U}_{13} \mathcal{U}_{23}$. It turns out that $\ww\in\M(\C_0(\GG)\ot\mathcal{K}(\Ltwo(\GG))$  is a unitary representation of $\GG$ on $\Ltwo(\GG)$ (the quantum analog of the left regular representation). 
For convenience we shall also call  a unitary representation its `right version', that is a unitary operator    $\mathcal{V} \in \M ( \mathcal{K}(\Hil)\ot \C_0(\GG))$  satisfying $(\id\otimes \Delta)(\mathcal{V}) = \mathcal{V}_{12}\mathcal{V}_{13}$. It turns out that the right regular representation $V\in\M(\mathcal{K}(\Ltwo(\GG))\ot\C_0(\GG))$ is an example of a unitary representation in the latter sense. 

\subsection{Actions of quantum groups on von Neumann algebras and the crossed product construction}\label{subsect.-crossedprod}
We will now recall the rudiments of the theory of actions of locally compact quantum groups on von Neumann algebras.
\begin{defn}
Let $\GG$ be a locally compact quantum group and $\sN$ be a von Neumann algebra. We say that an injective normal $*$-homomorphism $\alpha:\sN\to \Linf(\GG)\wot\sN$ is a left  action of $\GG$ on $\sN$ (or $\sN$ is a left $\GG$-von Neumann algebra) if the action equation holds: $(\Delta\ot\id)\circ\alpha = (\id\ot\alpha)\circ\alpha$. 
\end{defn}
Similarly we define right $\GG$-von Neumann algebras. In what follows we shall sometimes  write simply a $\GG$-von Neumann algebra having either a right or left case in mind, depending on the context. 

Given a $\GG$-von Neumann algebra $\sN$ one defines the corresponding fixed point subalgebra $\sN^\alpha = \{x\in\sN \mid \alpha(x) = 1\otimes x\}$ and the crossed product von Neumann algebra 
\begin{displaymath}
\GG\ltimes\sN = ((\Linf(\hh\GG)\otimes\I)\alpha(\sN))'' \subset B(\Ltwo(\GG))\wot\sN.
\end{displaymath}
The crossed product von Neumann algebra admits the right (dual) action of $\hh\GG$ denoted by $\hh\alpha:\GG\ltimes\sN\to \GG\ltimes\sN\wot\Linf(\hh\GG)$ and defined as follows:
\[\hh\alpha(y) = \hh{V}_{13}y_{12} \hh{V}^*_{13}, \;\;  y\in\GG\ltimes\sN\subset B(\Ltwo(\GG))\wot\sN, \] 
where $\hh{V}\in\Linf(\GG)'\wot\Linf(\hh\GG)$ is the right multiplicative unitary for $\hh\GG$: $\hh{V} = (J\otimes J)\ww^{\GG}(J\otimes J)$. It is clear that
\begin{itemize}
    \item if $y = x\otimes\I$ where  $x\in\Linf(\hh\GG)$, then 
$\hh\alpha(y) = \hh\Delta(x)_{13}$;
\item if $y = \alpha(z)$ where $z\in \sN$ then $\hh\alpha(y) = y_{12}$.
\end{itemize}
Given $\hh\alpha$ one defines a normal faithful operator valued weight $T:\GG\ltimes\sN\to \alpha(\sN)$, which is roughly speaking given by the formula
\[T = (\id\otimes\id\otimes\hh\varphi)\circ \hh\alpha.\]
The crossed product von Neumann algebra admits also a left $\GG$-von Neumann structure $\beta:\GG\ltimes\sN\to\Linf(\GG)\wot\GG\ltimes\sN$ given by 
\begin{equation}\label{eq1}\beta(y) = \ww^*_{12}y_{23}\ww_{12}, \;\;  y\in\GG\ltimes\sN\subset B(\Ltwo(\GG))\wot\sN .\end{equation}
It is clear that
\begin{itemize}
    \item if $y = x\otimes\I$ where  $x\in\Linf(\hh\GG)$, then 
$\beta(y) = \left(\ww^*(\I\otimes x)\ww\right)_{12}$; the left action \[\Linf(\hh\GG)\ni x\mapsto \ww^*(\I\otimes x)\ww\in\Linf(\GG)\wot\Linf(\hh\GG)\] will be referred to as the adjoint action of $\GG$ on $\Linf(\hh\GG)$ and occasionally will also be denoted by $\beta$.
\item if $y = \alpha(z)$ where $z\in \sN$ then $\beta(y) =(\Delta\otimes\id)(\alpha(z)) = (\id\otimes\alpha)(\alpha(z))$; identifying $\sN$ with the  image $\alpha(\sN)\subset\GG\ltimes\sN$ we see that $\beta|_{\sN} = \alpha$. 
\end{itemize}

\subsection{Quantum subgroups}

Given two locally compact quantum groups $\GG$ and $\HH$, a morphism from $\HH$ to $\GG$ is represented via a \cst-morphism $\pi\in \Mor(\C_0^u(\GG), \C_0^u(\HH))$ intertwining the respective coproducts:
\[ (\pi \ot \pi) \circ \Delta_{\GG} = \Delta_{\HH} \circ \pi.\]
It can be equivalently described via a \emph{bicharacter} from $\HH$ to $\GG$, i.e.\ a unitary $U \in \M (\C_0(\HH)\tens\C_0(\hh{\GG}))$ such that
\[ (\Delta_{\HH} \ot \id_{\C_0(\hh{\GG})}) (U) = U_{13} U_{23},\]
\[ (\id_{\C_0(\HH)} \ot \Delta_{\hh\GG}) (U) = U_{13} U_{12}.\]
In fact
$U= ((\Lambda^{\HH} \circ \pi) \ot\Lambda^{\hh\GG} )(\WW^{\GG})$. Each morphism $\pi$ from $\HH$ to $\GG$,
determines uniquely a \emph{dual morphism} $\hh{\pi}$ from $\hh\GG$ to $\hh\HH$
such that $(\pi \ot \id)(\WW^{\GG})= (\id \ot  \hh\pi)(\WW^{\HH})$. Finally note that although $\pi \in \Mor(\C_0^u(\GG), \C_0^u(\HH))$ describing a morphism from $\HH$ to $\GG$ need not have a reduced version $\pi_r \in \Mor(\C_0(\GG), \C_0(\HH))$ (such that $\pi_r \circ \Lambda^{\GG} = \Lambda^{\HH} \circ \pi$), if we are given  $\pi_r \in \Mor(\C_0(\GG), \C_0(\HH))$ intertwining the coproducts, then it always admits the universal version $\pi \in \Mor(\C_0^u(\GG), \C_0^u(\HH))$.
   For more information on this equivalence and other pictures of morphisms we refer to \cite{SLW12}, \cite{DKSS} (the cautious reader will note that our choice of the  "left" convention yields a corresponding difference in the definition of bicharacter). 
\begin{defn}
We say that a morphism $\pi\in \Mor(\C_0^u(\GG), \C_0^u(\HH))$ from $\HH$ to $\GG$ identifies $\HH$ with a closed quantum subgroup of $\GG$ (in the sense of Vaes) if there exists an injective normal unital $^*$-homomorphism $\gamma:\Linf(\hh\HH) \to \Linf(\hh\GG)$ such that
\[\bigl.\gamma\bigr|_{\C_0(\hh{\HH})}\comp\Lambda_{\hh{\HH}}=\Lambda_{\hh{\GG}}\comp\hh{\pi}.\]
Often in this case we simply say that $\HH$ is a closed quantum subgroup of $\GG$.
\end{defn}

The above definition is equivalent to the existence of an injective normal unital $^*$-homo--morphism $\gamma:\Linf(\hh\HH) \to \Linf(\hh\GG)$ intertwining the respective coproducts.
It then follows $\pi(\C_0^u(\GG))= \C_0^u(\HH)$ -- if the latter condition holds, we say that the underlying quantum group morphism identifies $\HH$ with a closed quantum subgroup of $\GG$ \emph{in the sense of Woronowicz}. These two notions are studied in detail in \cite{DKSS}, where in particular one can find the proofs of the facts stated above.

Note that it follows from \cite[Proposition 10.5]{BV}
that there is a bijective correspondence between closed quantum subgroups of $\GG$ and the so-called \emph{Baaj-Vaes} subalgebras of $\Linf(\hh{\GG})$ (i.e.\ those von Neumann subalgebras $\sN \subset \Linf(\hh\GG)$ for which $\Delta_{\hh\GG}(\sN)\subset\sN\wot\sN$, $\hh{R}(\sN)=\sN$ and $\hh{\tau}_t(\sN)=\sN$ for all $t\in\RR$).
More precisely, if $\sN$ is a Baaj-Vaes subalgebra of $\Linf(\hh{\GG})$ then there exists a locally compact quantum group $\HH$ such that $\sN = \Linf(\HH)$ (and the coproducts match), and more or less by definition $\hh\HH$ is a closed quantum subgroup of $\GG$.

Assume that $\HH$ is a closed quantum subgroup of $\GG$, determined by a morphism $\pi \in \Mor (\C_0^u(\GG), \C_0^u(\HH))$. Then $\HH$ acts on $\Linf(\GG)$ on the left (we will modify the language slightly and say simply that $\HH$ acts on $\GG$) by the following formula
\begin{equation} \alpha_{\HH}(x) =  U^*(\I \ot x) U,\;\;\; x \in \Linf(\GG), \label{subgroupaction}\end{equation}
where $U\in \M(\C_0(\HH) \ot \C_0(\hh\GG))$  denotes the bicharacter associated to the morphism $\pi$.
We then call the fixed point space of $\alpha_{\HH}$ the \emph{algebra of bounded functions on the quantum homogeneous space $\HH\backslash\GG $} and denote it by $\Linf (\HH\backslash\GG)$. One can similarly define the right quotient $\GG/\HH$ and actually we have $R(\Linf (\HH\backslash\GG)) = \Linf (\GG/\HH)$

Finally, we say that a von Neumann subalgebra $\sM\subset \Linf(\GG)$ is $\GG$-invariant if $\Delta(\sM)\subset \Linf(\GG)\overline{\otimes}\sM$. If $\HH$ is a closed quantum subgroup of $\GG$, then $\Linf(\GG/\HH)$ is a $\GG$-invariant von Neumann subalgebra of $\Linf(\GG)$. On the other hand if $\HH$ a closed quantum subgroup of $\widehat\GG$ then $\Linf(\widehat\HH)$ is a Baaj-Vaes subalgebra, hence a $\GG$-invariant von Neumann subalgebra of $\Linf(\widehat\GG)$.

\subsection{Normal quantum subgroups and quotient quantum groups}

The definition of a closed normal quantum subgroup was introduced in \cite{VainVaes}.
\begin{defn} \label{def:normal}
Let $\GG$ be a locally compact quantum group and $\KK$ its closed quantum subgroup identified by an injective morphism $\gamma:\Linf(\hh\KK) \to \Linf(\hh\GG)$. We say that $\KK$ is a normal quantum subgroup of $\GG$ if $\gamma(\Linf(\hh\KK))$ is a normal coideal in $\Linf(\hh\GG)$, that is \[\ww^* ( \I\ot \gamma(\Linf(\hh\KK)) ) \ww \subset \Linf(\GG) \overline{\ot} \gamma(\Linf(\hh\KK)).\] 
\end{defn}
Using terminology introduced in Subsection \ref{subsect.-crossedprod}, $\KK\subset \GG$ is a normal quantum subgroup of $\GG$ if $\gamma(\Linf(\hh\KK))$ is preserved by the adjoint action $\beta$ of $\GG$ on $\Linf(\hh\GG)$. 
The key consequence (and actually a characterization) of the fact that $\KK$ is a normal quantum subgroup of $\GG$  is that  $\Linf(\GG/\KK)$ is a Baaj-Vaes subalgebra of $\Linf(\GG)$, so that $  \GG/\KK$ becomes a locally compact quantum group, naturally called a \emph{quotient quantum group of $\GG$}. As $\Linf(  \GG/\KK)$ inherits the full quantum group structure from $\Linf(\GG)$, quotient quantum groups of quantum groups of Kac type are again of Kac type (we defined the Kac property only for compact quantum groups, the general definition can be found for example in \cite{EnockSchwartz}).

 The above facts lead naturally to the concept of short exact sequences, studied in detail in \cite{VainVaes} (see also \cite{KSprojections}). Denoting $\HH = \GG/\KK$ we have a short exact sequence
 \begin{equation}\label{ex1}\{e\}\rightarrow\KK\rightarrow\GG\rightarrow \HH\rightarrow\{e\}\end{equation}
together with the dual exact sequence
\begin{equation}\label{ex2}\{e\}\rightarrow\widehat\HH\rightarrow\widehat\GG\rightarrow\widehat\KK\rightarrow\{e\}\end{equation}
where the embedding $\Linf(\HH)\subset \Linf(\GG)$,  yields the identification of   $\widehat{\HH}$ with a (Vaes) closed normal subgroup of $\hat\GG$.  Thus we have a bijective correspondence between normal subgroups of $\GG$ and $\hat\GG$ given by \eqref{ex1} and \eqref{ex2}.

\subsection{Discrete  quantum groups and their actions}
Through most of the paper we will be primarily interested in discrete quantum groups. If $\G$ is a discrete quantum group, we usually just write $\ell^\infty(\G)$ for $L^\infty(\G)$, $\c0(\G)$ for $\C_0(\G)$ and $\ell^1(\G)$ for the predual of $L^\infty(\G)$. The set of equivalence classes of unitary representations of $\hh\G$ will be denoted by $\Irr(\hh\G)$;  we will often use the fact that $\ell^\infty(\G)\cong \prod_{\gamma \in \Irr(\G)} M_{n_\gamma}$, $\c0(\G) \cong \bigoplus_{\gamma \in \Irr(\G)} M_{n_\gamma}$ where we use the symbols $\prod$, $\bigoplus$ to denote $\ell^\infty$-, resp. $c_0$-direct sums of normed spaces. Discrete quantum groups are always coamenable:  $\c0(\G)\cong \cu^u_0(\G)$.  Moreover $\linf(\G) \cong \M(\c0(\G))$ and the counit character extends continuously to $\ell^\infty(\G)$.

Recall that a discrete quantum group $\G$ is amenable if and only if $\hh\G$ is co-amenable, i.e.\ the reduced and universal $\cst$-norms on the canonical Hopf$^*$-algebra $\Pol(\hh\G)$ coincide.

We will now discuss the notion of $\C^*$-algebraic actions of discrete quantum groups.

\begin{defn} 
Let $\G$ be a discrete quantum group, let $\sA$ be a $\cst$-algebra, and let $\alpha\in\Mor(\sA,\c0(\G)\tens\sA)$ be an injective morphism. We say that $\alpha$ is a (left) action of $\G$ on $\sA$ if
\begin{itemize} 
\item $(\id\otimes\alpha)\circ\alpha = (\Delta\otimes\id)\circ\alpha$;
 \item $(\c0(\G)\otimes\mathds{1})\alpha(\sA)$ spans a dense subspace of $\c0(\G)\otimes\sA$ (the Podle\'s condition holds).
\end{itemize}
In this case we say $\sA$ is a $\G$-$\cst$-algebra. 
\end{defn}
Occasionally  we will need the corresponding notion of a right action.

Given an action $\alpha$ of $\G$ on a $\C^*$-algebra $\sA$ and $a\in\sA$, $\mu\in\c0(\G)^*$ and $\nu\in\sA^*$  we define
\begin{itemize}
\item $\mu*\nu =(\mu\otimes\nu)\circ\alpha \in\sA^*$;
\item $a*\mu=(\mu\otimes\id)(\alpha(a))\in \sA $. 
\end{itemize}
Note that for $\tilde\mu\in\c0(\G)^*$ we have $(a*\mu)*\tilde\mu = a*(\mu*\tilde\mu)$.

A linear map $\Phi:\sA \to \sB$ between two $\G$-$\cst$-algebras $\sA$ and $\sB$ is said to be \emph{$\G$-equivariant} if for all $a \in \sA$ and $\mu \in \sA^*$ we have 
\[ \Phi (a* \mu) = \Phi(a) * \mu;\]
if the map $\Phi$ is itself a morphism, the above condition is equivalent to the equality $( \id \ot\Phi ) \circ \alpha_\sA = \alpha_\sB \circ \Phi$. 

\begin{remark}\label{multiplieralgebra}
	Note that for any $\C^*$-algebra $\sA$ we have $\c0(\G)\ot \sA\cong \bigoplus_{\gamma \in \Irr(\hh{\G})} M_{n_\gamma}(\sA)$. In particular if $\sM$ is a von Neumann algebra then  \[\M(\c0(\G)\ot \sM)\cong \prod_{\gamma \in \Irr(\G)} M_{n_\gamma}(\sM) \cong \ell^\infty(\G)\wot \sM.\]
This  implies that if $\sM$ is a von Neumann algebra equipped with an action  $\beta:\sM\to\ell^\infty(\G)\wot\sM$ then identifying  $\ell^\infty(\G)\wot\sM\cong \M(\c0(\G)\otimes \sM)$ we can view $\beta$ as an action of $\G$ on the $\cst$-algebra $\sM$. Indeed $\beta$ is unital; for every $\omega\in\ell^1(\G)$ and $x\in\sM$ we have $(\omega\otimes\id)\beta(x)\in\sM$ and $(\varepsilon\otimes\id)\beta(x) = x$. Thus we can use \cite[Proposition 5.8]{BSV}, and the fact that any discrete quantum group is regular, to see that the action  $\beta:\sM\to \M(\c0(\G)\otimes\sM)$ satisfies the Podle\'s condition. 
\end{remark}

Let $\G$ be a discrete quantum group and let $\alpha$ be an action of $\G$ on $\sA$. The \emph{Poisson transform} associated to $\nu\in\sA^*$ is the map $\cP_\nu: \sA\to \linf(\G)$ defined by 
\[ \cP_\nu(a) = (\id\otimes\nu)\alpha(a), \;\;a \in \sA. \]
The Poisson transforms $\cP_\nu$ are completely bounded and strict, that is continuous (on bounded subsets) with respect to the strict topology of $\sA$ and the strict topology of $\linf(\G)\cong \M(\c0(\G))$. 

Observe that for $\mu\in\ell^1(\G)$,  $\nu\in\sA^*$ and $a\in \sA$ we have
\begin{equation}\label{eq:Poi-map}
\cP_{\mu*\nu}(a) = (\id\otimes\mu)(\Delta(\cP_\nu(a))).
\end{equation}

\begin{prop}\label{Pois-map-props}
The mapping $(\nu, a) \mapsto \cP_\nu(a)$ is norm continuous (from $\sA^* \times \sA$ to $\linf(\G)$). For any $\nu\in\sA^*$ the Poisson transform $\cP_\nu$ is $\G$-equivariant. Moreover, if $\nu\in\sA^*$ is a state and $\sA$ is unital, the map $\cP_\nu$ is unital and completely positive.

Any $\G$-equivariant map from a unital $\G$-$\C^*$-algebra $\sA$ to $\linf(\G)$ is a Poisson transform, and any ucp $\G$-equivariant map from $\sA$ to $\linf(\G)$ is a Poisson transform associated to a state on $\sA$.
\end{prop}

\begin{proof}
The statements in the first paragraph are very easy to check. For the second part it suffices to note that if $\Phi:\sA\to \linf(\G)$ is $\G$-equivariant, then putting $\nu: = \epsilon \circ \Phi$ we get $\mu (\cP_\nu(a)) = \mu (\Phi(a))$ for all $\mu \in \sA^*$ and $a \in \sA$, so that $\Phi= \cP_\nu$. 
\end{proof}

The following definition will play a role in the later parts of the paper.

\begin{defn}\label{def:stationary}
Let $\G$ be a discrete quantum group acting on a $\cst$-algebra $\sA$, and let $\mu \in \ell^1(\G)$ be a state. A state $\omega \in S(\sA)$ is called $\mu$-stationary if $\mu * \omega = \omega$; in other words
\[\cP_{\omega} = (\id \ot \mu)\circ \Com \circ \cP_\omega.\]
\end{defn}

Finally we will need the notion of the co-kernel of a given action. 

\begin{defn}\label{BVact}
Let $\alpha\in\Mor(\sA,\c0(\G)\tens\sA)$ be an action of $\G$ on $\sA$.  We define the \emph{co-kernel of $\alpha$} to be the von Neumann algebra $$\sN_\alpha =\{\cP_\nu(a):\nu\in\sA^*, a\in\sA\}''\subset\linf(\G).$$  
We say that the action $\alpha$ is faithful if $\sN_\alpha = \linf(\G)$. 
\end{defn}

If $\alpha\in\Mor(\sA,\c0(\Gamma)\tens\sA)$ is an action of a classical group $\Gamma$ on a $\cst$-algebra $\sA$, then the cokernel of this action is  canonically isomorphic to the algebra of bounded functions on the quotient group $\Gamma/\ker\alpha$, where we identify $\alpha$ with a homomorphism from $\Gamma$ to $\Aut(\sA)$.

\begin{prop}\label{prop:coker->Baaj-Vaes}
Let $\alpha$ be an action of $\G$ on $\sA$. Then the co-kernel of $\alpha$ is a Baaj-Vaes subalgebra of $\linf(\G)$.
\end{prop}
\begin{proof}
Indeed, since for every $\mu \in \ell^1(\G)$ and $\nu \in \sA^*$ we have $P_\mu(P_\nu (a)) = P_{\mu*\nu}(a)\in\sN_\alpha$ and $(P_\nu(a))*\mu = P_\nu(a*\mu)\in\sN_\alpha$ we conclude that $\Delta(\sN_\alpha)\subset\sN_\alpha\vtens\sN_\alpha$.  By the results of \cite{NeshYam} we see that $\sN_\alpha$ is a Baaj-Vaes subalgebra. 
\end{proof}

Verifying that a given action is faithful may be based on a simple observation, which we formulate as a proposition for the ease of reference.

\begin{prop}\label{prop_p0}
Denote by $p_0$ the support projection of the counit in $\ell^\infty (\G)$. If $\sN \subset \linf(\G)$ is a Baaj-Vaes subalgebra and $p_0 \in \sN$, then $\sN = \linf (\G)$.	
\end{prop}
\begin{proof}
The desired fact follows from the formula 
\[ \Delta(p_0) = \sum_{\gamma \in \Irr \hh\G}  t_\gamma  t_\gamma^*,\]
where $t_\gamma \in H_\gamma \ot H_\gamma^{c}$ is a (suitably normalised) invariant vector for the tensor product of the representation $\gamma$ with its contragredient  $\gamma^c$, and the fact that this vector is necessarily non-degenerate, in the sense that slicing $t_\gamma t_\gamma^*$ on the right by functionals on $B(H_{\gamma^c})$ we can obtain the whole $B(H_\gamma)$.
\end{proof}	

The (reduced) crossed product for an action $\alpha$ of a discrete quantum group $\G$ on a $\C^*$-algebra $\sA$ is defined similarly to the von Neumann construction, see Subsection \ref{subsect.-crossedprod}: $\G\ltimes_r\sA$  is the $\C^*$-algebra that is obtained as the closed linear span of  $(\C(\hh\GGamma)\otimes\I)\alpha(\sA)$ inside $\M(\mathcal{K}(\ell^2(\G))\otimes \sA)$. Moreover $\G\ltimes_r\sA$ is equipped with the right dual action $\hh\alpha$ of $\hh\G$ and the left action $\beta$ of $\G$ (c.f. Subsection \ref{subsect.-crossedprod}). The action $\beta$ restricts to the (adjoint) action on $\C(\hh\G)$ (which we also denote by $\beta$). Let us summarize these observations   in the form of the following lemma.  
\begin{lem}\label{lemcrossed}
If  $\alpha\in\Mor(\sA,\c0(\G)\tens\sA)$ is an action of a discrete quantum group $\G$ on a unital $\cst$-algebra $A$, then the crossed product $\G \ltimes_r \sA$ is a $\G$-space (when equipped with the canonical adjoint action), and both standard embeddings of $(\C(\hh\G),\beta)$ and of $(\sA,\alpha)$ into $\G \ltimes_r \sA$ are $\G$-equivariant ucp maps.
\end{lem}

\section{Amenability and coamenability}
All notions of boundary actions are intimately connected to the notion of amenability. Non-triviality of boundary actions is in a sense a measure of non-amenability of a group. In fact, the main applications of theories of boundary actions of locally compact groups, as developed by Furstenberg, have been in the problems that are related to rigidity properties that are extreme opposites of amenability.

One of the main motivations behind this work is to develop new tools to tackle similar problems in the context of discrete quantum groups. 
For this, we first need to generalize some of the main properties of amenable and co-amenable subgroups to the quantum setting. 
As it turns out, some of the basic well-known facts in the group setting become rather non-trivial in the quantum world.

\subsection{Relative amenability}

\begin{defn}
Let $\G$ be a discrete quantum group, and let $\LLambda$ be a quantum subgroup of $\GGamma$. We say $\LLambda$ is \emph{relatively amenable} in $\G$ if there is a $\LLambda$-invariant mean on $\linf(\G)$.
\end{defn}

In the classical setting, a subgroup $\Lambda$ of a discrete group $\Gamma$ is relatively amenable iff it is amenable. This follows from the fact that there is a unital positive $\Lambda$-equivariant map from $\ell^\infty(\Lambda)$ into $\ell^\infty(\Gamma)$, which is easily constructed using a set of representatives for the  space of cosets, $\Gamma/\Lambda$. This reasoning is not immediately available in the non-commutative case and providing the answer requires significantly more work.

Another easily seen characterization of amenability of a subgroup in the classical setting, is the existence of an equivariant map with values in the algebra of functions on the quotient space. Motivated by this, we introduce the following definition.
\begin{defn} \label{defn:rel-amen-subalg}
Let $\G$ be a discrete quantum group. A $\G$-invariant von Neumann subalgebra $\sM$ of $\linf(\G)$ is said to be \emph{relatively amenable} in $\linf(\G)$ if there is a ucp $\G$-equivariant map $\Psi : \linf(\G)\to\sM$.
\end{defn}

Recall that if $\LLambda$ is a quantum subgroup of $\GGamma$, then $\linf(\GGamma/\LLambda)$ is an invariant von Neumann subalgebra of $\linf(\G)$. The main result of this section is the equivalence of the notions introduced above for $\linf(\GGamma/\LLambda)$ to amenability of the quantum subgroup $\LLambda$, which will be shown to hold in Theorem \ref{thm:rel_amen}.

We need some preparation before proving the above theorem. 
Let $\GGamma$ be a discrete quantum group. Given  $\gamma\in\textrm{Irr}(\hh\GGamma)$ we denote by $\Hil_\gamma$ the corresponding Hilbert space, by $U_\gamma\in B(\Hil_\gamma)\otimes \C(\hh\GGamma)$ the unitary representation and by $\pi_\gamma:\linf(\GGamma)\to B(\Hil_\gamma)$ the representation of $\linf(\GGamma)$ on $\Hil_\gamma$, where 
\[U_\gamma = (\pi_\gamma\otimes\id)(\ww^*).\] There is a unique state $\qTr_\gamma:\linf(\GGamma)\to\CC$ satisfying \[\qTr_\gamma(x)\I_{\Hil_\gamma} = (\id\otimes h^{\hh\GGamma})(U^*_\gamma(\pi_\gamma(x)\otimes\I) U_{\gamma}) \] for all $x\in\linf(\G)$, where $h^{\hh\GGamma}$ denotes the Haar state of $\hh\GGamma$. We will need the following result, which is essentially contained in \cite[Lemma 2.2]{Izumi} (see also  \cite[Subsection 4.2]{Tomatsu}), but for the reader's convenience, as we are using somewhat different conventions, the proof will be presented below. For any von Neumann algebra $\sM$ we denote its centre by $Z(\sM)$.
\begin{thm}[Izumi]\label{thm_center}
If $y\in Z(\linf(\GGamma))$  then $(\id\otimes \qTr_\gamma)\Delta(y)\in Z(\linf(\GGamma))$.
\end{thm}
\begin{proof}
Let us recall that $\ww\in\ell^\infty(\GGamma)\wot L^\infty(\hh\GGamma)$ and $\Delta(x)=\ww^*(\I\otimes x)\ww$ for all $x\in\ell^\infty(\GGamma)$. 
 Clearly $(\id\otimes \qTr_\gamma)\Delta(y)\in Z(\linf(\GGamma))$ if and only if 
 \[\ww^*((\id\otimes \I\cdot\qTr_\gamma)\Delta(y))\ww  =(\id\otimes \I\cdot\qTr_\gamma)\Delta(y). \]
 Since \begin{align*}(\id\otimes \I\cdot\qTr_\gamma)\Delta(y)& =(\id\otimes \id\otimes h^{\hh\GGamma})(U_{23}^*(\id\otimes\pi_{\gamma}\otimes\id)(\Delta(y)_{12})U_{23})\\& =(\id\otimes \pi_\gamma\otimes h^{\hh\GGamma})(\ww_{23}\ww^*_{12}(\I\otimes y\otimes\I)\ww_{12}\ww^*_{23}).
   \end{align*}
our theorem will follow when we establish the following equality 
  \begin{align*}(\id\otimes\pi_\gamma\otimes h^{\hh\GGamma}\otimes\id)&(\ww^*_{14}\ww_{23}\ww^*_{12}(\I\otimes y\otimes\I\otimes\I)\ww_{12}\ww^*_{23}\ww_{14})\\&=(\id\otimes\pi_\gamma\otimes h^{\hh\GGamma}\otimes\id)(\ww_{23}\ww^*_{12}(\I\otimes y\otimes\I\otimes\I)\ww_{12}\ww^*_{23}).\end{align*}
Indeed, we compute:
 \begin{align*}(\id\otimes\pi_\gamma\otimes h^{\hh\GGamma}\otimes\id)&(\ww^*_{14}\ww_{23}\ww^*_{12}(\I\otimes y\otimes\I\otimes\I)\ww_{12}\ww^*_{23}\ww_{14})\\&=
 (\id\otimes\pi_\gamma\otimes h^{\hh\GGamma}\otimes\id)(\ww^*_{14}\ww_{23}\ww^*_{12}\ww_{23}^*(\I\otimes y\otimes\I\otimes\I)\ww_{23}\ww_{12}\ww^*_{23}\ww_{14})\\&=
 (\id\otimes\pi_\gamma\otimes h^{\hh\GGamma}\otimes\id)(\ww^*_{14}\ww_{13}^*\ww_{12}^*(\I\otimes y\otimes\I\otimes\I) \ww_{12}\ww_{13}\ww_{14})\\&=
 (\id\otimes\pi_\gamma\otimes h^{\hh\GGamma}\otimes\id)((\id\otimes\id\otimes\Delta^{\textrm{op}}_{\hh\GGamma})(\ww_{13}^*\ww_{12}^*(\I\otimes y\otimes\I) \ww_{12}\ww_{13}))
 \\&=
 (\id\otimes\pi_\gamma\otimes h^{\hh\GGamma}\otimes\id)(\ww_{13}^*\ww_{12}^*(\I\otimes y\otimes\I\otimes\I) \ww_{12}\ww_{13})\\&=
 (\id\otimes\pi_\gamma\otimes h^{\hh\GGamma}\otimes\id)(\ww_{23}\ww^*_{12}\ww_{23}^*(\I\otimes y\otimes\I\otimes\I) \ww_{23}\ww_{12}\ww_{23}^*)
 \\&=
 (\id\otimes\pi_\gamma\otimes h^{\hh\GGamma}\otimes\id)(\ww_{23}\ww^*_{12}(\I\otimes y\otimes\I\otimes\I) \ww_{12}\ww_{23}^*)
 \end{align*}
 where the first and the last equality hold since $y$ is central, the second and the fifth use the pentagonal equation for $\ww$, and the fourth equality follows from the invariance of the Haar measure of $\hh\GGamma$.  
\end{proof}
Suppose that $\LLambda$ is a quantum subgroup of $\GGamma$ and $\pi:\linf(\GGamma)\to \linf(\LLambda)$ is the corresponding surjective $*$-homomorphism satisfying $(\pi\otimes\pi)\circ\Delta^\GGamma = \Delta^\LLambda\circ \pi$. Let $\I_\LLambda\in Z(\linf(\GGamma))$ be the central carrier of $\pi$. 
By the theorem above  we have $(\id\otimes \qTr_{\gamma^c})\Delta^\GGamma(\I_\LLambda)\in Z(\linf(\GGamma))$ for all $\gamma\in\textrm{Irr}(\hh\GGamma)$. Moreover since $\Delta^\GGamma(\I_\Lambda)\in\linf(\LLambda\backslash\GGamma)\vtens\linf(\GGamma/\LLambda)$ we conclude that \begin{equation}\label{qtr}(\id\otimes \qTr_{\gamma^c})\Delta^\GGamma(\I_\LLambda)\in Z(\linf(\LLambda\backslash\GGamma))\end{equation}

In what follows we shall denote by $\sim_{\LLambda}$  the equivalence relation  on $\Irr(\hh{\GGamma})$ induced by $\LLambda$.  Recalling that $\Irr(\hh\LLambda)\subset  \Irr(\hh\GGamma)$,  we write $\tau\sim_\LLambda\sigma$ if there exists $\rho\in\Irr(\hh\LLambda)$ such that $\tau\subset \sigma\otimes\rho$. Equivalently (cf \cite{Vergnioux_Amalgamated} and \cite[Theorem 5.6]{Kenny}) \begin{equation}\label{vvrel}\tau\sim_\LLambda\sigma\textrm{ if } (\pi_\sigma\otimes\pi_{\tau^c})\Delta^\GGamma(\I_\LLambda)\neq  0,\end{equation} 
where $\tau^c$ denotes the contragredient representation of $\tau$.

Let $(q_i)_{i\in\mathcal{I}}$ be the maximal family of non-zero central minimal projections in $\linf(\LLambda\backslash\GGamma)$.  The equivalence relation in $\mathcal{I}$ as defined in \cite{Kenny} will be denoted by $\sim$. We shall use  following property of $\sim$  (c.f. \cite[Theorem 5.6]{Kenny} and  \cite[Theorem 5.2]{Kenny}): given $i\in\mathcal{I}$ there exists $\gamma\in\Irr(\hh\GGamma)$ such that
\begin{equation}\label{igam}\sum_{\sigma\sim_\LLambda\gamma}p_\sigma=\sum_{j\sim i}q_j.\end{equation}
Moreover   $\sigma\in\textrm{supp}(q_j)$, i.e.\ $\pi_\sigma(q_j)\neq 0$. 
\begin{lem}\label{eqrel}
Let $i\in\mathcal{I}$ and $\gamma\in\Irr(\hh\GGamma)$ be related as in \eqref{igam}. Suppose that \begin{equation}\label{igam1}\sum_{\sigma\sim_\LLambda\gamma}t_\sigma p_\sigma=\sum_{j\sim i}s_j q_j\end{equation} for some $t_\sigma,s_j\in\CC$. Then $t_\sigma = s_j$ for all $\sigma$ and $j$.  
\end{lem}
\begin{proof}
Applying $\pi_\gamma$ to both sides of \eqref{igam1} we conclude that \[t_\gamma\I_{\Hil_\gamma} = \sum_{j\sim i}s_j \pi_\gamma(q_j).\] Since $ (\pi_\gamma(q_j))_{j\sim i}$ is a family of non-zero mutually orthogonal projections (c.f. the line following \eqref{igam}) we conclude that $s_j = t_\gamma$ for all $j\sim i$. Replacing $\gamma$ by $\sigma\sim_\LLambda\gamma$ in the above argument we can see that $t_\sigma = s_j$ for all $\sigma$ and $j$ as above. 
\end{proof}
\begin{cor}
For every $\gamma\in\Irr(\hh\GGamma)$ there exists $t_\gamma>0$ such that   \[(\id\otimes \qTr_{\gamma^c})\Delta^\GGamma(\I_\LLambda) = t_\gamma\sum_{\sigma\sim_{\LLambda}\gamma} p_\sigma \] 
\end{cor}
\begin{proof}
Using \eqref{vvrel} and Theorem \ref{thm_center} we conclude that there exist $t_\sigma\in\RR$ such that $(\id\otimes \qTr_{\gamma^c})\Delta^\GGamma(\I_\LLambda) = \sum_{\sigma\sim_{\LLambda}\gamma} t_\sigma p_\sigma$. On the other hand using \eqref{qtr} and \eqref{igam} we see that there exists $i\in\mathcal{I}$ and $s_j\in\RR$ such that $\sum_{\sigma\sim_{\LLambda}\gamma} t_\sigma p_\sigma = \sum_{j\sim i}s_j q_j$. We conclude using Lemma \ref{eqrel}. 
\end{proof}

This leads us to the following result.
\begin{prop} \label{Theorem:ucp lifts}
Let $\GGamma$ be a discrete quantum group and $\LLambda$ a quantum subgroup of $\GGamma$. Then there exists a normal ucp $\LLambda$-equivariant map $\Phi:\linf(\LLambda)\to\linf(\GGamma)$. 
\end{prop}

\begin{proof}
Identify  $\linf(\LLambda)$ with $\linf(\GGamma)\I_\LLambda$. 
Given $\gamma\in\textrm{Irr}(\hh{\GGamma})$ we define a ucp map  $\Phi_\gamma:\linf(\LLambda)\to\bigoplus_{\sigma\sim_\LLambda\gamma}p_\sigma\linf(\GGamma)$ by the formula \[\Phi_\gamma(x) = t_\gamma^{-1}(\id\otimes \qTr_{\gamma^c})\Delta^\GGamma(x) .\] 
Let  $X:\textrm{Irr}(\hh\GGamma)/\sim_\LLambda\to \textrm{Irr}(\hh\GGamma)$ be a section of the canonical surjection $ \textrm{Irr}(\hh\GGamma)\to \textrm{Irr}(\hh\GGamma)/\sim_\LLambda$. Now it suffices to note that  $\Phi:=\bigoplus_{x\in\textrm{Irr}(\hh\GGamma)/\sim_\LLambda}\Phi_{X(x)}:\linf(\LLambda)\to\linf(\GGamma)$ is a ucp $\LLambda$-equivariant map, that is
\[\alpha\circ\Phi = (\id\otimes\Phi)\circ \Delta^\LLambda,\]
where $\alpha:\linf(\GGamma)\to\linf(\LLambda)\vtens\linf(\GGamma)$ is the left action of $\LLambda$ on $\linf(\GGamma)$.	
\end{proof}

We are now ready to formulate and prove the main result of this section.

\begin{thm}\label{thm:rel_amen}
	Let $\GGamma$ be a discrete quantum group and $\LLambda$ a quantum subgroup of $\GGamma$. The following are equivalent:
	\begin{enumerate}
		\item
		$\LLambda$ is relatively amenable in $\G$;
		\item
		$\linf(\G/\LLambda)$ is relatively amenable in $\linf(\G)$;
		\item
		$\LLambda$ is amenable.
	\end{enumerate}
\end{thm}

\begin{proof}
(1)$\implies$(3): suppose $\LLambda$ is relatively amenable in $\G$, i.e. there is a $\LLambda$-invariant mean on $\linf(\G)$. Composing this mean with the $\LLambda$-equivariant map $\Phi:\linf(\LLambda)\to\linf(\GGamma)$ from Theorem \ref{Theorem:ucp lifts}, we obtain a $\LLambda$-invariant mean on $\linf(\LLambda)$. \\
(3)$\implies$(2): suppose $\LLambda$ is amenable, and let $m$ be an invariant state on $\linf(\LLambda)$. Denote by $\rho^r:\linf(\G)\to\linf(\G)\overline{\otimes}\linf(\LLambda)$ the right action of $\LLambda$ on $\linf(\G)$. Consider the ucp map $\Psi:=(\id\otimes m)\circ\rho^r: \linf(\G)\to\linf(\G)$. For any $a\in\linf(\G)$ we have
\[\begin{split}
\rho^r(\Psi(a)) &= \rho^r \big((\id\otimes m)\rho^r(a)\big) = (\id\otimes \id\otimes m)\big( (\rho^r\otimes \id)\rho^r(a)\big) 
\\&= 
(\id\otimes \id\otimes m)\big( (\id\otimes\Delta_\LLambda)\rho^r(a) \big)= \I\otimes [(\id\otimes m)\rho^r(a)]
\\&=
\I\otimes\Psi(a),
\end{split}\]
which shows $\Psi(a)\in \linf(\G/\LLambda)$. Moreover, for any $a\in\linf(\G)$ and $\om\in \ell^1(\G)$ we have
\[\begin{split}
(\om\otimes \id)\big(\Delta_\G(\Psi(a))\big) &= (\om\otimes \id)\big(\Delta_\G((\id\otimes m)\rho^r(a))\big) 
= 
(\om\otimes\id\otimes m)\big((\Delta_\G\otimes \id)\rho^r(a)\big) 
\\&= 
(\om\otimes\id\otimes m)\big((\id\otimes\rho^r)\Delta_\G(a)\big) = 
(\om\otimes[(\id\otimes m)\circ\rho^r])\Delta_\G(a)
\\&=
\Psi\big((\om\otimes \id)\Delta_\G(a)\big),
\end{split}\]
which shows that $\Psi$ is $\G$-equivariant. Hence $\linf(\G/\LLambda)$ is relatively amenable in $\linf(\G)$.\\
(2)$\implies$(1): let $\Psi : \linf(\G)\to\linf(\G/\LLambda)$ be a ucp $\G$-equivariant map. 
The restriction of the co-unit $\varepsilon$ of $\linf(\G)$ to $\linf(\G/\LLambda)$ is $\LLambda$-invariant. Hence $\tilde m= \varepsilon\circ \Psi$ is a $\LLambda$-invariant state on $\linf(\G)$. \end{proof}

\bigskip

Next we investigate the connection of the above notion of relative amenability for invariant subalgebras with respect to the structure of the dual compact quantum group $\GG = \widehat\G$. Recall indeed that invariant subalgebras of $\ell^\infty(\G)$ can also arise as subalgebras $\ell^\infty(\widehat\HH)$ associated with closed quantum subgroups $\HH$ of $\GG$.

\begin{defn}\label{rel_coam}
Let $\QG$ be a compact quantum group and let $\QH$ be a closed quantum subgroup of $\QG$ realized by $\pi\in\Mor(\C^u(\QG), \C^u(\QH))$. Denote 
\begin{itemize}
    \item $\C^u(\QH\backslash\QG) =\{x\in \C^u(\QG):(\pi\otimes\id)\Delta^u(x) = \I\otimes x\}$;
    \item $\C(\QH\backslash\QG) = \Lambda_\QG(\C^u(\QH\backslash\QG))$.
\end{itemize}
We say that $\QH\backslash\QG$ (which is in general just a symbol, not a compact quantum group) is a  co-amenable quantum quotient  of  $\QG$ if the restriction $\varepsilon|_{\C^u(\QH\backslash\QG)}:\C^u(\QH\backslash\QG)\to\CC$ admits a reduced version, i.e. there exists  $\varepsilon^r:\C(\QH\backslash\QG)\to\CC$ satisfying \[\varepsilon^r\circ{\Lambda_\QG}|_{\C^u(\QH\backslash\QG)} = \varepsilon|_{\C^u(\QH\backslash\QG)}\] where $\Lambda_{\QG}:\C^u(\QG)\to \C(\QG)$ is the reducing surjection.
\end{defn}

\begin{remark}\label{rem:coam-quo1}
Let $\QG$ be a compact quantum group and let $\QH\subset \QG$ be a normal quantum subgroup. The quotient $\QH\backslash\QG$ viewed as a compact quantum group will be denoted by $\mathbb{L}$. Let us note that the quotient $\QH\backslash\QG$ is co-amenable if and only if $\mathbb{L}$ is a co-amenable quantum group. Indeed the  exact sequence of quantum groups
\[\{e\}\to\QH\to\QG\to\mathbb{L}\to\{e\}\]
has the dual version 
\[\{e\}\to\hh{\mathbb{L}}\to\hh\QG\to\hh\QH\to\{e\}.\]
Using \cite[Theorem 3.2]{Induction} we see that $\C^u(\QH\backslash\QG)$ can be identified with $\C^u(\mathbb{L})$.
\end{remark}

\begin{thm}\label{thm:rel_amen_coamen}
Let $\QG$ be a compact quantum group and let $\QH\subset \QG$ be a normal quantum subgroup. Then the normal Baaj-Vaes subalgebra $\linf(\hh\QH)\subset\linf(\hh\QG)$ is relatively amenable iff $\QH\backslash \QG$ is a co-amenable quotient.
\end{thm}

\begin{proof}
By the above remark we have $\linf(\hh\QH) = \linf(\hh\QG/\hh{\mathbb{L}})$. Thus, by Theorem \ref{thm:rel_amen}, the relative amenability of $\linf(\hh\QH)$ is equivalent to amenability of $\hh{\mathbb{L}}$, hence equivalent to co-amenability of the compact quantum group $\mathbb{L}$ (\cite{Tomatsu06}). The latter is equivalent to co-amenability of the quotient $\QH\backslash \QG$ as mentioned in Remark \ref{rem:coam-quo1}.
\end{proof}

We do not know the relation between relative amenability of $\linf(\hh\QH)$ in $\linf(\hh\QG)$ and the co-amenability of the quotient $\QH\backslash \QG$ in general for non-normal quantum subgroups $\QH\subset \QG$.

\begin{thm}\label{thm_fun}
Let $\QG$ be a compact quantum group and $\QH\subset \QG$ be a compact quantum subgroup  given by $\pi\in\Mor(\C^u(\QG),\C^u(\QH))$. Then $\QH\backslash \QG $ is co-amenable if and only if  $\pi\in\Mor(\C^u(\QG),\C^u(\QH))$ admits a reduced version, that is  there exists $\tilde\pi\in\Mor(\C(\QG),\C(\QH))$ such that $\tilde\pi\circ\Lambda_{\QG} = \Lambda_{\QH}\circ\pi$. 
\end{thm}
\begin{proof}
Suppose that $\QH\backslash\QG$ is a co-amenable quantum quotient. In order to prove that $\pi$ admits the reduced version it is enough to show that $\Lambda_{\QH}\circ\pi$ viewed as a map from $\C^u(\QG)$ to  $B(\Ltwo(\QH))$ admits a reduced version  $\C(\QG)\to B(\Ltwo(\QH))$.  Note that the map  $\C^u(\QG)\to B(\Ltwo(\QH))$ is obtained by the GNS construction applied to the state $h_{\QH}\circ\pi\in \C^u(\QG)^*$ where $h_{\QH}$ is the Haar state on $\C^u(\QH)$.  

Let  $E:  \C(\QG)\to \C(\QH\backslash\QG)$ be the conditional expectation and $E^u: \C^u(\QG)\to \C^u(\QH\backslash\QG)$  its universal version:\[E^u(c) = (h_{\QH}\circ\pi\otimes\id)\Delta^u(c)\] for all $c\in \C^u(\QG)$ and note that $E(\Lambda_{\QG}(c)) = \Lambda_{\QG}(E^u(c))$. 
Let us apply the Rieffel induction (\cite{Rieffel}) to  $\varepsilon_r:\C(\QH\backslash\QG)\to \CC$ and denote the induced representation  by 
  $\pi_R:\C(\QG)\to B(\Hil_R)$, where   $\Hil_R =  \C(\QG)\otimes_{\C(\QH\backslash\QG)}\CC$. Clearly $\Omega = \I\otimes 1$ is the cyclic vector for $\pi_R$. Let us fix $a,b\in \C(\QG)$ and $a^u,b^u\in \C^u(\QG)$ such that $\Lambda_\QG(a^u) = a$ and $\Lambda_\QG(b^u) = b$.   The scalar product in $\Hil_R$  is given by 
\begin{align*}
    \langle a\otimes 1,b\otimes 1\rangle &=\varepsilon^r(E(a^*b))=\varepsilon^r(\Lambda_{\QG}(h_{\QH}\circ\pi\otimes\id)(\Delta^u(a^*_ub_u)))=
    h_{\QH}(\pi(a_u^*b_u))
\end{align*}
In particular the state $\omega:\C(\GG)\ni a\to \langle \Omega|\pi_R(a)\Omega\rangle\in\CC$ satisfies $h_\QH(\pi(a_u)) = \omega(a)$ and $\pi_R$ is the GNS representation assigned to $\omega$. This shows that   $\pi_R$ is  the reduced version of $\pi$. 

Conversely, if a given morphism $\pi\in\Mor(\C^u(\QG),\C^u(\QH))$ identifying $\HH$ with a closed quantum subgroup of $\QG$ admits a reduced version $\tilde\pi\in\Mor(\C(\QG),\C(\QH))$ then $\QH\backslash\QG$ is co-amenable. Indeed,  for every $x\in \C(\QH\backslash\QG)$ we have  $(\tilde\pi\otimes\id)(\Delta(x)) = \I\otimes x$. In particular  $\Delta_\QH(\tilde\pi(x)) = \I\otimes \tilde\pi(x)$ and $\tilde\pi(x)\in\CC\I$. It is easy to check that $\tilde\pi(x) = \varepsilon(x^u)\I$ where $x^u\in \C^u(\QH\backslash\QG)$ is such that $\Lambda_{\QG}(x^u) = x$. 
\end{proof}

\begin{thm}
Let $\QG$ be a compact quantum group and $\QH$ a closed quantum subgroup of $\QG$. Then the following are equivalent:
\begin{itemize}
    \item $\QG$ is co-amenable;
    \item $\QH$ is co-amenable quantum group and $\QH\backslash \QG$ is co-amenable quantum quotient of $\QG$. 
\end{itemize}
Moreover if $\QH_1\subset\QH_2\subset \QG $ and $\pi_{\QH_1}\in\Mor(\C^u(\QG),\C^u(\QH_1))$ admits a reduced version, 
then $\pi_{\QH_2}\in\Mor(\C^u(\QG),\C^u(\QH_2))$ also admits a reduced version.
\end{thm}
\begin{proof}
If $\QG$ is co-amenable then using \cite{Tomquot} we see that $\HH$ is co-amenable. Clearly the quotient is $\QH\backslash\QG$ is co-amenable in this case. Conversely suppose that $\QH$ and $\QH\backslash\QG$ are 
co-amenable. Then $\pi\in\Mor(\C^u(\QG),\C^u(\QH))$ admits the reduced version $\tilde\pi\in\Mor(\C(\QG),\C(\QH))$ and thus $\varepsilon_\QH\circ\tilde\pi$ defines the  reduced counit on $\C(\QG)$.

The last statement of the theorem follows from Theorem \ref{thm_fun} and the fact that $\C^u(\QG/\QH_2)\subset \C^u(\QG/\QH_1)$. 
\end{proof}
\subsection{The amenable radical}\label{sec_amrad}
We finish the section by proving  a quantum version of Day's result from \cite{Day}, showing that  that every discrete quantum group $\G$ has a unique maximal normal amenable closed quantum subgroup, which contains all normal amenable closed quantum subgroups of $\G$.

We will need the following special case of a recent result of Crann.

\begin{thm}[\cite{Jason}, Theorem 3.2]\label{Theorem:Jason}
	Let $\QG$ be a locally compact quantum group with a closed normal quantum subgroup $\QH$. Then $\QG$ is amenable if and only if both $\QH$ and $\QG/\QH$ are amenable.
\end{thm}
\begin{proof}
	It suffices to observe that the fact that in the context above the amenability of the $\QG$ action on $\QG/\QH$ as defined in \cite{Jason} is the same as the amenability of the locally compact quantum group $\QG/\QH$ and apply Theorem 3.2 of that paper.
\end{proof}

Suppose then that $\QG$ is a locally compact quantum group and we have a collection $(\QH_i)_{i \in \Ind}$ of closed quantum subgroups of $\QG$, so that we have the respective normal embeddings $\gamma_i: \Linf(\hQH_i) \to \Linf(\hQG)$. We say that $\QG$ is \emph{generated} by the family $(\QH_i)_{i \in \Ind}$ if $\Linf(\hQG)$ is the smallest von Neumann algebra containing all $\gamma_i( \Linf(\hQH_i))$. Further if $\Ind$ is an ordered set, then we say that the net $(\QH_i)_{i \in \Ind}$ is \emph{increasing} if for all $i,j \in \Ind$ such that $i \leq j$ we have $\gamma_i(\Linf(\hQH_i))\subset \gamma_j(\Linf(\hQH_j))$.

\begin{lem}\label{Lemma:maximal}
	Let $\G$ be a discrete quantum group and let $(\QH_i)_{i \in \Ind}$ be an increasing net of amenable quantum subgroups of $\G$. Then the quantum subgroup generated by $\{\QH_i: i \in \Ind\}$ is amenable.
\end{lem}
\begin{proof}
Without loss of generality we shall assume that $\GGamma$ is generated by  by $\{\QH_i: i \in \Ind\}$. 	By Theorem 6.2 of \cite{DKSS} each $\QH_i$ is also discrete. 
	Thus we have the respective Hopf embeddings $\gamma_i: \Pol(\hQH_i)\to \Pol(\hh\G)$, which are isometries with respect to the respective reduced $\cst$-norms. Using \cite[Theorem 3.2]{Induction} in the case of $\HH_i\subset \GGamma$ these embeddings are isometric for the universal norms too. It is easy to see that  $\Pol(\hh\G)=\bigcup_{i \in \Ind} \gamma_i (\Pol(\hQH_i))$ (this follows from the orthogonality relations). In particular, since each $\HH_i$ is amenable, the universal and the reduced norm  coincide on $\Pol(\hh\G)$. Since $\Pol(\hh\G)$ is dense in $\C^u(\hh\GGamma)$ these norms must be equal and
we see that  $\hh\G$ is co-amenable, so that $\G$ is amenable.
\end{proof}

\begin{prop}
	Every discrete quantum group $\G$ admits the largest normal amenable quantum subgroup, that we call the amenable radical of $\QG$.
\end{prop}

\begin{proof}
	The proof follows the original argument of  Section 4 of \cite{Day}. 
	
	Consider $\mathcal{H}$, the collection of all normal amenable quantum subgroups of $\G$. Kuratowski-Zorn's lemma together with Lemma \ref{Lemma:maximal} show that $\mathcal{H}$ contains a maximal element, say $\QH$. Suppose then that there is $\QH'\in \mathcal{H}$ which is not contained in $\QH$. Consider, as in \cite{IsomorphismTheorems} (Definition 2.1 and Definition 1.9), the subgroup generated by $\QH$ and $\QH'$, denoted $\QH\vee \QH'$. Again, by Definitions 2.1 and 1.9 in \cite{IsomorphismTheorems} and the discussion before the latter it follows that $\QH\vee \QH'$ is a normal quantum subgroup of $\G$ (strictly containing $\QH$). Then Corollary 4.2 of \cite{IsomorphismTheorems}, applicable to discrete quantum groups, says that $(\QH \vee \QH')/\QH' = \QH/(\QH \wedge \QH')$. By Theorem \ref{Theorem:Jason} first the right hand side is amenable, and then so is $\QH \vee \QH'$. This leads to a contradiction and ends the proof.
\end{proof}

\section{Noncommutative topological boundaries and injective envelopes}

In this section we define topological $\G$-boundaries, and prove the existence of a (unique) universal $\G$-boundary, by identifying that object with the $\G$-injective envelope of the trivial action, in the sense of Hamana. 

All the $\G$-$\C^*$-algebras in this section are assumed to be unital.

\subsection{Topological $\G$-boundaries}
We begin by defining boundary actions.

\begin{defn}\label{def:boundary}
Let $\G$ be a discrete quantum group. A $\G$-$\cst$-algebra $\sA$ is a $\G$-boundary if for every state $\nu$ on $\sA$ the Poisson transform $\p_\nu$ is completely isometric.
\end{defn}

By ${\rm P}(\G)$ we shall denote  the set of all states $\mu\in\ell^1(\G)$. The next lemma connects the above notion with the classical notion of boundary actions as developed by Furstenberg. Note indeed that when $\sA$ and $\ell^\infty(\G)$ are commutative, $P_\nu$ is isometric on the self-adjoint part of $\sA$ {\bf iff} it is completely isometric, and on the other hand the set $\{\mu*\nu : \mu\in \rm{P}(\G)\}$ is weak* dense in the state space of $\sA$ for every state $\nu$ on $\sA$ {\bf iff} the action of $\G = \Gamma$ on the spectrum of $\sA$ is minimal and strongly proximal.

\begin{lem}\label{prox-act_c*}
Let $\G$ be a discrete quantum group, let $\sA$ be a $\G$-$\cst$-algebra, and let $\nu\in \sA^*$ be a state.
The Poisson transform $\cP_\nu$ is isometric on the self-adjoint part of $\sA$ if and only if the set $\{\mu*\nu : \mu\in \rm{P}(\G)\}$ is weak* dense in the state space of $\sA$.
\end{lem}

\begin{proof}
Suppose there is a state $\rho\in\sA^*$ that is not contained in the weak* closure of $\{\mu*\nu : \mu\in {\rm P}(\G)\}$. Then by the Hahn-Banach separation theorem, there is a self-adjoint element $x \in \sA$ and $\ep > 0$ such that
\[
\mu(\p_\nu (x)) = \mu*\nu (x) \leq \rho(x) - \ep \leq \|x\| - \ep, \quad \forall \mu \in {\rm P}(\G). 
\]
Since the Poisson transform $\cP_\nu$ is completely positive, $\p_\nu (x)$ is also self-adjoint, and therefore it follows $\|\p_\nu (x)\| = \sup\{|\mu(\p_\nu (x))| : \mu\in {\rm P}(\G)\} \leq \|x\| - \ep$. 

Conversely, suppose the weak* closure of the set $\{\mu*\nu : \mu\in {\rm P}(\G)\}$ contains every state on $\sA$. Let $x\in \sA$ be self-adjoint. Let $\ep>0$ be given. There is a state $\rho$ on $\sA$ such that $|\rho(x)|>\|x\|-\ep$. By weak* density of the set above, there is $\mu\in {\rm P}(\G)$ such that $|\mu(\p_\nu (x))| = |\mu*\nu(x)|>\|x\|-2\ep$. This implies $\|\cP_\nu(x)\| \geq\|x\|-2\ep$. Since $\ep$ was arbitrary, and $\cP_\nu$ is a contraction, it follows $\|\cP_\nu(x)\| = \|x\|$.
\end{proof}

We record the following simple fact as we use it in several several places later.

\begin{prop}\label{prop:bnd-act-eqvs}
Let $\sA$ be a $\G$-$\cst$-algebra. The following are equivalent:
\begin{enumerate}
\item 
$\sA$ is a $\G$-boundary;
\item
every ucp $\G$-equivariant from $\sA$ into $\linf(\G)$ is completely isometric;
\item
every ucp $\G$-equivariant map from $\sA$ into any $\G$-$\cst$-algebra $\sB$ is completely isometric.
\end{enumerate}
\end{prop}

\begin{proof}
By Proposition \ref{Pois-map-props} every ucp $\G$-equivariant map from $\sA$ into $\linf(\G)$ is the Poisson transform associated to some state on $\sA$, hence the implication (1)$\implies$(2). To show
(2)$\implies$(3), let $\psi$ be a ucp $\G$-equivariant map from $\sA$ into a $\G$-$\cst$-algebra $\sB$. Let $\nu$ be a state on $\sB$, then $\p_\nu\circ \psi$ is a ucp $\G$-equivariant map from $\sA$ into $\linf(\G)$, hence completely isometric assuming (2). Since $\p_\nu$ is a complete contraction, this implies $\psi$ is  completely isometric. The implication (3)$\implies$(1) is immediate from the definition.
\end{proof}

\begin{prop}\label{prop:subsp-bnd-->bnd}
Let $\G$ be a discrete quantum group and let $\sB$ be a $\G$-boundary. Suppose $\sA$ is a $\G$-$\cst$-algebra such that there is a completely isometric ucp $\G$-equivariant map $\iota: \sA\to \sB$. Then $\sA$ is a $\G$-boundary.
\end{prop}

\begin{proof}
Denote by $\alpha_\sA$ and $\alpha_\sB$ the actions of $\G$ on $\sA$ and $\sB$. Let $\nu$ be a state on $\sA$. Extend the state $\nu\circ\iota^{-1}$ on $\iota(\sA)$ to a state $\tilde\nu$ on $\sB$. Then 
we have $\p_\nu = \p_{\tilde\nu}\circ\iota$. Since $\sB$ is a $\G$-boundary, $\p_{\tilde\nu}$ is completely isometric, hence $\p_\nu$ is completely isometric and the proof is complete.
\end{proof}

\begin{prop}\label{prop:subgrp-bnd-->bnd}
Let $\G$ be a discrete quantum group and let $\sA$ be a $\G$-$\cst$-algebra. If $\sA$ is a $\LLambda$-boundary for some quantum subgroup $\LLambda\subset \G$ with the canonical restriction action, then $\sA$ is a $\G$-boundary.
\end{prop}

\begin{proof}Let $\pi:\linf(\G)\to\linf(\LLambda)$ be the surjective normal $*$-homomorphism which identifies $\LLambda$ as a quantum subgroup of $\G$. Let $\nu$ be a state on $\sA$. Denoting $\p^\G_\nu:\sA\to\linf(\G)$ and $\p^\LLambda_\nu:\sA\to\linf(\LLambda)$ the corresponding Poisson maps we have $\p^\LLambda_\nu = \pi\circ \p^\G_\nu$. Since $\p^\LLambda_\nu$ is completely isometric the same must hold for $\p^\G_\nu$. 
\end{proof}

\begin{prop}\label{prop:quo-bnd-->bnd}
Let $\G$ be a discrete quantum group and let $\LLambda\subset \G$ be a normal quantum subgroup. Then every $\G/\LLambda$-boundary endowed with the canonical action of $\G$ is a $\G$-boundary.
\end{prop}

\begin{proof} 
Suppose that $\sA$ is a $\G/\LLambda$ boundary via $\alpha\in\Mor(\sA,\c0(\G/\LLambda)\otimes\sA)$. Recall that $\linf(\G/\LLambda)$ is a von Neumann subalgebra of $\linf(\G)$ and that we have the corresponding injective morphism $\iota\in\Mor(\c0(\G/\LLambda) ,\c0(\G))$. Then the canonical action of $\beta\in\Mor(\sA,\c0(\G)\otimes\sA)$ of $\G$ on $\sA$ is $\beta = (\iota\otimes\id)\circ\alpha$, and for a state $\nu$ on $\sA$ the corresponding Poisson maps $\p_\nu^\G:\sA\to\linf(\G)$ and $\p_\nu^{\G/\LLambda}:\sA\to\linf(\G/\LLambda)$ satisfy $\p_\nu^\G = \iota\circ \p_\nu^{\G/\LLambda}$.
In particular $\p_\nu^{\G/\LLambda}$ is completely isometric if and only if $\p_\nu^\G$ is completely isometric. 
\end{proof}

\subsection{$\G$-Injective envelopes}
In this section we show existence of a (unique) minimal injective object in the category of $\G$-$\C^*$-algebras with ucp $\G$-equivariant maps as morphisms and completely isometric ucp $\G$-equivariant maps as injections. The construction and most arguments below follow the work of Hamana \cite{Ham85} and some simplifications made in \cite{Paulsen}, but we include all details for the convenience of the reader.

However this general categorical construction only yields an injective $\ell^1(\G)$-module object. In the classical case, for discrete group $\Gamma$, every $\ell^1(\Gamma)$-module has an obvious natural $\Gamma$-action. But this fact is by no means obvious in the quantum setting. The main new result of this section (Theorem \ref{thm:genuine-action}) is that in the case of discrete quantum groups $\G$, certain $\ell^1(\G)$-module actions, including the one on the minimal $\G$-injective object, indeed come from a $\G$-action.

Let $\mathcal G$ be the set of all ucp $\G$-equivariant maps $\phi: \ell^\infty(\G) \to \ell^\infty(\G)$. On the set $\mathcal G$ define the partial pre-order
\[
\phi \leq \psi \, \text{ if } \, \|\phi(x)\| \leq \|\psi(x)\|\, \text{ for all }\, x\in \ell^\infty(\G).
\]

\begin{prop}
The set $\mathcal G$ contains a minimal element.
\end{prop}

\begin{proof}
We show that every decreasing net $(\phi_i)_{i \in \Ind}$ in $\mathcal G$ has a lower bound. Then the result follows from the Kuratowski-Zorn Lemma.

To show this, first note that $\mathcal G$ is point-weak* compact.
Therefore, there is a subnet $(\phi_{i_j})_{j \in \mathcal{J}}$ and $\phi_0\in \mathcal G$
such that $\phi_{i_j}(x) \to \phi_0(x)$ in the weak* topology for all $x\in \ell^\infty(\G)$.
We obviously have $\|\phi_0(x)\| \leq \limsup_{i_j} \|\phi_{i_j}(x)\| = \inf_{i\in \Ind} \|\phi_{i}(x)\|$
for all $x\in \ell^\infty(\G)$. Thus, $\phi_0\leq \phi_i$ for all $i\in \Ind$.
\end{proof}

\begin{lem}\label{lem:min-Gmap->idem}
If $\phi_0$ is a minimal element of $\mathcal G$, then $\phi_0$ is an idempotent.
\end{lem}

\begin{proof}
Let $\phi^{(n)} := \frac1n \sum_{k=1}^n \phi_0^k$. 
We have $\|\phi^{(n)}(x)\| \leq \frac1n \sum_{k=1}^n \|\phi_0^k(x)\|\leq \frac1n \sum_{k=1}^n \|\phi_0(x)\| = \|\phi_0(x)\|$ for all $x\in \ell^\infty(\G)$ and $n\in\mathbb N$.
Let $\Phi: \ell^\infty(\G) \to \ell^\infty(\G)$ be a point-weak* limit point of $(\phi^{(n)})$.
Then $\phi_0 \Phi = \Phi\phi_0 = \Phi$, and moreover 
$\|\Phi(x)\| \leq \limsup_n \|\phi^{(n)}\| \leq \|\phi_0(x)\|$ for all $x\in \ell^\infty(\G)$. 
By minimality of $\phi_0$ we have $\|\Phi(x)\| = \|\phi_0(x)\|$ for all $x\in \ell^\infty(\G)$.
Therefore
\[\begin{split}
\|\phi_0(x) - \phi_0^2(x)\| &= \|\phi_0(x - \phi_0(x))\| = \|\Phi(x - \phi_0(x))\| 
\\&= \|\Phi(x) - \Phi\phi_0(x)\| = 0 
\end{split}\]
for all $x\in \ell^\infty(\G)$, which implies that $\phi_0$ is idempotent.
\end{proof}

The standard Choi-Effros product makes the image of a (minimal) ucp $\G$-equivariant idempotent as above a unital $\cst$-algebra, which naturally has an $\ell^1(\G)$-module structure. We will show in the following theorem that this $\ell^1(\G)$-module structure in fact comes from an action of $\G$. 

\begin{thm}\label{thm:genuine-action}
Let $\phi: \ell^\infty(\G) \to \ell^\infty(\G)$ be an idempotent $\G$-equivariant ucp map, let $\sX = \phi(\ell^\infty(\G))$ and let $\sA$ denote the $\cst$-algebra that is obtained from equipping $\sX$ with the Choi-Effros product $a\cdot b := \phi(ab)$, $a, b\in \sX$.
Then the restriction of the coproduct defines a $\G$-action on $\sA$ satisfying the Podle\'s condition. 
\end{thm}

\begin{proof}
In order to distinguish the abstract $\cst$-algebra $\sA$ from the image $\sX$ of $\phi:\ell^\infty(\G)\to\ell^\infty(\G)$, we shall be using the embedding  $j:\sA\to\ell^\infty(\G)$ of $\sA$ into $\ell^\infty(\G)$. The product of elements $a,b$ viewed as elements of abstract $\cst$-algebra $\sA$ will be denoted by $ab$. The definition of the Choi-Effros product can now be written as 
\[ j(ab)  =\phi(j(a)j(b)), \;\;\; a,b \in \sA.\]  Note that $j$ is a unital completely isometric map. 

Consider the Fubini product map (see e.g. \cite[Chapter 7]{EffrosRuan})
 \[\Phi:=\id\otimes_{\cF}\phi:\ell^\infty(\G)\vtens\ell^\infty(\G) \to \ell^\infty(\G)\vtens\ell^\infty(\G).\] 
Recall that $\Phi$ is a complete contraction, characterised by the formula 
 \[(\omega \ot \mu) (\Phi(x)) = \mu (\phi((\omega \ot \id)(x))), \;\;\; x \in \ell^\infty(\G)\vtens \ell^\infty(\G), \omega, \mu \in \ell^1(\G). \]
The above formula allows us to deduce that for $x$ and $\omega$ as above we have $(\omega \ot \id) \Phi(x) = \phi((\omega \ot \id)(x))$, and further that $\Phi$ is idempotent. The fact that $\phi$ is a ucp $\G$-equivariant map implies after another easy computation that
\[ \Phi \circ \Delta = \Delta \circ \phi,\]
and that the image of $\Phi$ equals 
\[Z:=\left\{z\in \ell^\infty(\G)\vtens\ell^\infty(\G) \,: \,  (\om \ot \id)(z) \in \sX \, \text{ for all } \, {\om \in  \ell^1(\G)} \right\}.\] 
We then need another identification: as $\ell^\infty(\G) \approx \prod_{\gamma \in \Irr(\hG)} M_{n_\gamma}$, we can identify canonically $\ell^\infty(\G)\vtens\ell^\infty(\G)$ with $\prod_{\gamma \in \Irr(\hG)} M_{n_\gamma}(\ell^\infty(\G))$. In this identification we have $\Phi(\prod_{\gamma \in \Irr(\hG)} m_\gamma) = 
\prod_{\gamma \in \Irr(\hG)} \phi(m_\gamma)$ and $Z \approx \prod_{\gamma \in \Irr(\hG)} M_{n_\gamma}(\sX)$. Finally use the `pointwise' application of the map $j^{-1}:\sX \to \sA$ to define the ucp isometry $J^{-1}:Z \to \prod_{\gamma \in \Irr(\hG)} M_{n_\gamma}(\sA)\approx \M(\c0(\G) \ot \sA)$.

We are finally ready to define the action $\alpha:\sA \to \M(\c0(\G) \ot \sA)$; it is informally just the suitably interpreted coproduct, and formally
\[ \alpha = J^{-1} \circ \Com \circ j.\]
It is now easy to check that for any $z, w \in Z$ we have $J^{-1}(z)J^{-1}(w) = J^{-1}(\Phi(zw))$. 
Thus finally for $a,b \in A$ we have
\[\begin{split} 
\alpha(ab) &= J^{-1} \left( \Delta (j(ab)) \right) = J^{-1} \left(\Delta (\phi(j(a)j(b)) \right) 
\\&= J^{-1} \left(\Phi(\Delta(j(a) j(b))   \right)) = (J^{-1} \circ \Phi)\left(\Delta(j(a)) \Delta(j(b)) \right)
\\&= J^{-1} \left(\Delta(j(a))\right) J^{-1} \left(\Delta(j(b)) \right) = \alpha(a) \alpha(b),
\end{split}\]
which is what we intended to prove.

To prove that $\alpha$ satisfies the Podle\'s condition $[\alpha(\sA)(\c0(\G)\otimes\I)] = \c0(\G)\otimes \sA$ we shall use \cite[Proposition 5.8]{BSV}. 
Let us first note that $\alpha$, being a unital $*$-homomorphism, is non-degenerate, i.e.\ $[\alpha(\sA)(\c0(\G)\otimes\sA)] = \c0(\G)\otimes \sA$. Moreover for all $\omega\in\ell^1(\G)$ and $a\in\sA$ we have $(\omega\otimes\id)(\alpha(a))\in\sA$. The counit $\varepsilon\in \c0(\G)^*$  satisfies $(\varepsilon\otimes\id)\alpha(a) = a$ and since $\varepsilon\in \ell^1(\G)$ we see that the action $\alpha$ is weakly continuous, that is \[\left[\left\{(\omega\otimes\id)\alpha(a):\omega\in\ell^1(\G),a\in\sA\right\}\right] = \sA.\] Since   discrete quantum  groups are regular,  \cite[Proposition 5.8]{BSV} indeed applies to  $\alpha$. 
\end{proof}

In the following, for any idempotent ucp $\G$-equivariant $\phi$ on $\ell^\infty(\G)$, we consider $\im(\phi)$ as a $\G$-$\cst$-algebra, with the $\G$-action given in Theorem \ref{thm:genuine-action}.

\begin{prop}\label{prop:rigidity&essentiality}
Suppose $\phi_0$ is a minimal element of $\mathcal G$ (which is idempotent by Lemma \ref{lem:min-Gmap->idem}). Then 
\begin{enumerate}
\item
\emph{($\G$-rigidity)} the identity map is the unique ucp $\G$-equivariant map on $\im(\phi_0)$;
\item
\emph{($\G$-essentiality)} any ucp $\G$-equivariant map $\psi : \im(\phi_0) \to \sA$ from $\im(\phi_0)$ into any $\G$-$\cst$-algebra $\sA$ is completely isometric.
\end{enumerate}
\end{prop}

\begin{proof}
(1)\, Suppose $\psi : \im(\phi_0) \to \im(\phi_0)$ is a ucp $\G$-equivariant map.
Similarly as in the proof of Lemma \ref{lem:min-Gmap->idem} above we may find a ucp $\G$-equivariant $\Psi$ acting on $\im(\phi_0)$
such that $\Psi\psi = \psi\Psi = \Psi$. 
Since $\Psi$ is a contraction we have 
$\|\Psi\phi_0(x)\| \leq \|\phi_0(x)\|$ for all $x\in \im(\phi_0)$.
Hence by minimality of $\phi_0$, it follows that 
$\|x\| = \|\phi_0(x)\| \leq \|\Psi(\phi_0(x))\|$ for all $x\in \im(\phi_0)$, which 
shows that $\Psi$ is isometric on $\im(\phi_0)$.
Thus, $\|x-\psi(x)\| = \|\Psi(x-\psi(x))\| = 0$ for all $x\in \im(\phi_0)$, which implies that $\psi=\id_{\im(\phi_0)}$.

\noindent
(2)\, Let $\sA$ be a $\G$-$\cst$-algebra and $\psi : \im(\phi_0) \to \sA$ be a ucp $\G$-equivariant map. Let $\omega\in S(\sA)$ be a state on $\sA$, and $\cP_\omega:\sA\to \ell^\infty(\G)$ the corresponding Poisson transform. Then $\phi_0\circ\cP_\omega\circ \psi$ is a ucp $\G$-equivariant map on $\im(\phi_0)$, hence the identity map by part (1). Since both $\phi_0$ and $\cP_\omega$ are completely contractive, it follows that $\psi$ is completely isometric.
\end{proof}

\begin{cor}\label{cor:min-image}
If $\phi_0$ is a minimal element of $\mathcal G$, then $\im(\phi_0)$ is minimal among subspaces of $\ell^\infty(\G)$ that are images of idempotent ucp $\G$-equivariant maps.
\end{cor}

\begin{proof}
Let $\phi:\ell^\infty(\G) \to \ell^\infty(\G)$ be a ucp $\G$-equivariant map such that $\im(\phi)$ is contained in $\im(\phi_0)$. Then the restriction of $\phi$ to $\im(\phi_0)$ is a ucp $\G$-equivariant map on $\im(\phi_0)$, hence the identity map by $\G$-rigidity (Proposition \ref{prop:rigidity&essentiality}). This implies $\im(\phi_0) \subseteq \im(\phi)$.
\end{proof}

\begin{prop}\label{prop:Hamana-bnd}
The minimal image of a ucp $\G$-equivariant idempotent on $\linf(\G)$ is unique up to isomorphism (as a $\G$-$\cst$-algebra).
\end{prop}

\begin{proof}
Suppose $\phi:\ell^\infty(\G) \to \sX$ and $\psi: \ell^\infty(\G) \to\sY$ are ucp $\G$-equivariant idempotents, and
$\sX$ and $\sY$ are minimal among subspaces of $\ell^\infty(\G)$ that are images of ucp $\G$-equivariant idempotents.

Then by $\G$-rigidity (Proposition \ref{prop:rigidity&essentiality}) the composition $\phi\psi$, when restricted to $\sX$, is the identity map. Similarly,
the restriction of $\psi\phi$ to $\sY$ is the identity map. Hence $\psi: \sX\to\sY$ is a $\G$-isomorphism (of $\cst$-algebras).
\end{proof}

Note that the construction above remains valid for any sub-class of $\G$-invariant ucp maps, as long as it is closed under composition and pointwise weak$^*$-limits. 

The next fact is standard, but extremely useful; in fact injectivity is what makes the $\G$-envelopes interesting.

\begin{prop}\label{prop:injectivity}
Suppose $\phi_0$ is an idempotent in $\mathcal G$. Then $\im(\phi_0)$ is injective in the category of $\G$-$\cst$-algebras, i.e.\ for any $\G$-$\cst$-algebras $\sA$ and $\sB$ equipped with ucp $\G$-equivariant maps $\psi:\sA \to \im(\phi_0)$ and $\iota: \sA \to \sB$, with $\iota$ completely isometric, there exists a ucp $\G$-equivariant map $\Psi: \sB \to \im(\phi_0)$ such that $\psi = \Psi \circ \iota$.
\end{prop}

\begin{proof}
As $\im(\phi_0) \subset \ell^\infty(\G)$ and $\phi_0$ is a ucp $\G$-equivariant projection, it is enough to show that  $\ell^\infty(\G)$ itself is injective as a $\G$-$\cst$-algebras. This fact is a special case of the next proposition.
\end{proof}

An injective object in the category of $\G$-$\cst$-algebras is called $\G$-injective.
Recall Remark \ref{multiplieralgebra}.

\begin{prop}\label{prop:injectivity-general}
Let $\G$ be a discrete quantum group, and let $\Hil$ be a Hilbert space. The $\C^*$-algebra $\linf(\G)\overline{\otimes} B(\Hil)$ with the $\G$-action $\Delta\otimes \id$ is $\G$-injective.
\end{prop}
\begin{proof}
Assume that we have $\G$-$\cst$-algebras $\sA$ and $\sB$ equipped with ucp $\G$-equivariant maps $\psi:\sA \to \ell^\infty(\G)\overline{\otimes} B(\Hil)$ and $\iota: \sA \to \sB$, with $\iota$ completely isometric. Put $\phi:=(\epsilon\otimes\id) \circ \psi$, obtaining thus a ucp map $\sA\to B(\Hil)$, which then extends to a ucp map $\tilde{\phi}:\sB\to B(\Hil)$, by injectivity of $B(\Hil)$. Then for any $b \in \sB$ define the map $\tilde{\psi}(b):\ell^1(\G)\to B(\Hil)$ by
\[\left(\tilde{\psi}(b)\right) (\om): = \tilde{\phi}(b*\omega) ;\quad \om\in\ell^1(\G) .\]
The map $\tilde{\psi}(b)$ is obviously completely bounded, hence identifies with an element in $\linf(\G)\overline{\otimes} B(\Hil)$ (see e.g. \cite[Corollary 7.1.5]{EffrosRuan}). Thus, we obtain the map $\tilde\psi:\sB \to \ell^\infty(\G)\overline{\otimes} B(\Hil)$; it is straightforward to see $\tilde\psi$ is a ucp $\G$-equivariant map that extends $\psi$.
\end{proof}

\begin{cor}\label{cor:inj-embd}
Every $\G$-$\cst$-algebra $\sA$ embeds equivariantly into a $\G$-injective $\cst$-algebra.
\end{cor}

\begin{proof}
Let $\alpha:\G\act \sA$ be an action of $\G$ on $\sA$, and let $\sA$ be faithfully represented on a Hilbert space $\Hil$. Then  $\alpha: \sA\to \M(\c0(\G)\otimes \sA)\subset \linf(\G)\overline{\otimes} B(\Hil)$  is a ucp $\G$-equivariant map, when we equip $\linf(\G)\overline{\otimes} B(\Hil)$ with the $\G$-action $\Delta\otimes\id$.
\end{proof}

We are ready to define the central object of this paper.

\begin{thm}\label{thm:universalF}
Every discrete quantum group $\G$ admits a unique (up to $\G$-isomorphism) universal $\G$-boundary $\C(\fb)$, in the sense that for any $\G$-boundary $\sA$ there is a completely isometric ucp $\G$-equivariant map $\sA\to \C(\fb)$.

We call $\C(\fb)$ the (algebra of continuous functions on the) \emph{Furstenberg boundary} of $\G$.
\end{thm}

\begin{proof}
Let $\C(\fb)$ be the the minimal image of a ucp $\G$-equivariant idempotent on $\linf(\G)$, which is unique up to $\G$-isomorphism by Proposition \ref{prop:Hamana-bnd}. It follows from Proposition \ref{prop:rigidity&essentiality} that $\C(\fb)$ is indeed a $\G$-boundary.

Now, to show the universality property, suppose $\sA$ is a $\G$-boundary, and let $\p_\nu: \sA\to \linf(\G)$ be the Poisson transform associated to a state $\nu$ on $\sA$. Composing $\p_\nu$ with a ucp $\G$-equivariant idempotent $\linf(\G)\to \C(\fb)$ yields a ucp $\G$-equivariant map from $\sA$ into $\C(\fb)\subset \linf(\G)$, which is completely isometric by Proposition \ref{prop:bnd-act-eqvs}.
\end{proof}

\begin{cor}\label{cor:bnd->C*-incl}
Any $\G$-equivariant ucp map between two $\G$-boundaries $\sA$ and $\sB$ is an injective *-homomorphism, and any such map is unique.
\end{cor}

\begin{proof}
Suppose $\sA$ and $\sB$ are $\G$-boundaries and $\phi$ and $\psi$ are $\G$-equivariant ucp maps from $\sA$ to $\sB$. By Theorem \ref{thm:universalF} there are completely isometric $\G$-equivariant ucp maps $\iota_\sA:\sA\to \C(\fb)$ and $\iota_\sB:\sB\to \C(\fb)$. By $\G$-injectivity of $\C(\fb)$ we may extend $\iota_\sB\circ\phi\circ\iota_\sA^{-1}: \iota_\sA(\sA)\to\iota_\sB(\sB)$ and $\iota_\sB\circ\psi\circ\iota_\sA^{-1}: \iota_\sA(\sA)\to\iota_\sB(\sB)$ to $\G$-equivariant ucp maps on $\C(\fb)$. By $\G$-rigidity both extensions are the identity map, hence $\phi=\psi$ is completely isometric. Thus, it only remains to show $\phi$ is *-homomorphism, and for this it is enough to show that $\iota_\sA(\sA)$ is a $\C^*$-subalgebra of $\C(\fb)$. By Proposition \ref{prop:injectivity-general} we can extend $\alpha\circ\iota_\sA^{-1}: \iota_\sA(\sA)\to\alpha(\sA)$ to a $\G$-equivariant ucp map from $\Phi:\C(\fb)\to\linf(\G)\overline{\otimes} B(\Hil)$, where $\Hil$ is a Hilbert space on which $\sA$ acts faithfully. The map $\Phi$ is completely isometric by $\G$-essentiality, and hence a $\cst$-isomorphism between $\C(\fb)$ and $\Phi(\C(\fb))$, where the latter gets its $\cst$-algebra structure from $\linf(\G)\overline{\otimes} B(\Hil)$ via the Choi-Effros product. Since $\alpha(\sA)$ is a subalgebra of $\linf(\G)\overline{\otimes} B(\Hil)$, it is closed under the Choi-Effros product, and hence a subalgebra of $\Phi(\C(\fb))$. This implies $\iota_\sA(\sA)$ is a $\cst$-subalgebra of $\C(\fb)$.
\end{proof}

Note that in particular by the above any $\G$-boundary $\sA$ is identified with a $\G$-$\C^*$-subalgebra of Furstenberg boundary $\C(\fb)$. We will use this fact later.

\begin{prop}
A discrete quantum group $\G$ is amenable if and only if its Furstenberg boundary $\C(\fb)$ is trivial.
\end{prop}

\begin{proof}
Let $\G$ be an amenable discrete quantum group. Let $m\in \linf(\G)$ be an invariant mean and $\nu$ a state on $\C(\fb)$. Then $\eta = m\circ \p_\nu$ is easily seen to be an invariant state on $\C(\fb)$. Then invariance of $\eta$ implies that the range of the Poisson transform $\p_\eta$ is the scalars. Hence $\C(\fb)=\CC$.

Conversely, if $\C(\fb)=\CC$, then the ucp $\G$-equivariant idempotent $\phi: \linf(\G) \to \C(\fb)$ is obviously an invariant mean, and hence $\G$ is amenable.
\end{proof}

We finish this section with providing a relatively explicit criterion allowing us to check that a given action of $\G$ is indeed a $\G$-boundary action. It will be used in Section \ref{orthog}. 

\begin{thm}\label{thm:unq-stn-->bnd}
	Let $\G$ be a discrete quantum group and let $\mu\in\ell^1(\G)$ be a state. Suppose $\sA$ is  a unital $\G$-$\cst$-algebra that admits a unique $\mu$-stationary state $\nu$, and that the Poisson transform $\p_\nu$ is completely isometric. 
	Then $\sA$ is a $\G$-boundary.
\end{thm}
\begin{proof}
	Recall that the Poisson boundary $\mathcal{H}_\mu:=\{x \in \ell^\infty(\G):(\id \ot \mu)\circ \Com(x) = x\}$ is  the image of a ucp $\G$-equivariant projection $\Phi$ acting on $\linf(\G)$. 
	The fact that $\nu$ is $\mu$-stationary implies that $\p_\nu$ takes values in $\mathcal{H}_\mu$. Suppose that we have any  ucp $\G$-equivariant map $\Psi$ from $\sA$ to $\mathcal{H}_\mu \subset \ell^\infty(\G)$. By Proposition  \ref{Pois-map-props} it must be a Poisson transform associated to a state $\omega\in S(\sA)$; and then Definition \ref{def:stationary} implies that $\omega$ is $\mu$-stationary. Hence $\Psi=\p_\nu$ is completely isometric.
	
	Consider then any state $\omega \in S(\sA)$ and the map $\p_\omega:\sA \to \ell^\infty(\G)$. Then $\Phi \circ \p_\omega$ is a ucp $\G$-equivariant map from $\sA$ to $\mathcal{H}_\mu$; hence by the above it is completely isometric. Thus $\p_\omega$ is also completely isometric, and $\sA$ is a $\G$-boundary.
\end{proof}

\section{The unique trace property for unimodular discrete quantum groups and its generalization}\label{sec:unique trace}
In this section we consider the problem of identifying discrete quantum groups $\G$ with the unique trace property, and more generally these whose reduced group $\C^*$-algebras do not admit $\ad$-invariant states. Similarly to the classical case, this problem is related to the faithfulness of the action $\G\act\partial_F\G$. 

By a result of Furman \cite{Furman}, the kernel of the action $G\act\partial_F G$ of a locally compact group $G$ on its Furstenberg boundary is the amenable radical of $G$. The following theorem is a version of this result in the quantum setting.

\begin{thm}\label{amen_cor}
The co-kernel $\sN_{\rm F}$ of the action $\G\act \C(\partial_F\G)$ is the unique minimal relatively amenable Baaj-Vaes subalgebra of $\linf(\G)$. It is contained in every other relatively amenable Baaj-Vaes subalgebra of $ \linf(\G)$.
\end{thm}

\begin{proof}
By Proposition \ref{prop:coker->Baaj-Vaes} the co-kernel $\sN_{\rm F}$ is a Baaj-Vaes subalgebra of $\linf(\G)$. Note that, in fact, $\sN_{\rm F}$ is the von Neumann algebra generated by all different copies of $\C(\partial_F\G)$ in $\linf(\G)$. By definition, there is an idempotent ucp $\G$-equivariant map from $\linf(\G)$ onto any copy of $\C(\partial_F\G)$. Thus, $\sN_{\rm F}$ is relatively amenable.
Now, let $\sM\subset\linf(\GG)$ be a relatively amenable Baaj-Vaes subalgebra, so that there is a $\G$-equivariant ucp map $\Psi:\linf(\G)\to\sM$. Let $\nu$ be a state on $\C(\partial_F\G)$. Then $\Psi\circ\cP_\nu : \C(\partial_F\G)\to\sM\subset \linf(\G)$ is a $\G$-equivariant ucp map. 
Thus, by Proposition \ref{Pois-map-props} there is a state $\nu'$ on $\C(\partial_F\G)$ such that $\Psi\circ\cP_\nu = \cP_{\nu'}$ and in particular  $\cP_{\nu'}(\C(\partial_F\G))\subset \sM$. Since  $\Delta(\sM)\subset\sM\wot\sM$, by Equation  \eqref{eq:Poi-map} we have
\[
\cP_{\mu*\nu'}(\C(\partial_F\G)) = (\id\otimes\mu)\Delta(\cP_{\nu'}(\C(\partial_F\G))) \subset \sM 
\]
for all $\mu\in\ell^1(\G)$. Since $\cP_{\nu'}$ is isometric, Lemma \ref{prox-act_c*} implies that the set \[\left\{\mu*\nu' :  \mu\in \rm{P}(\G)\right\}\] is weak* dense in the state space of $\C(\partial_F\G)$. Thus, it follows from Proposition \ref{Pois-map-props} that $\{P_{\nu''}(a): a\in \C(\partial_F\G)\} = \cP_{\nu''}(\C(\partial_F\G))\subset \sM$ for all $\nu''\in \C(\partial_F\G)^*$. Hence $\sN_{\rm F}\subset\sM$. 
\end{proof}

Recall that the adjoint action of $\G$ on $\C(\hh\G)$ is given by the formula
$\beta(x) = \ww^*(1 \ot x)\ww$ for all $x \in \C(\hh\G)$.  When we talk about
$\C(\hh\G)$ as a $\G$-space, this is the action we have in mind and we say for
instance that a functional $\mu\in \C(\hh\G)^*$ is $\G$-invariant if
$(\id\otimes \mu)(\beta(x)) = \mu(x)\I$ for all $x\in \C(\hh\G)$. Let us also
recall that when $\ww$ is viewed as an operator on
$\ell^2(\G)\otimes \ell^2(\G)$ then $\Delta(y) = \ww^*(\I\otimes y)\ww$ for all
$y\in\ell^\infty(\G)$.

\begin{lemma}\label{lem:ad-invariance}
  Let $\G$ be a discrete quantum group and $\mu \in \C(\hh\G)^*$ a state. Then
  $\mu$ is $\G$-invariant {\bf iff} it is a KMS-state for the scaling
  automorphism group $(\tau_t)_{t\in \br}$ (at the inverse temperature $1$). In
  particular when $\G$ is unimodular, $\mu$ is $\G$-invariant {\bf iff} it is a
  trace. On the other hand if $\G$ is not unimodular, then the Haar state of
  $\G$ is not $\G$-invariant.
\end{lemma}
\begin{proof}
  In this proof the symbols adorned with hats refer to $\G$ and these without
  hats to $\hh \G$. Thus we use the antipode $\hat S$ of $\ell^\infty(\GGamma)$,
  which is defined in particular on {\em finitely supported} elements.  More
  precisely, if $p$ is a finite-rank central projection in
  $\ell^\infty(\GGamma)$ then $\hat S_p : x \to \hat S(px)$ is a well defined,
  bounded linear map on $\ell^\infty(\GGamma)$, and $\bar p := \hat S(p)$ is
  another finite-rank central projection in $\ell^\infty(\GGamma)$. We recall
  moreover that $(\hat S_p\otimes\id)(\ww) = (\bar p\otimes 1)\ww^*$
  (\cite[Proposition~8.3]{KV}). Since $\hh S^2 = \hh \tau_{-i}$ on finitely
  supported elements (\cite[Proposition~5.22]{KV}) and
  $(\hat \tau_t\otimes\id)(\ww) = (\id\otimes\tau_{-t})(\ww)$ for all
  $t \in \br$ (see before Proposition~8.25 in \cite{KV}) we have also
  $(\hat S_p\otimes\id)(\ww^*) = (\bar p\otimes \tau_i)(\ww)$.

  Now let $\mu \in \C(\hat\GGamma)^*$ be a $(\tau_t)_{t\in \br}$-KMS state. By
  the KMS-property and antimultiplicativity of $\hat S$ we can write, for any
  $x \in \ell^\infty(\G)$ and $p$ as above,
    \begin{align}\label{eq_ad_inv}
      (\hat S_p\otimes\mu)(\ww^*(1\otimes x) \ww) &= 
      (\id\otimes\mu)((\hat S_p\otimes\tau_i)(\ww)(\hat S_p\otimes\id)(\ww^*)(1\otimes x)) \\
      \nonumber &= (\id\otimes\mu)((\bar p\otimes\tau_i)(\ww^*)
                  (\bar p\otimes\tau_i)(\ww) (1\otimes x))  = \mu(x) \bar p.
    \end{align}
    Applying $\hat S^{-1}$ yields $(p\otimes\mu)\beta(x) = \mu(x)p$, hence the
    $\mathrm{ad}$-invariance of $\mu$.

    On the other hand, assume that $\mu$ is $\mathrm{ad}$-invariant. Then the
    computation above shows that \eqref{eq_ad_inv} holds for all suitable $x$
    and $p$ (note that it holds as $\mathrm{ad}$-invariance means that the first
    expression in \eqref{eq_ad_inv} equals to the last one, and the second and
    third equalities always hold). Fix $\varphi\in\ell^1(\GGamma)$, a
    finite-rank central projection $p\in\ell^\infty(\GGamma)$ and consider
    $y = (\varphi \circ \hat S_p\otimes\id)(\ww) \in \Pol(\hh \G)$. Denote
    $(\id\otimes S)\Delta(y) = \sum y_{(1)}\otimes y'_{(2)}$ --- the sum is
    naturally finite. Since $(\id\otimes\Delta)(\ww) =\ww_{13} \ww_{12}$ and
    $(\id\otimes S)(\ww) = \ww^*$ (we are now using the antipode of
    $\C(\hat\GGamma)$) we have
\begin{align*}
  \sum y_{(1)}\otimes y'_{(2)} &= (\varphi \circ \hat S_p\otimes\id\otimes\id)(\ww_{13}^* \ww_{12}) \\
  &= (\varphi\otimes\id\otimes\id)((\hat S_p\otimes\id)(\ww)_{12} (\hat S_p\otimes\id)(\ww^*)_{13}).
\end{align*}
Hence, applying $\varphi$ to the first equality in  \eqref{eq_ad_inv} yields that 
$\sum \mu(y_{(2)}'xy_{(1)}) = \sum \mu(\tau_i(y_{(1)}) y'_{(2)} x)$. Now take
arbitrary elements $w,z \in \Pol(\hh \G)$ . We know that $w\otimes z$ can be written as a finite sum of
elements of the form $\sum y'_{(2)}x\otimes y_{(1)}$, where $y$ is associated to some $\varphi$ and $p$ as above (cf.\
e.g.\ \cite[Theorem~4.9]{Woronowicz_Matrix}), and by bilinearity the above
identity yields $\mu(wz) = \mu(\tau_i(z)w)$. Hence $\mu$ is $(\tau_t)_{t\in \br}$-KMS. In the unimodular case $(\tau_t)_{t\in \br}$ is trivial, so that KMS states correspond to traces.

For the Haar state $h$ we compute as above, using the corresponding modular group $(\sigma_t)_{t\in \br}$ (again $p$ is central finite rank projection and $x \in \ell^\infty(\G)$:
\begin{align*}
  (\hat S_p\otimes h)(\ww^*(1\otimes x)\ww) &= 
  (\id\otimes h)((\hat S_p\otimes\sigma_i)(\ww)(\hat S_p\otimes\id)
   (\ww^*)(1\otimes x)) \\
  &= (\id\otimes h)((\bar p\otimes\sigma_i)( \ww^*)
  (\bar p\otimes\tau_i)( \ww) (1\otimes x)).
\end{align*}
This is equal to $h(x)\bar p$ for any $x\in \ell^\infty(\G)$ {\bf iff}
$(\bar p\otimes\sigma_i)(\ww^*) (\bar p\otimes\tau_i)(\ww) = \bar p$ {\bf iff}
$(\bar p\otimes\tau_i)(\ww) = (\bar p\otimes\sigma_i)(\ww)$, which happens for
all central finite rank projections $p$ {\bf iff} $\tau_i = \sigma_i$, i.e.\ when we are in the Kac case.
\end{proof}

\begin{thm}\label{thm:faith-Fur-bnd-->unq-trc} 
Let $\G$ be a discrete quantum group and assume that the action $\G\act \sA$ of $\G$ on some $\G$-boundary $\sA$  is faithful. Then if $\G$ is unimodular, it has the unique trace property, i.e.\ $\C(\hh\G)$ admits a unique tracial state. If $\G$ is not unimodular, then it does not admit any $\G$-invariant functional, nor any KMS-state for the scaling automorphism group at the inverse temperature $1$.
\end{thm}
\begin{proof}

Note first that if $\sA$ is a $\G$-boundary and the action of $\G$ on $\sA$ is faithful, so is the action of $\G$ on its Furstenberg boundary $\C(\fb)$, as we have $\sN_{\alpha_{\sA}} \subset \sN_{\alpha_{F}}$, where $\alpha_{\sA}$ denotes the action of $\G$ on $\sA$ and $\alpha_{F}$ the action of $\G$ on $\C(\fb)$. Indeed, it suffices to use Theorem \ref{thm:universalF} and note that if $\Psi:\sA \to \C(\fb)$ is an isometric ucp $\G$-invariant map, then for any functional $\nu \in \sA^*$ and $a\in \sA$ we have
\[ (\id \ot \nu)(\alpha_{\sA}(a)) = (\id \ot \rho)(\alpha_{F}(\psi(a))) \in \sN_{\alpha_F},\]  	
where $\rho$ is any functional on $\C(\fb)$ extending the functional $\nu \circ \Psi^{-1} \in \Psi(\sA)^*$. Thus we may and do assume that $\sA=\C(\fb)$.

We want to show that if $\tau\in \C(\hh\G)^*$ is a $\G$-invariant state on $\C(\hh\G)$, then $(\id\otimes\tau)(\ww)\in\linf(\G)$ is the support of the counit $\varepsilon:\linf(\G)\to\mathbb{C}$. 

In view of Lemma \ref{lemcrossed}, we can extend $\tau$ to a ucp $\G$-equivariant map $\tilde\tau:\G\ltimes_r \C(\fb)\to \C(\fb)$, using $\G$-injectivity of $\C(\fb)$.
By $\G$-rigidity of $\C(\fb)$ (see Proposition \ref{prop:rigidity&essentiality}) the restriction of $\tilde{\tau}$ to $\alpha(\C(\fb))$ equals $\alpha^{-1}$.  
Thus, $\alpha(\C(\fb))$ lies in the multiplicative domain of $\tilde{\tau}$, and therefore we have $$\tilde\tau\big(\alpha(a)(y\ot 1)\big) = a\tau(y)$$ for all $a\in \C(\fb)$ and $y\in \C(\hh\G)$. 

Consider the action $\beta$ from \eqref{eq1}. Applying $(\id\otimes\tilde\tau)$ to the equation $(1 \ot x)(\ww \ot 1) = (\ww \ot 1) \beta(x)$, with $x=\alpha(a),a\in  \C(\fb)$ and remembering that $\beta$ restricted to $\alpha(\C(\fb))$ coincides with $\id \ot \alpha$, we get
\[((\id\otimes\tau)(\ww) )\otimes a = (((\id\otimes\tau)(\ww) )\otimes\I)\alpha(a).\]
In particular, setting $x = (\id\otimes\tau)(\ww)$ we have $\mu(a) x = x (P_\mu(a))$
for all $\mu\in \C(\fb)^*$ and $a\in \C(\fb)$. Since $\mu(a) = \varepsilon(P_\mu(a))$ this yields $\varepsilon(y)x = xy$ for all $y\in \sN_\alpha$. By faithfulness of $\alpha$ we can take for $y$ the support $p_0$ of $\varepsilon$ and we see that $x = \varepsilon(x)p_0$. Since $(\varepsilon\otimes\id)(\ww)=\I$ we conclude that $x=p_0$; thus $\tau$ is the Haar state.

The rest of the statements follows then from Lemma \ref{lem:ad-invariance}; note that in the unimodular case the Haar state is tracial, and in the non-unimodular case the Haar state is not $\G$-invariant.
\end{proof}

Finally the next proposition gives a connection to the classical result of Furman mentioned in the beginning of the section. By combining Theorem \ref{thm:faith-Fur-bnd-->unq-trc} and the next proposition we see that faithfulness of the action on the Furstenberg boundary implies, in the unimodular case, triviality of the amenable radical. This also results from Theorem \ref{amen_cor}, but we do not know whether the converse holds since in principle there might be relatively amenable Baaj-Vaes subalgebras of $\ell^\infty(\GGamma)$ which do not arise from normal amenable subgroups.

\begin{prop}\label{thm:unq-trc-->triv-amen-rad}
Let $\G$ be a unimodular discrete quantum group with the unique trace property, and let $\QG = \hh\G$.
If $\QH\subset \QG$ is such that $\QH\backslash\QG$ is co-amenable then $\QH = \QG$; in particular the amenable radical of $\G$ is trivial. 
\end{prop}

\begin{proof}
Suppose that $\QH$ is a subgroup of $\QG$ realized by $\pi\in\Mor(\C^u(\GG),\C^u(\HH))$ such that $\QH\backslash\QG$ is co-amenable. Then by Theorem \ref{thm_fun} the map $\pi$ admit a reduced version $\tilde\pi\in\Mor(\C(\GG),\C(\HH))$. If $\QH\subset\QG$ is a proper subgroup then $\tilde\pi$ is not injective and therefore $h_\QH\circ\tilde\pi$ is a trace on $\C(\QG)$ which is not faithful. This proves the first statement. To prove the second one, suppose $\LLambda\subset \G$ is a non-trivial normal amenable subgroup. Then $\QH =\hh{\LLambda/\G}$ is a subgroup of $\QG$ such that $\QH\backslash\QG$ is co-amenable. Thus, we have $\QH = \QG$, which implies $\LLambda = \{e\}$.
\end{proof}

\section{Further applications: $\C^*$-simplicity and crossed products}

In this  section we discuss some applications of the concepts discussed earlier to simplicity and existence of nuclear embeddings. These will find concrete interpretations in the last section.

Recall first that a discrete quantum group is said to be \emph{$\cst$-simple} if the reduced $\cst$-algebra $\C(\hh\G)$ of the dual compact quantum group $\hh\G$ is simple.

\begin{prop}\label{thm:C*-simple-->triv-amen-rad}
	Let $\G$ be a $\cst$-simple discrete quantum group, and let $\QG = \hh\G$. If $\QH\subset \QG$ is a closed quantum subgroup such that $\QH\backslash\QG$ is co-amenable then $\QH = \QG$. In particular, $\G$ has trivial amenable radical.
\end{prop}

\begin{proof}
	Let $\QH$ be a quantum subgroup of $\QG$ such that $\QH\backslash\QG$ is co-amenable. Then by Theorem \ref{thm_fun} there is a canonical *-homomorphism $\tilde\pi: \C(\QG)\to \C(\QH)$. Since $\G$ is $\cst$-simple, i.e.\ $\C(\QG)$ is simple, $\tilde\pi$ is injective, hence $\QH = \QG$. This proves the first statement. The second follows exactly as in Corollary \ref{thm:unq-trc-->triv-amen-rad}. 
\end{proof}


\begin{lem}\label{lem:aut-equiv}
Let $\Hil$ be a Hilbert space. Consider the action $(\Delta\otimes\id)$ of $\G$ on $B(\ell^2(\G) \ot \Hil)$, given by the formula
\[ (\Delta\ot \id)(x) = \ww^*_{12}x_{23}\ww_{12}, \;\;\; x \in B(\ell^2(\G) \ot \Hil).\] 
Suppose $\psi$ is a ucp map on $B(\ell^2(\G) \ot \Hil)$ whose restriction to $\C(\hh\G)\otimes \I$ is the identity map. Then $\psi$ is $\G$-equivariant.
\end{lem}

\begin{proof}
The ucp map $\id\otimes\psi$ on $B(\ell^2(\G) \ot\ell^2(\G) \ot \Hil)$ restricts to the identity map on $\c0(\G) \ot \C(\hh\G)\otimes \I$, and therefore contains the latter in its multiplicative domain. Thus, it follows that $\id\otimes\psi$ also restricts to identity on $\M(\c0(\G)\otimes \C(\hh\G)\otimes \I)$. In particular, $\ww\ot 1$ is in the multiplicative domain of $\id \ot \psi$. Hence, for every $x \in B(\ell^2(\G) \ot \Hil)$ we get
\[\begin{split}
(\id\otimes\psi)((\Delta\otimes\id)(x)) &= (\id\otimes\psi)(\ww^*_{12}x_{23}\ww_{12}) = \ww^*_{12}((\id\otimes\psi)(x_{23}))\ww_{12} \\&= \ww^*_{12}\psi(x)_{23}\ww_{12} =
(\Delta\otimes\id)(\psi(x)) .
\end{split}\]
\end{proof}

Before we formulate the next theorem, we recall that $\G$ acts canonically on $\ell^\infty(\G)$ (by the coproduct, viewed as the right action), and also on $B(\ell^2(\G))$; the latter right action, denoted by $\Delta^r$, is given by the same implementation, namely
\[ \Com^r(x) = V (x \ot 1)V^*, \;\;\; x \in B(\ell^2(\G)),\]
with $V$ denoting again the right multiplicative unitary.
We note here that there exists a (neccesarily unique) $\G$-equivariant faithful normal conditional expectation $\mathbb{E}:B(\ell^2(\G)) \to \ell^\infty(\G)$. It is given by the formula
\[ \mathbb{E}(x)  = (\omega_\Omega \ot \id)\circ \Delta^r(x), \;\;\; x \in B(\ell^2(\G)),  \]
where $\Omega$ is the unit vector in $\ell^2(\G)$ corresponding to the counit of $\ell^\infty(\G)$. The fact that $\mathbb{E}$ is a $\G$-equivariant ucp idempotent with image $\ell^\infty(\G)$ is easy to check. In order to justify the  faithfulness of $\mathbb{E}$ we invoke the crossed product construction $\G\ltimes \ell^\infty(\G) = [(\Linf(\hh\G)\otimes\I)\Delta(\ell^\infty(\G))]^{''}$ and the canonical faithful conditional expectation $\mathbb{E}_{\G\ltimes \ell^\infty(\G)}:\G\ltimes \ell^\infty(\G)\to \Delta(\ell^\infty(\G))$ as described in Subsection \ref{subsect.-crossedprod}. The latter is uniquely characterised by two conditions:
\begin{align*}
    \mathbb{E}_{\G\ltimes \ell^\infty(\G)}(x\otimes\I) &= h(x)\\
    \mathbb{E}_{\G\ltimes \ell^\infty(\G)}(\Delta(y))& = \Delta(y)
\end{align*}
for all $x\in\Linf(\hh\G)$ and $y\in\ell^\infty(\G)$, where $h:\Linf(\hh\G)\to\CC$ is the Haar state on the compact quantum group $\hh\GG$. Noting that 
\[(x\otimes\I)\Delta(y) = (x\otimes\I)V(y\otimes\I)V^* = V(xy\otimes\I)V^* \] we get the standard identification of $\G\ltimes \ell^\infty(\G)$ with $[\Linf(\hh\G)\ell^\infty(\G)]^{''} = B(\ell^2(\G))$. 
It is easy to check that 
$\mathbb{E}(x) =h(x)\I$  and   $\mathbb{E}(y)=y$ for all $x\in\Linf(\hh\G)$ and $y\in\ell^\infty(\G)$; thus   under the identification 
 $\G\ltimes \ell^\infty(\G)  \cong B(\ell^2(\G))$ both conditional expectations coincide. In particular $\mathbb{E}$ is faithful. 

We also recall that an inclusion of operator systems $\sV\subset \sW$ is said to be essential if a ucp map $\phi$ from $\sW$ into any operator system is completely isometric when its restriction to $\sV$ is completely isometric. We will use the facts that $\sV\subset \sW$ is essential iff $\sW\subset I(\sV)$, where $I(\sV)$ denotes the injective envelope of $V$, and that any essential operator system inclusion of $\cst$-algebras is a $\cst$-inclusion (see \cite{Ham79b} for these facts and more details).

In the proof of the following theorem we will make use of the notion of Fubini tensor product $\sV\otimes_\cF \sW$ of two operator systems $\sV$ and $\sW$. For relevant definitions and properties we refer the reader to \cite{Tomiyama} or to \cite[Chapter 7]{EffrosRuan}.

\begin{thm}\label{embed:crossed}
Let $\G$ be a discrete quantum group, and let $\sA$ be a $\G$-$\cst$-algebra. Then $\sA$ is a $\G$-boundary if and only if 
	\begin{equation}\label{C*-incl<->bnd}\GGamma\ltimes_r \sA \subset I(\C(\hh\G)).\end{equation}
\end{thm}

\begin{proof}
Suppose $\sA$ is a $\G$-$\cst$-algebra such that $\GGamma\ltimes_r \sA \subset I(\C(\hh\G))$. Let $\phi: \sA\to \linf(\G)$ be a ucp $\G$-equivariant map. Let $\sK$ be a Hilbert space on which $\sA$ acts. Extend $\phi$ to a ucp map from $B(\Kil)$ to $\linf(\G)$, still denoted by $\phi$. The ucp map $\id\ot\phi: B(\ell^2(\G))\ot B(\sK) \to B(\ell^2(\G))\ot \linf(\G)$ restricts to identity on $\C(\hh\G)\ot\I$, hence restricts to a completely isometric map $\G\ltimes_r \sA\to B(\ell^2(\G))\ot \linf(\G)$ by essentiality. Denote by $\alpha$ the action of $\G$ on $\sA$. Since $\phi$ is $\G$-equivariant, for all $T\in M_n(\mathbb{C})\ot\sA$ and $n\in\mathbb{N}$ we get
\[\|(\id_{M_n(\mathbb{C})}\ot\phi)(T)\| = \|(\id\ot\Delta)(\id_{M_n(\mathbb{C})}\ot\phi)(T)\| = \|(\id_{M_n(\mathbb{C})}\ot\id\ot\phi)(\id_{M_n(\mathbb{C})}\ot\alpha)(T))\| = \|T\| ,\]
which shows that $\phi$ is completely isometric. Hence $\sA$ is a $\G$-boundary.

	Conversely, suppose $\sA$ is a $\G$-boundary. Let $\iota: \sA\to \C(\fb)$ be a $\G$-equivariant completely ucp map. Then $\id\ot\iota: \GGamma\ltimes_r \sA \to \GGamma\ltimes_r \C(\fb)$ is an operator system inclusion. Hence, it is enough to show that $\GGamma\ltimes_r \C(\fb) \subset I(\C(\hh\G))$.
	
Since $\C(\fb)$ is $\G$-injective there is a ucp $\G$-equivariant map $\psi:\ell^\infty(\G)\otimes_\cF \C(\fb)\to\C(\fb)$ such that $\psi\circ\Delta$ is the identity map on $\C(\fb)$. Consider the map $\Phi = (\id\otimes \psi)\circ(\Delta^r\otimes\id)$ on $B(\ell^2(\G))\otimes_\cF \C(\fb)$, where as above $\Delta^r$ denotes the canonical right action of $\G$ on $B(\ell^2(\G))$. We have
	\[\begin{split}
	\Phi\circ\Phi &= (\id\otimes \psi)\circ(\Delta^r\otimes\id)\circ(\id\otimes \psi)\circ(\Delta^r\otimes\id)
	\\&=
	(\id\otimes \psi)\circ(\id\otimes\id\otimes \psi)\circ(\Delta^r\otimes\id\otimes\id)\circ(\Delta^r\otimes\id)
	\\&=
	(\id\otimes \psi)\circ(\id\otimes\id\otimes \psi)\circ(\id\otimes\Delta\otimes\id)\circ(\Delta^r\otimes\id)
	\\&=
	(\id\otimes \psi)\circ(\id\otimes\Delta)\circ(\id\otimes \psi)\circ(\Delta^r\otimes\id)
	\\&=
	(\id\otimes \psi\circ\Delta)\circ(\id\otimes \psi)\circ(\Delta^r\otimes\id)
	\\&=
	(\id\otimes \psi)\circ(\Delta^r\otimes\id)
	\\&=
	\Phi ,
	\end{split}\] 
	which shows that $\Phi$ is a projection. The $\C^*$-algebra $\C(\fb)$ is injective, as $\ell^\infty(\G)$ is a direct product of matrix algebras, and $\C(\fb)$ is the image of a ucp projection on $\ell^\infty(\G)$. Thus $B(\ell^2(\G))\otimes_\cF \C(\fb)$ is also injective, for example by comments in \cite[Section 4]{Paulsen}. Hence finally ${\rm im}(\Phi)$ is injective. Furthermore, for $a\in \C(\fb)$ and $\hat x\in \C(\hh\G)$ we have
	\[\begin{split}
	\Phi((\hat x\otimes 1)\Delta(a)) &= (\id\otimes \psi)\big((\Delta^r(\hat x)\otimes 1)(\Delta\otimes\id) \Delta(a)\big)
	\\&=
	(\hat x\otimes 1)(\id\otimes \psi\circ\Delta)(\Delta(a))
	\\&=
	(\hat x\otimes 1)\Delta(a) ,
	\end{split}\]
	which shows that $\GGamma\ltimes_r \C(\fb) \subset {\rm im}(\Phi)$.
	
So, in particular, $\C(\hh\G)\otimes\I \subset {\rm im}(\Phi)$ and by the afore-mentioned injectivity of ${\rm im}(\Phi)$ we have an inclusion $I(\C(\hh\G)\otimes\I) \subseteq {\rm im}(\Phi)$. Let $\phi: {\rm im}(\Phi) \to I(\C(\hh\G)\otimes\I)$ be a projection. Since $\phi$ restricted to $\C(\hh\G)\otimes\I$ is the identity map, if we show that the restriction of $\phi$ to $\Delta(\C(\fb))$ is also the identity map then it finishes the proof (by multiplicative domain arguments, standard properties of the crossed product, and the obvious identification $I(\C(\hh\G)\otimes\I) = I(\C(\hh\G))$).

Extend $\phi$ to a ucp map $\phi: B(\ell^2(\G)\ot \Hil) \to I(\C(\hh\G)\otimes\I)\subset B(\ell^2(\G)\ot \Hil)$, where $\Hil$ is a Hilbert space on which the $\C^*$-algebra $\C(\fb)$ acts faithfully. Consider the action $(\Delta\otimes\id)$ of $\G$ on $B(\ell^2(\G) \ot \Hil)$. Then $\phi$ is $\G$-equivariant by Lemma \ref{lem:aut-equiv}. 

Denote by $\mathbb{E}:B(\ell^2(\G))\to \ell^\infty(\G)$ the canonical conditional expectation, which is faithful and $\G$-equivariant by the remarks before the theorem.
	Then for any $T\in B(\ell^2(\G))\otimes_\cF \C(\fb)$ we have
	\[\begin{split}
	\Delta\big((\varepsilon\otimes\id)(\mathbb{E}\otimes \id)\Phi(T)\big)
	&= 
	\Delta\big((\varepsilon\otimes\id)(\mathbb{E}\otimes \id)(\id\otimes \psi)(\Delta^r\otimes\id)(T)\big)
	\\&=
	\Delta\big((\varepsilon\otimes\id)(\id\otimes \psi)(\mathbb{E}\otimes \id\otimes \id)(\Delta^r\otimes\id)(T)\big)
	\\&=
	\Delta\big((\varepsilon\otimes\id)(\id\otimes \psi)(\Delta\otimes\id)(\mathbb{E}\otimes \id)(T)\big)
	\\&=
	\Delta\big((\varepsilon\otimes\id)(\Delta\circ\psi)(\mathbb{E}\otimes \id)(T)\big)
	\\&=
	(\Delta\circ\psi)(\mathbb{E}\otimes \id)(T)
	\\&=
	(\mathbb{E}\otimes \id)\Phi(T) ,
	\end{split}\]
which shows that the conditional expectation $(\mathbb{E}\otimes \id)$ maps ${\rm im}(\Phi)$ onto $\Delta(\C(\fb))$.
	
Thus, the composition of $(\mathbb{E}\ot \id)$ with the restriction of $\phi$ to $\Delta(\C(\fb))$ yields the ucp $\G$-equivariant map 
$(\mathbb{E}\ot \id)\circ \phi|_{\Delta(\C(\fb))} : \Delta(\C(\fb))\to \Delta(\C(\fb))$, which by Proposition \ref{prop:rigidity&essentiality} must be the identity map. Since the conditional expectation $(\mathbb{E}\otimes \id)$ is faithful, it follows that $\phi$ is identity on $\Delta(\C(\fb))$.
\end{proof}

We now have all we need to generalize the main result of \cite{KalKen}, which led to applications of the Furstenberg boundary action in $\C^*$-simplicity problems. 

The argument below follows that of \cite[Theorem 6.2]{KalKen}, with several necessary modifications.
\begin{thm} \label{thm:c-star-simplicity}
Let $\G$ be a discrete quantum group. Then the following are equivalent:
\begin{enumerate}
\item
$\C(\hat \G)$ is simple; 
\item
$\G\ltimes_r \sA$ is simple for every $\G$-boundary $\sA$;
\item
$\G\ltimes_r \sA$ is simple for some $\G$-boundary $\sA$;
\item
$\G\ltimes_r \C(\fb)$ is simple.
\end{enumerate}
\end{thm}

\begin{proof}
The implications $(1)\Rightarrow (2)$ and $(3)\Rightarrow (4)$ follow from Theorem \ref{embed:crossed}, Corollary \ref{cor:bnd->C*-incl}, and the general fact that if we are given a $\cst$-inclusion $\sB \subset \sC  \subset I(\sB)$ and $\sB$ is a simple $\cst$-algebra, then so is $\sC$. Indeed suppose that $\sJ\subset \sC$ is an ideal. Since $\sB$ is simple the map $\sB\to \sC/\sJ$ is injective. Therefore there is a ucp map $\gamma:\sC/\sJ \to I(\sB)$ that extends the canonical embedding $\sB\to I(\sB)$. Consider the induced map $\tilde\gamma:\sC\to I(\sB)$. Since $\sC$ embeds into $I(\sB)$ there is an extension of $\tilde\gamma$ to $\tilde{\tilde{\gamma}}:I(\sB)\to I(\sB)$. Noting that the restriction of  $\tilde{\tilde{\gamma}} $ to $\sB$ is completely isometric we see that the map  $\tilde{\tilde{\gamma}}$ must be isometric (we use here the enveloping property of $I(\sB)$). In particular $\tilde\gamma$ has trivial kernel and hence $\sJ = \{0\}$.

The implication $(2)\Rightarrow(3)$ is trivial, so it remains to prove $(4)\Rightarrow (1)$.
Suppose that $\C(\hat \G)$ is not simple, let $\sJ$ be a non-trivial ideal of $\C(\hat \G)$, and let $\pi : \C(\hat \G) \to \C(\hat \G)/{\sJ}$ denote the corresponding quotient map. Recall that the action $\beta$ of $\G$ on $\C(\hat \G)$ is implemented by (the left) multiplicative unitary $\ww\in \M(\c0(\G) \ot \C(\hat \G))$, i.e.
	\[ \beta(x) = \ww^*(1 \ot x) \ww, \;\;\; x \in \C(\hat \G).\]
It is then easy to see that  $\beta$ induces an action $\tilde{\beta}$ on the quotient $\C(\hat \G)/{\sJ}$:
	\[
	\tilde{\beta}(\pi(x)) := (\id\otimes \pi)\beta(x) \quad x\in \C(\hat \G) .
	\]
	Observe that $\pi$ is $\G$-equivariant by definition.

By Corollary \ref{cor:inj-embd} the $\G$-$\C^*$-algebra $\C(\hat \G)/{\sJ}$ embeds $\G$-equivariantly into a $\G$-injective $\cst$-algebra, which we will denote by $\mathsf{B}$. 

Let $\tilde{\pi} : \G\ltimes_r \C(\fb) \to \mathsf{B}$ be a $\G$-equivariant ucp extension of $\pi$. Note that the restriction of $\tilde{\pi}$ to $\C(\fb)$ (or rather strictly speaking to its copy in the crossed product, which we will identify with the original object), is isometric by the rigidity established in Proposition \ref{prop:bnd-act-eqvs}. The last fact implies that there is a $\G$-equivariant projection $\phi : \mathsf{B}\to \tilde{\pi}(\C(\fb))$. Equipped with the corresponding Choi-Effros product, $\tilde{\pi}(\C(\fb))$ is a $\G$-injective $\cst$-algebra isomorphic to $\C(\fb)$. Moreover, the restriction of $\phi$ to the $\cst$-algebra $\C^*(\tilde{\pi}(\C(\fb)))$ is easily seen to be a 
$\G$-equivariant surjective $^*$--homomorphism; denote by $\sI$ the kernel of this $^*$--homomorphism. Then $\sI$ is a $\G$-invariant ideal in $\C^*(\tilde{\pi}(\C(\fb)))$, and define $\tilde{\sI} = \overline{\sI \,\tilde\pi(\C(\hat \G))}$. We will often use below the fact that $\G\ltimes_r \C(\fb)$ is spanned (as a normed space) by the products of elements in $\C(\hat \G)$ and in $\C(\fb)$, and by the multiplicative domain argument we have also that  
$\tilde{\pi}(\G\ltimes_r \C(\fb))$ is spanned by the products of elements in $\C(\hat \G)/\sJ$ and in $\tilde{\pi}(\C(\fb))$.

Observing  that $\sI$ is contained in the multiplicative domain of $\phi$ we conclude that for $x \in \tilde{\sI}$ we have $\phi(x)=0$. In particular since $\phi|_{\C(\fb)} = \id_{\C(\fb)}$ we see that  $\tilde{\sI}\,\cap\, \tilde\pi(\C(\fb)) = \{0\}$, thus  $\tilde{\sI}$ is strictly contained in $\C^*(\tilde{\pi}(\G\ltimes_r \C(\fb))$. 

Let us show that
$\tilde{\sI}$ is an ideal in $\C^*(\tilde{\pi}(\G\ltimes_r \C(\fb)))$; for that it suffices to show that   $\overline{\sI \,\tilde\pi(\C(\hat \G))} = \overline{\tilde\pi(\C(\hat \G))\,  \sI}$. 
As $\tilde{\pi}$ is a ucp $\G$-equivariant map, and $\G$ acts on the copy of $\C(\fb)$ in the crossed product by the formula 
\[ \alpha(y) = \ww^*_{12} (1 \ot y) \ww_{12},\;\;\;\; y \in \C(\fb) \subset \G \ltimes \C(\fb) \subset \M(\mathcal{K}(\ell^2(\G)) \ot \C(\fb)), \]
the action of $\G$ on $\tilde{\pi}(\C(\fb))$ -- hence also on  $\C^*(\tilde{\pi} (\C(\fb)))$  --  is  implemented
 by the unitary operator $U=(\id \ot \tilde{\pi})(\ww)\in \M(\c0(\G) \ot \C(\hat \G)/\sJ) \subset \M(\c0(\G) \ot \mathsf{B})$. Further we have the equality $\tilde{\pi}(\C(\hat \G)) = \overline{\{(\omega \ot \id)(U): \omega \in \ell^1(\G)\}}$. For any $z \in \sI$, $\omega \in \ell^1(\G)$ and $b \in \c0(\G)$ we thus note that
\begin{align*} z\, [(b\omega \ot \id)(U)] &= (b \omega \ot \id)(UU^* (1 \ot z)U) = 
(\omega \ot \id) (U U^*(1\ot z)U (b \ot 1)) \\&=
(\omega \ot \id)(U (\alpha(z)(b \ot 1))) \in \overline{\tilde{\pi}(\C(\hat \G))\, \sI},
\end{align*}
where in the last equality we use the $\G$-invariance of $\sI$ and the Podle\'s condition for the restricted action. This now implies that $\overline{\sI \,\tilde\pi(\C(\hat \G))} = \overline{\tilde\pi(\C(\hat \G))\,  \sI}$.

Finally consider the quotient map $q: \C^*(\tilde{\pi}(\G\ltimes_r \C(\fb))) \to \C^*(\tilde{\pi}(\G\ltimes_r \C(\fb)))/\,\tilde{\sI}$ and the composition 
$q\circ \tilde{\pi} : \G\ltimes_r \C(\fb) \to \C^*(\tilde{\pi}(\G\ltimes_r \C(\fb)))/\,\tilde{\sI}$. The latter is a priori just a ucp map, but we claim it is in fact multiplicative. Indeed, note that since $\tilde{\pi}$ is multiplicative on $\C(\hat\G)$, so is $q\circ \tilde{\pi}$.
Also, it follows from the construction that $\sI$ (as the kernel of $\phi$) contains all the elements of the form $\tilde{\pi}(xy) - \tilde{\pi}(x)\tilde{\pi}(y)$, $x, y \in \C(\fb)$.
Hence $q\circ \tilde{\pi}$ is also multiplicative on $\C(\fb)$, so that 
 another multiplicative domain argument yields the desired fact. 

Now $q\circ \tilde{\pi}$ vanishes on $\sJ$, which was assumed to be non-trivial and on the other hand is a non-zero map, as the target algebra is non-zero. This implies the crossed product $\G\ltimes_r \C(\fb)$ is not simple.
\end{proof}

\section{Boundary actions for free orthogonal quantum groups } \label{orthog}
In this section we give concrete examples of boundary actions for free orthogonal quantum groups.

\subsection{Free orthogonal quantum groups}\label{subsec_free_QG}

We denote by $\FO_Q$ Van Daele and Wang's free orthogonal discrete quantum group,
with $Q \in M_N(\CC)$ such that $Q\bar Q = \pm I_N$, $N\geq 2$. When $Q = I_N$,
we write also $\FO_N$ instead $\FO_Q$. The discrete quantum group in question is defined via the full Woronowicz \Cst
algebra $\sA_o(Q) = \cst_u(\FO_Q) = C^u(O^+_Q)$, so that we have $\hh \FO_Q = O_Q^+$. The corresponding reduced \Cst algebra is  $\cst_r(\FO_Q) = C(O^+_Q)$.

Following Banica's classification of the irreducible representations of $\hh\FO_Q$, we denote by
$\Hil_n$ ($n \in \mathbb{N}_0$) the spaces of the irreducible representations (up to equivalence), in
such a way that $\Hil_0 = \CC$, $\Hil_1 = \CC^N$ and
$\Hil_1\otimes \Hil_n \simeq \Hil_{n-1} \oplus \Hil_{n+1}$ equivariantly. Even more
precisely we consider $\Hil_n$ as a (uniquely defined) subspace of
$\Hil_1^{\otimes n}$ and denote by $P_n^+ \in B(\Hil_1^{\otimes n})$ the corresponding
orthogonal projection. We write $t\subset r\otimes s$ if there is an equivariant
embedding $\Hil_t \subset \Hil_r\otimes \Hil_s$, this happens exactly for $t = r+s-2a$
with $0\leq a\leq\min(r,s)$.

The dual algebras are given by $\sA = \c0(\FO_Q) = \bigoplus_{n\in\mathbb{N}} \sA_n$ with
$\sA_n = B(\Hil_n)$, and $M = \ell^\infty(\FO_Q) = \prod_{n\in\mathbb{N}} B(\Hil_n)$.
All coproducts are denoted $\Delta$. We denote by $p_n \in B(\Hil_n)\subset  \sM$ the
minimal central projections and write  $p_{\geq n} = \sum_{k\geq n} p_k \in \sM$.

\bigskip

Woronowicz' modular matrices are denoted $F_n \in B(\Hil_n)_+$ and normalized by
$\Tr(F_n) = \Tr(F_n^{-1})$; in particular $F_1 = {}^t(Q^*Q)$. We will use the
standard and normalized categorical ``traces'' given for $a\in B(\Hil_n)$ by
$\qTr_n(a) = \Tr(F_na)$ and $\qtr_n(a) = \qTr_n(a)/\qdim(n)$, where
$\qdim(n) = \Tr(F_n)$ is the quantum dimension of $\Hil_n$. These dimensions are
given by 
\begin{displaymath}
  \qdim(n) = [n+1]_q = \frac{q^{n+1}-q^{-n-1}}{q-q^{-1}}
\end{displaymath}
for a unique value $q \in \left]0,1\right]$. Here we assume that $q<1$, or
equivalently, $Q \notin U_2$. Then there are constants $C_1$, $C_2 > 0$
(depending on $q$) such that $C_1 q^{-n} \leq \qdim(n) \leq C_2 q^{-n}$ for all
$n$.

Note that the identity $q+q^{-1} = \qdim(1) = \Tr(F_1)$, with $q^{-1}\geq 1$,
resolves into $q^{-1} = f(\Tr(F_1))$ where $f(t) = \frac 12(t+\sqrt{t^2-4})$. On
the other hand it is known that the spectrum of $F_1$ is symmetric, so that
$\Tr(F_1) > \|F_1\| + \|F_1\|^{-1}$ when $N\geq 3$. Since $f$ is strictly
increasing, this yields $q^{-1} > \|F_1\|$. We will need later this fact under
the form $q \|F_1\| < 1$, which is satisfied if and only if $N\geq 3$.

Denote by $t_n \in \Hil_n\otimes \Hil_n$ a fixed vector such that $\|t_n\|^2 = \qdim(n)$.
Such a vector realizes the self-duality of $\Hil_n$ as well as the quantum trace
above, in the following sense: we have $\qTr_n(a) = t_n^*(a\otimes \id)t_n$.  According to the conjugate
equation we have $(\id_n\otimes\qTr_n)(t_nt_n^*) = \id_n$. On the other hand one
can compute $(\qTr_n\otimes\id_n)(t_nt_n^*) = F_n^{-2}$ (to get the identity here
one would need to use the right quantum trace). We shall moreover choose the maps $t_n$ in a coherent way, by putting $t_{n+1} = (P_{n+1}^+\otimes P_{n+1}^+)(\id\otimes t_n\otimes\id)t_1$. We have then $t_{m+n} = (P_{m+n}^+\otimes P_{m+n}^+)(\id\otimes t_n\otimes\id)t_m$ for any $n$, $m \in  \mathbb{N}_0$.

There are well-defined antilinear maps $j_n : \Hil_n \to \Hil_n$ such that $t_n  = \sum_i e_i\otimes j_n(e_i)$ for any  ONB $(e_i)_i$ of $\Hil_n$. For $\xi \in \Hil_n$ we denote also $\bar\xi = j_n(\xi) = (\xi^*\otimes\id)t_n$. Associated to $t_n$ is the invariant form $s_n= \sum_ie_i^*\ot(j_n^{-1}e_i)^*$ on $\Hil_n\otimes \Hil_n$. We have $F_n = j_n^*j_n$, $F_n^{-1} = j_nj_n^*$ and $j_n^2 = \pm\id_n$, where the sign is the same as in the standing assumption $Q\bar Q = \pm I_N$. One can moreover check that $(F_n^p\otimes F_n^p) t_n = t_n$ for all $p\in\RR$ and $j_{n\otimes n}t_n = \pm t_n$ where $j_{n\otimes n} = (j_n\otimes j_n)\sigma$. We will also use the identity
\begin{equation}\label{usefulident}
t_n = \pm \sigma (F_n^{-1}\otimes\id_n) t_n.
\end{equation}

 Given  $a \in B(\Hil_n)$ one can consider the unique element $\tilde a \in B(\Hil_n)$ satisfying $(a\otimes\id_n)t_n = (\id\otimes \tilde a)t_n$. Explicitly we have $\tilde a = j_na^*j_n^{-1}$ and $\tilde a = (s_n\otimes\id)(\id\otimes a\otimes\id)(\id\otimes t_n)$.

\bigskip

Recall that by our choice for the irreducible spaces $\Hil_n$, the space $\Hil_{r+s}$ is a subspace of $\Hil_r\otimes\Hil_s$.
From Banica's fusion rules we know that the orthogonal complement of $\Hil_{r+s}$ in $\Hil_r\otimes\Hil_s$ is isomorphic, as a corepresentation of $\FO_Q$, to $\Hil_{r-1}\otimes\Hil_{s-1}$. This isomorphism can be realized by the intertwiner $(P_r^+\otimes P_s^+)(\id\otimes t_1\otimes\id) : \Hil_{r-1}\otimes\Hil_{s-1} \to \Hil_r\otimes\Hil_s$, which is injective although not isometric, see e.g. \cite[Proposition~2.3]{Vergnioux_Cayley}.

A related intertwiner is 
the map $(P_r^+\otimes P_s^+)(\id\otimes t_a\otimes\id) P_t^+ : \Hil_t \to \Hil_r\otimes \Hil_s$ for $t = r+s-2a$, $0\leq a\leq \min(r,s)$. We denote $\kappa_t^{r,s}$ the inverse of its norm, so that $V_{r,s}^{t} = \kappa_t^{r,s} (P_r^+\otimes P_s^+)(\id\otimes t_a\otimes\id) P_t^+$ is an isometric intertwiner. The quantity $\kappa_t^{r,s}$ has been studied by many authors: here we follow the notation of \cite{FreslonVergnioux}, in \cite{BrannanCollins} it is denoted
\begin{displaymath}
\kappa_t^{r,s} = \|A_t^{r,s}\|^{-1} = 
\left(\frac{[t+1]_q}{\theta_q(t,r,s)}\right)^{1/2},
\end{displaymath}
and in \cite{Vergnioux_Decay} it appears in the proof of Lemma~4.8 in the form
\begin{displaymath}
\left(\kappa_t^{r,s}\right)^{-2} = \frac{\qdim(r)}{\qdim(r-m)} 
N_{r-a,s-a}^t \cdots N_{r-1,s-1}^t \text{, \quad with~}
N_{r,s}^t = 1 - \frac{\qdim(r-a)\qdim(s-a-1)}{\qdim(r+1)\qdim(s)}.
\end{displaymath}
Note that we have an evident lower bound
$\kappa_t^{r,s} \geq \|t_a\|^{-1} = \qdim(a)^{-1/2} \geq C_2^{-1/2}q^{a/2}$. In
fact one can show that there are constants $D_1$, $D_2 > 0$ (depending on $q$)
such that $D_1 q^{a/2} \leq \kappa_t^{r,s} \leq D_2 q^{a/2}$ for all
$t\subset r\otimes s$: cf \cite[Lemma~4.8]{Vergnioux_Decay} and
\cite[Prop.~3.1]{BrannanCollins}.

In this article we need one more result about these quantities.
\begin{lemma}\label{lem:kappa}
	Assume that $q<1$. For each $k \in \NN$ there exists a constant $E_k$
	(depending also on $q$) such that
	\begin{displaymath}
	\left|1-\left({\kappa_{t-k}^{r-k,s}}/{\kappa_t^{r,s}}\right)^2\right| \leq E_k \, q^{2(r-a)}
	\end{displaymath}
	for all $t = r+s-2a$, $0\leq a\leq\min(r-k,s)$.
\end{lemma}

\begin{proof}
	We start from the explicit expression that can be found in
	\cite{KauffmannLins,BrannanCollins}:
	\begin{displaymath}
	\left(\kappa_t^{r,s}\right)^{-2} = 
	[a]_q! \frac{[r-a]_q![s-a]_q![t+a+1]_q!}{[r]_q![s]_q![t+1]_q!},
	\end{displaymath}
	where $[n]_q! = [n]_q[n-1]_q \cdots [1]_q$. This yields:
	\begin{equation}\label{kappa_quotient}
	K :=
	\left(\frac{\kappa_{t-k}^{r-k,s}}{\kappa_t^{r,s}}\right)^2 =
	\frac{[r-k]_q![r-a]_q![t-k+1]_q![t+a+1]_q!}
	{[r]_q![r-k-a]_q![t+1]_q![t-k+a+1]_q!}.
	\end{equation}
	Now we have by definition $[n]_q = \delta q^{-n}(1+O(q^{2n}))$ where
	$\delta = (q^{-1}-q)^{-1}$. From this it follows
	\begin{displaymath}
	\frac{[n]_q!}{[n-k]_q!} = [n]_q[n-1]_q\cdots[n-k+1]_q
	= \delta^k q^{-nk} q^{k(k-1)/2}(1+O(q^{2n}))
	\end{displaymath}
	where the constant involved in the $O(\cdot)$ depends on $q$ and $k$. We apply
	this, with $n = r$, $r-a$, $t+1$, $t+a+1$ respectively, to the $4$ quotients
	appearing in~\eqref{kappa_quotient}. The factors $\delta^k$, $q^{k(k-1)/2}$
	simplify and we are left with
	\begin{displaymath}
	K = q^{kr}q^{-k(r-a)}q^{k(t+1)}q^{-k(t+a+1)} 
	\frac{(1+O(q^{2(r-a)}))(1+O(q^{2(t+a+1)}))}
	{(1+O(q^{2r}))(1+O(q^{2(t+1)}))}.
	\end{displaymath}
	The powers of $q$ also simplify and since $t\geq r-a$ we obtain
	$K = 1 + O(q^{2(r-a)})$ as claimed.
\end{proof}

\subsection{A $\GGamma$-boundary}

The Gromov boundary of $\FO_Q$ is constructed in \cite{VaesVergnioux} using the ucp maps
$\psi_{m,n} : \sA_m \to \sA_n$ given by
\begin{displaymath}
\psi_{m,n}(a) = V_{m,n-m}^{n*}(a\otimes\id_{n-m})V_{m,n-m}^{n}.
\end{displaymath}
Using the embedding of $\Hil_n$ in $\Hil_1^{\otimes n}$ one can also write
$\psi_{m,n}(a) = P_n^+(a\otimes P_{n-m}^+) P_n^+ = P_n^+(a\otimes \id) P_n^+$.
The algebra $\sB\subset \sM$ is then defined as the norm closure of the subspace
\begin{displaymath}
\sB_0 = \{ b = (b_n)_n \in \sM \mid \exists m ~ 
\forall n\geq m ~ b_n = \psi_{m,n}(b_m)\}.
\end{displaymath}
For $a\in \sA_m$, put $\psi_{m,\infty}(a) = (\psi_{m,n}(a))_{n\geq m} \in \sB_0$.
One can show that $\sB$ is a sub-$\C^*$-algebra of $\sM$ which contains $\sA$
\cite[Prop.~3.4]{VaesVergnioux} and one defines
$\sB_\infty = \sC(\partial \FO_Q) = \sB/\sA$ with the canonical projection
$\pi : \sB \to \sB_\infty$.   For $b = (b_m)_{m\in \mathbb{N}_0} \in \sM$, the norm of $\pi(\sB)$ in $\sM/\sA$
is $\limsup\limits_{m \to\infty} \|b_m\|$, moreover for $b = \psi_{m,\infty}(a)$ the sequence of
norms is decreasing, hence by density the norm of $\pi(b) \in \sB_\infty$ is
$\lim\limits_{m\to \infty} \|b_m\|$.

One also proves that $\Delta(\sB) \subset \M(\sA\otimes \sB)$
\cite[Prop.~3.6]{VaesVergnioux}, and since $\Delta(\sA)\subset 
\M(\sA\otimes \sA)$ the
restriction of $\Delta$ factors to an action of the discrete quantum group $\FO_Q$ on $\sB_\infty$, which we denote by 
$\beta : \sB_\infty \to \sM(\sA\otimes \sB_\infty)$. The algebra $\sB_\infty$ equipped with this action is called the algebra of continuous functions on the \emph{Gromov boundary of $\FO_Q$}. 
 There is a well-defined state
$\omega$ on $B$ such that $\omega(b) = \lim\limits_{m\rightarrow\infty} \qtr_m(b_m)$
\cite[Prop.~5.5]{VaesVergnioux}. It vanishes on $\sA$ and one denotes
$\omega_\infty$ its factorization to a state on $\sB_\infty$. Note that for
$b = \psi_{m,\infty}(a)$ one has $\omega(b) = \qtr_n \psi_{m,n}(a) = \qtr_m(a)$
for all $n\geq m$.

\bigskip

Since $(\qtr_1\otimes\qtr_n)\Delta$ is a convex combination of $\qtr_{n-1}$ and
$\qtr_{n+1}$ \cite[Thm.~5.6]{VaesVergnioux}, the state $\omega_\infty$ is {\em
	stationary} (with respect to the ``nearest neighborhood'' random walk),
meaning that $(\qtr_1\otimes\omega_\infty)\beta_\infty = \omega_\infty$. In fact
this property characterizes $\omega_\infty$.

\begin{thm}\label{thm:unique stn}
Let $N\geq 3$ and let $Q \in M_N(\CC)$ be such that $Q\bar Q = \pm I_N$.	Consider the Gromov boundary algebra $\sB_\infty$ of a discrete quantum group
	$\FO_Q$. If a state $\nu$ on $\sB_\infty$ is
$\qtr_1$-stationary, i.e.\  $\nu = (\qtr_1\otimes\nu)\beta_\infty$ then $\nu = \omega_\infty$.
\end{thm}

\begin{proof}
	By continuity and density, to prove $\nu = \omega_\infty$ it is enough to
	prove that $\nu_k := \nu\circ \pi\circ \psi_{k,\infty}$ on $B(\Hil_k)$ coincides
	with $\qtr_k$ for all $k$. Since $\nu_k$ and $\qtr_k$ are both states,
	$\nu_k\leq\qtr_k$ is in fact sufficient. Hence we fix $k \in \NN$,
	$a \in B(\Hil_k)$ such that $a\geq 0$ and $\|a\|\leq 1$, and we shall prove that
	$\nu_k(a) \leq \qtr_k(a)$.
	
	\bigskip
	
	Since $(\qtr_n\otimes\qtr_1)\Delta$ is a (non trivial) convex combination of
	$\qtr_{n-1}$ and $\qtr_{n+1}$, an easy induction shows that the stationarity
	assumption implies $(\qtr_n\otimes\nu)\beta_\infty = \nu$ for all $n\in \mathbb{N}_0$. Hence
	we have $\nu_k(a) = (\qtr_n\otimes \nu \pi)\Delta(\psi_{k,\infty}(a))$ for all
	$n$. Moreover, given $\epsilon > 0$, we know by the proof of
	\cite[Prop.~3.6]{VaesVergnioux} that for $r\geq k+n$ large enough
	\begin{displaymath}
	\| (\qtr_n\otimes p_{\geq r})\Delta (\psi_{k,\infty}(a)) - 
	(\qtr_n\otimes \psi_{r,\infty})(b)\| \leq \epsilon,
	\end{displaymath}
	where $b = \sum_{s\subset n\otimes r} V_{n,r}^s \psi_{k,s}(a)V_{n,r}^{s*}$. In
	particular this shows that $\nu_k(a) = \lim_{r\to\infty}\nu_k^{n,r}(a)$ with
	\begin{displaymath}
	\nu_k^{n,r}(a) = {\textstyle\sum_{m=0}^n} 
	(\qtr_n\otimes\nu_r)(V_{n,r}^{n+r-2m} \psi_{k,n+r-2m}(a) V_{n,r}^{n+r-2m*}).
	\end{displaymath}
	
	To compare this with $\qtr_k(a)$ we will decompose $\qtr_k(a)$ in a slightly
	unnatural way. For all $r\geq n\geq k$ we have
	\begin{displaymath}
	\qtr_k(a) = \qtr_n\psi_{k,n}(a) = \qtr_n(P_n^+(a\otimes 1)P_n^+) =
	(\qtr_n\otimes\nu_r)((P_n^+\otimes
	P_r^+)(a\otimes \id)(P_n^+\otimes P_r^+))
	\end{displaymath}
	since $\nu_r(P_r^+) = 1$.  Next we decompose
	$P_n^+\otimes P_r^+ = \sum_{m=0}^n V_{n,r}^{n+r-2m}V_{n,r}^{n+r-2m*}$, so that
	\begin{displaymath}
	\qtr_k(a) = {\textstyle\sum_{m=0}^n} {{\textstyle\sum_{m'=0}^n}}
	(\qtr_n\otimes\nu_r) (V_{n,r}^{n+r-2m}
	{\phi^{n,r}_{n+r-2m,n+r-2m'}}(a)V_{n,r}^{n+r-2{m'}*}),
	\end{displaymath}
	where
	$\phi^{n,r}_{s,s'}(a) = V_{n,r}^{s*}(a\otimes\id)V_{n,r}^{s'} =
	V_{n,r}^{s*}(a\otimes P_{n-k}^+\otimes P_r^+)V_{n,r}^{s'}$.
	Our strategy will be to perform a term-by-term comparison with
	$\psi_{k,s}(a) = P_s^+(a\otimes P_{s-k}^+)P_s^+$.
	
	We will take care simultaneously of the terms $s\neq s'$, but let us note
	already that $\phi^{n,r}_{s,s'}(a)$ vanishes if $|m-m'|>k$. Indeed, one can
	obtain $a \in B(\Hil_k)$ by slicing $t_kt_k^* \in B(\Hil_k)\otimes B(\Hil_k)$ on the
	left with some functional $\omega\in B(H_k)^*$, hence one can obtain $\phi^{n,r}_{s,s'}(a)$ by slicing
	$(\id\otimes V_{n,r}^{s*})(t_kt_k^*\otimes\id)(\id\otimes V_{n,r}^{s'}) \in
	B(\Hil_k)\otimes B(\Hil_{s'},\Hil_s)$
	on the left as well. But this map is an intertwiner from $\Hil_k\otimes \Hil_{s'}$ to
	$\Hil_k\otimes \Hil_s$, which have no irreducible component in common if
	$|s-s'| = 2|m-m'|>2k$.

	To compare $\phi_{s,s'}^{n,r}(a)$ and $\psi_{k,s}(a)$, for $s = n+r-2m$ we write
	$V_{n,r}^s = \kappa_{s}^{n,r} (P_n^+\otimes P_r^+)(\id\otimes t_m\otimes\id)
	P_s^+$ as explained previously. We use moreover the estimate
	$\|P_n^+ - (P_{n-m}^+\otimes P_m^+)(P_k^+\otimes P_{n-k}^+)\| \leq B
	q^{n-m-k}$ from \cite[Lemma~A.4]{VaesVergnioux} which yields
	\begin{align*}
	(\kappa_s^{n,r}{\kappa_{s'}^{n,r})^{-1}} 
	{\phi^{n,r}_{s,s'}}(a) &\simeq P_s^+(\id\otimes t_m^*\otimes\id)
	(P_{n-m}^+\otimes P_m^+\otimes P_r^+) (a\otimes P_{n-k}^+\otimes P_r^+) \\
	& \makebox[5cm]{} (P_{n-{m'}}^+\otimes P_{{m'}}^+\otimes P_r^+)
	(\id\otimes t_{{m'}}\otimes\id)P_{{s'}}^+,
	\end{align*}
	meaning more precisely that the difference of both terms has a norm dominated
	by {$B (q^{n-m-k} + q^{n-m'-k})\|t_m\|\|t_{m'}\|$}, hence by
	$BC_2 q^{n-k}(q^{-m} + q^{-m'})q^{-m/2}q^{-m'/2}$.
	
	In the second term the projection $(P_{n-m}^+\otimes P_m^+\otimes P_r^+)$ is absorbed 
	by other projections on the left and on the right, so that for $m$, $m' \leq n-k$ this term equals:
	\begin{align*}
	&P_s^+ (a\otimes[(\id\otimes t_m^*\otimes \id)(P_{n-k}^+\otimes P_r^+)
	(\id\otimes t_{m'}\otimes \id)]) P_{s'}^+ = \\
	&\qquad\qquad = 
	P_s^+ (a\otimes[P_{s-k}^+(\id\otimes t_m^*\otimes \id)(P_{n-k}^+\otimes P_r^+)
	(\id\otimes t_{m'}\otimes \id)P_{s'-k}^+]) P_{s'} ^+ \\
	&\qquad\qquad = \delta_{s,s'} (\kappa_{s-k}^{n-k,r})^{-2} 
	P_s^+ (a\otimes V_{n-k,r}^{s-k*}V_{n-k,r}^{s-k}) P_s^+  
	= \delta_{s,s'} (\kappa_{s-k}^{n-k,r})^{-2} P_s^+ (a\otimes P_{s-k}^+) P_s^+.
	\end{align*}
	Indeed
	$P_{s-k}^+(\id\otimes t_m^*\otimes \id)(P_{n-k}^+\otimes P_r^+) (\id\otimes
	t_{m'}\otimes \id)P_{s'-k}^+$
	is an intertwiner between irreducibles, hence a scalar. So it vanishes if
	$s\neq s'$, and when $s=s'$ one recognizes
	$(\kappa_{s-k}^{n-k,r})^{-2} V_{n-k,r}^{s-k*}V_{n-k,r}^{s-k} =
	(\kappa_{s-k}^{n-k,r})^{-2} P_{s-k}^+$.
	
	\bigskip
	
	Using the inequalities $\kappa_s^{n,r}\leq D_2 q^{m/2}$,
	$\kappa_{s'}^{n,r}\leq D_2 q^{m/2}$ and $\kappa_{s-k}^{n-k,r}\leq D_2 q^{m/2}$
	this yields, for $0\leq m$, $m'\leq n-k$ and $s = n+r-2m$ different from
	$s' = n+r-2m'$:
	\begin{align*}
	\|\phi^{n,r}_{s,s'}(a)\| &\leq 
	BC_2 \kappa_s^{n,r}\kappa_{s'}^{n,r}  q^{n-k}(q^{-m} + q^{-m'})q^{-m/2}q^{-m'/2} \\
	&\leq BC_2D_2^2 q^{n-k}(q^{-m} + q^{-m'}). 	\end{align*}
Further for $s$ as above 
\[	\|\psi_{k,s}(a) - (\kappa_{s-k}^{n-k,r}/\kappa_s^{n,r})^{2} \phi^{n,r}_{s,s}(a)\| 
	\leq 2BC_2 (\kappa_{s-k}^{n-k,r})^2 q^{n-2m-k} \\
	\leq 2BC_2 D_2^2 q^{n-m-k}.
	\]
	Finally for the terms $s=s'$ we can apply Lemma~\ref{lem:kappa} and we obtain,
	since $\|\phi^{n,r}_{s,s}(a)\| \leq \|a\|\leq 1$:
	\begin{displaymath}
	\|\psi_{k,s}(a) - \phi^{n,r}_{s,s}(a)\| \leq 2BC_2D_2^2 q^{n-m-k} + E_k q^{2(n-m)}
	\leq F_k q^{n-m}
	\end{displaymath}
	for some constant $F_k$. We can assume also that $F_k \geq BC_2D_2^2 q^{-k}$
	so as to control the term $s\neq s'$ at the same time. We will use these
	estimates later in the case $m$, $m' < n/2 + k$ --- assuming that $n$ is even
	and greater than $4k$.
	
	\bigskip
	
	For the terms corresponding to $m\geq n/2$, (and in particular for the terms corresponding to $m>n-k$), we will
	concentrate instead on the ``exterior part'' of the formulae for
	$\nu_k^{n,r}(a)$ and $\qtr_k(a)$, that is, on the linear forms
	$\mu_{n,r}^{s,s'} : x \mapsto (\qtr_n\otimes\nu_r)(V_{n,r}^s x
	V_{n,r}^{s'*})$,
	which in both cases are applied to elements $x \in B(\Hil_{s'}, \Hil_s)$ with
	$\|x\| \leq \|a\| \leq 1$.
	
	Without loss of generality we can assume $m' =: m+p \geq m$. We compute again
	$V_{n,r}^s$ using $t_m$ and we recall that
	$t_{m'} = (P_{m+p}^+ \otimes P_{m+p}^+)(\id\otimes t_m\otimes\id)t_p$. Since
	the projections $P_{m+p}^+$  are absorbed by $P_n^+$ resp. $P_r^+$ (e.g. $ (\id\otimes P_{m+p})P_n= P_n$, i.e. $\Hil_n$ is a
subrepresentation of $\Hil_{n-m-p}\otimes \Hil_{m+p}$) this allows
	to write
	\begin{align*}
	V_{n,r}^s x V_{n,r}^{s'*} &= \kappa_{n,r}^s \kappa_{n,r}^{s'} (P_n^+\otimes P_r^+)
	(\id \otimes t_m\otimes\id)P_s^ +x 
	P_{s'}^+ (\id\otimes t_{p}^*\otimes \id)(\id\otimes t_{m}^*\otimes \id) (P_n^+\otimes P_r^+) \\
	&= \kappa_{n,r}^s \kappa_{n,r}^{s'} (P_n^+\otimes P_r^+)
	(\id \otimes t_m\otimes\id)y(\id\otimes t_{m}^*\otimes \id) (P_n^+\otimes P_r^+)
	\end{align*}
	where
	$y := P_s^+ x P_{s'}^+(\id\otimes t_p^*\otimes\id) \in B(\Hil_{n-m}\otimes \Hil_{r-m})$
	satisfies $\|y\| \leq \|t_p\| \|x\| \leq C_2^{1/2}q^{-p/2}$.
	
	We see in particular that $\mu_{n,r}^{s,s'}(x)$ is a composition of a bounded map $x \mapsto y$ and a positive linear
	functional of $y$, whose norm can be computed by evaluating at $y=1$. Since
	$\qTr_n = (\qTr_{n-m}\otimes\qTr_m) (P_n^+ \cdot P_n^+) \leq
	(\qTr_{n-m}\otimes\qTr_m)$, where we view both sides of the inequality  as functionals acting on $B(\Hil_{n-m} \ot \Hil_m)$, we can write
	\begin{align*}
	& \kappa_{n,r}^s \kappa_{n,r}^{s'} (\qtr_n\otimes\nu_r)[(P_n^+\otimes P_r^+)
	(\id \otimes t_mt_{m}^*\otimes \id) (P_n^+\otimes P_r^+)] \leq \\
	& \hspace{2cm} \leq \frac{\kappa_{n,r}^s \kappa_{n,r}^{s'}}{\qdim(n)}
	(\qTr_{n-m}\otimes\qTr_m\otimes\tilde\nu_r) 
	(P_{n-m}^+ \otimes t_mt_m^*\otimes P_{r-m}^+),
	\end{align*}
	where $\tilde\nu_r = \nu_r(P_r^+\cdot P_r^+)$ on $B(\Hil_m\otimes\Hil_{r-m})$.
	Since $(\qTr_m\otimes\id)(t_mt_m^*) = {F_m^{-2}}$ and
	$\qTr_{n-m}(P_{n-m}^+) = \qdim(n-m)$, the last quantity can be simplified and
	bounded above as follows:
	\begin{align*}
	\kappa_{n,r}^s \kappa_{n,r}^{s'} \frac{\qdim(n-m)}{\qdim(n)} 
	\tilde\nu_r({F_m^{-2}}\otimes P_{r-m}^+) 
	&\leq \kappa_{n,r}^s \kappa_{n,r}^{s'}  \frac{\qdim(n-m)}{\qdim(n)} 
	{\|F_m\|^2}  \\
	&\leq D_2^2 q^{m/2} q^{m'/2} \frac{C_2 q^{-n+m}}{C_1 q^{-n}}{\|F_m^{-1}\|^2} 
	=: G~ q^{2m}q^{p/2}{\|F_m^{-1}\|^2}.
	\end{align*}
	Altogether we have obtained the estimate
	$\|\mu_{n,r}^{s,s'}\| \leq GC_2^{1/2}~ q^{2m} \|F_m\|^2$.  Since
	$\|F_m^{-1}\|^2 = \|F_m\|^2 = \|F_1\|^{2m}$ this can also be written, without any assumptions on
	$m$, $m'$, as \[\|\mu_{n,r}^{s,s'}\| \leq GC_2^{1/2}~ (q\|F_1\|)^{2\min(m,m')}.\]
	
	\bigskip
	
	We finally put all terms together as follows, for $r\geq n\geq 4k$, $n$
	even, dropping the terms $\phi^{n,r}_{s,s}(a)$ for $m\geq n/2$ which are
	positive:
	\begin{align*}
	\nu_k^{n,r}(a) - \qtr_k(a) &\leq 
	{\textstyle\sum_{m=0}^{n/2-1}} (\qtr_n\otimes\nu_r)(V_{n,r}^{s} [\psi_{k,s}(a)-\phi^{n,r}_{s,s}(a)] V_{n,r}^{s*}) \\
	& \qquad\qquad + {\textstyle\sum_{n/2}^n} (\qtr_n\otimes\nu_r)(V_{n,r}^{s} \psi_{k,s}(a) V_{n,r}^{s*}) \\
	& \qquad\qquad + {\textstyle \sum_{m\neq m'}} \big |(\qtr_n\otimes\nu_r)(V_{n,r}^{s} \phi^{n,r}_{s,s'}(a)V_{n,r}^{s'*})\big|.
	\end{align*}
	Using the fact that the terms $m\neq m'$ vanish if $|m-m'|>k$, we can split
	the last sum according to whether $m$, $m' \geq n/2$ or $m$, $m' < n/2+k$,
	up to adding some of the terms twice. Using the estimates obtained previously
	we arrive at
	\begin{align*}
	\nu_k^{n,r}(a) - \qtr_k(a) & \leq F_k {\textstyle\sum_{m=0}^{n/2-1}} q^{n-m} + 
	GC_2^{1/2} {\textstyle\sum_{n/2}^n} (q\|F_1\|)^{2m} + \\
	& \qquad\qquad + F_k q^n{\textstyle \sum_{m,m'< n/2+k}}(q^{-m}+q^{-m'}) \\
	& \qquad\qquad + GC_2^{1/2} {\textstyle \sum_{m,m'\geq n/2}} (q\|F_1\|)^{2\min(m,m')} \\
	& \leq F_k {\textstyle\frac n2} q^{n/2} + 
	G C_2^{1/2} ({\textstyle\frac n2}+1) {(q\|F_1\|)^n} + \\
	& \qquad\qquad + 2 F_k ({\textstyle\frac n2}+k)^2 q^{n/2-k} 
	+ GC_2^{1/2} ({\textstyle\frac n2}+1)^2 (q\|F_1\|)^{n},
	\end{align*}
	where we have used the fact that $q\|F_1\|<1$. Letting $r\to\infty$ we obtain
	the same upper bound for $\nu_k(a) - \qtr_k(a)$.  Then we let $n\to\infty$ and
	obtain $\nu_k(a) \leq \qtr_k(a)$.
\end{proof}


Note that the assumption $N\geq 3$ in the previous theorem is optimal, since for $N=2$ (and still $q<1$) the discrete quantum group $\FO_Q$ is amenable whereas the algebra $\sB_\infty$ is still infinite-dimensional. When $q=1$ the construction of $\sB_\infty$ in \cite{VaesVergnioux} does not apply.

As advertised before, the concept of boundary actions allows us to confirm Ozawa's conjecture for the exact $\cst$-algebras $\C_r^*(\FO_Q) = C(O^+_Q)$, and to obtain a $\cst$-simplicity result for the crossed product of the Gromov boundary actions.

\begin{cor}\label{KeyCorollary}
Let $N\geq 3$ and let $Q \in M_N(\CC)$ be such that $Q\bar Q = \pm I_N$. The action of the discrete quantum group $\FO_Q$ on its Gromov boundary $\sB_\infty$, introduced in \cite{VaesVergnioux}, is a boundary action. Moreover we have the embedding 
	\[\C(O^+_Q)\subset \FO_Q\ltimes_r \sB_\infty \subset I(\C(O^+_Q)),\]
and the $\C^*$-algebra in the middle is nuclear. 	

If in addition $\|Q\|^8 \leq \frac{3}{8} \textup{Tr} (QQ^*)$, then the crossed product $\FO_Q\ltimes_r \sB_\infty$ is simple.
\end{cor}

\begin{proof}
The last theorem, Theorem  \ref{thm:unq-stn-->bnd} and Theorem 5.6 of \cite{VaesVergnioux} imply that the Gromov boundary action is a boundary action in the sense of Definition \ref{def:boundary}. 

The displayed inclusion is then a consequence of Theorem  \ref{embed:crossed}. 
Corollary 6.2 of \cite{VaesVergnioux} implies that $\FO_Q$ is exact; Theorem 4.5 of that paper (amenability of the Gromov boundary action), together with the results in Section 3 of \cite{VaesVergnioux} 	(nuclearity of $\sB_\infty$) imply via Proposition 4.4 of that paper that the crossed product $\FO_Q\ltimes_r \sB_\infty$ is nuclear.

Finally the extra condition on $Q$ is shown in Theorem of \cite{VaesVergnioux} to guarantee $\C^*$-simplicity of $\FO_Q$, and the last statement follows now from Theorem \ref{embed:crossed}. 
\end{proof}

\subsection{Faithfulness}

Our next aim is to prove that the action $\beta_\infty$ of $\FO_Q$ on $\sB_\infty$ is faithful.  Let
$\sN_\infty = \{(\id\otimes\varphi)\beta_\infty(x) \mid x\in \sB_\infty, \varphi\in
\sB_\infty^*\}''\subset \ell^\infty(\FO_Q)$ be the cokernel of $\beta_\infty$.
 Using Proposition \ref{prop_p0} we see that showing that  $\beta_\infty$ is faithful is equivalent with  showing that $p_0\in\sN_\infty$. Before proving this we first fix some further  notation. Recall that for $a \in B(\Hil_n)$ we denote $\tilde a \in B(\Hil_n)$ the unique element satisfying
$(a\otimes\id_n)t_n = (\id\otimes \tilde a)t_n$.
 
\begin{defn}
  We consider $z_n = H_n(A_n) \in \sN_\infty$, where
  $A_n \in B(\Hil_n)\otimes B(\Hil_n)$ is such that
  $\sum \tilde A_{n(1)} \otimes A_{n(2)} = t_nt_n^*$, and
  $H_n : B(\Hil_n) \otimes B(\Hil_n) \to \sN_\infty$ is defined by
  \begin{displaymath}
    H_n (a\otimes b) = (\id\otimes\omega_\infty)
    (\beta_\infty(\pi\psi_{n,\infty}(a))
    (1\otimes \pi\psi_{n,\infty}(b))).
  \end{displaymath}
\end{defn}
Recalling that $t_n = \sum_i e_i\otimes j(e_i) = \sum j^*(e_i)\otimes e_i$ if $(e_i)_i$ is an ONB of $\Hil_n$ one can check that
\begin{equation}\label{expl_A}A_n = \sum e_j(F_n^{-1}e_i)^*\otimes e_ie_j^*.\end{equation} In particular for every  $g \in B(\Hil_n)$ we have 
\begin{align*}
   {\textstyle \sum } A_{n(1)} g A_{n(2)} &= {\textstyle\sum}_{i,j}e_ie_i^*\langle e_j|F_n^{-1} g e_j\rangle \\&= \qTr_n'(g){\textstyle\sum}_ie_ie_i^* = \I\qTr_n'(g)
\end{align*}  where $\qTr'_n(a) = t_n^*(\id_n\otimes a)t_n = \Tr_n(F_n^{-1}a)$ is the ``right''
  quantum trace.
In particular we have the following property:
 \begin{equation} \label{lem_traces}
  {\textstyle\sum}  \qTr_n( f A_{n(1)} g A_{n(2)}) = 
    \qTr_n(f) \qTr_n'(g), \;\;\; f, g \in B(\Hil_n). 
\end{equation}

\bigskip

Our aim is now to show that $q^nz_n$ converges $\sigma$-weakly to $p_0$, thus establishing that $p_0 \in \sN_\infty$. This will require a few lemmas. Let us formulate first two remarks about our strategy.

\begin{remark}
   In the case of a classical free group $F_N$, the analogue of $q^nz_n$ is the function $f_n \in \ell^\infty(F_N)$ defined as follows. Let $\partial F_N$ be the Gromov boundary of $F_N$ identified with the space of infinite reduced words on the generators and their inverses. For $g\in F_N$, denote $X_g \subset \partial F_N$ the subspace of infinite reduced words starting by $g$, and $\chi_g \in C(\partial F_N)$ its characteristic function. The function $f_n$ is then given by
   \begin{equation*}
   f_n(g) = \sum_{|h|=n} (\chi_h\mid g\cdot\chi_h)
  \end{equation*}
  where we use the hermitian structure on $C(\partial F_N)$ given by the canonical ``uniform harmonic'' probability measure on $\partial F_N$. We have $f_n(e) = 1$ and one can check that $f_n(g) \to 0$ as $n\to\infty$ for fixed $g\neq e$. This is of course not the easiest way to establish faithfulness of the boundary action for $F_N$.
  \end{remark}

\begin{remark}
  Imagine that instead of working in $\sB_\infty$ we remain in $B(\Hil_n)$ and
  consider $G_n (a\otimes b) = (\id\otimes\qtr_n) (\Delta(a) (1\otimes b))$,
  $y_n = G_n(A_n) \in \ell^\infty(\FO_Q)$ with the same element $A_n$ as above. One can then compute
  directly for any $k$, using the equality~\eqref{lem_traces}:
  \begin{align*}
    p_k y_n &= {\textstyle\sum} (p_k\otimes\qtr_n)(V_{k,n}^nA_{n(1)}V_{k,n}^{n*}
    (1\otimes A_{n(2)})) \\
    &= {\textstyle\sum} (\id\otimes\qtr_n)(V_{k,n}^n)(\id\otimes\qTr'_n)(V_{k,n}^{n*}). 
  \end{align*}
 Moreover
  $(\id\otimes\qTr_n)(V_{k,n}^n) = (\id_k\otimes t_n^*)(V_{k,n}^n\otimes\id_n)t_n$
  vanishes for all $k>0$ because it is an invariant vector in $\Hil_k$. As a result $y_n = \qdim(n)p_0$ for all $n$. However the connection of $y_n$ to our element $z_n$ is not clear.
  \end{remark}

Our first lemma shows that $q^nz_n$ has the correct value ``at the unit''.
\begin{lemma}  \label{lem_eval_unit}
  We have $p_0 z_n = q^{-n} p_0$.
\end{lemma}
\begin{proof}
  Since $p_0$ is the support of the co-unit $\varepsilon$ of
  $\ell^\infty(\GGamma)$, and $(\varepsilon\otimes\id)\alpha = \id$ for any
  continuous action $\alpha$, we have
  $p_0 H_n(a\otimes b) = \omega(\psi_{n,\infty}(a)\psi_{n,\infty}(b))p_0 =
  \lim_{t\to\infty} \qtr_t(\psi_{n,t}(a)\psi_{n,t}(b))p_0$, where we are using the approximate multiplicativity of the maps $\psi_{n,t}$.

  Notice that we have $\qTr_t(f) = (\qTr_n\otimes \qTr_{t-n})(f)$ for
  $f \in B(\Hil_t)$. Applying the equality \eqref{lem_traces} we can write
  \begin{align*}
    {\ts\sum}\qTr_t(\psi_{n,t}(A_{n(1)})\psi_{n,t}(A_{n(2)})) 
    &= {\ts\sum} \qTr_{t-n}(\qTr_n\otimes\id_{t-n})
    [P_t^+(A_{n(1)}\otimes\id_{t-n})P_t^+(A_{n(2)}\otimes\id_{t-n})] \\
    &= \qTr_{t-n}[(\qTr_n\otimes\id)(P_t^+) (\qTr'_n\otimes\id)(P_t^+)].
  \end{align*}
  Now $(\qTr'_n\otimes\id)(P_t^+)$ is an intertwiner of $H_{t-n}$, hence a
  multiple of the $P_{t-n}^+$, and by applying $\qTr'_{t-n}$ one finds
  $(\qTr'_n\otimes\id)(P_t^+) = \qdim(t)/\qdim(t-n) P_{t-n}^+$. As a result we have
  \begin{align*}
    {\ts\sum}\qtr_t(\psi_{n,t}(A_{n(1)})\psi_{n,t}(A_{n(2)}))
    = \qdim(t-n)^{-1} (\qTr_n\otimes \qTr_{t-n})(P_t^+)
    = \qdim(t)/\qdim(t-n).
  \end{align*}
  Since $\qdim(t) \sim C q^{-t}$ when $t\to\infty$, with a
  constant $C>0$ depending only on $q$, the limit of the above ratio is $q^{-n}$
  as stated.
\end{proof}
\begin{lemma} \label{lem_cones} 
  Let $B\in B(\Hil_t)$ and let $e\in B(\Hil_n)$ be a rank one matrix. Then we have
  \begin{displaymath}
    |\qtr_t(B\psi_{n,t}(e))| \leq \frac{\qdim(t-n)}{\qdim(t)} \|B\| \|eF_n\|.
  \end{displaymath}
\end{lemma}
\begin{proof}
  Recall that $\qTr_t(\cdot) = (\qTr_n\otimes\qTr_{t-n})(P_t^+\cdot P_t^+)$. Since $P_t^+$ commutes with 
  $F_n\otimes F_{t-n}$ we have
  \begin{equation*}
    \qtr_t(B\psi_{n,t}(e)) = (\qdim t)^{-1} (\qTr_n\otimes\qTr_{t-n})
    (P_t^+B P_t^+(e\otimes \id_{t-n})).
  \end{equation*}
  The norm of the functional $\qTr_{t-n}$ is $\qdim(t-n)$. If $e$ has rank one,
  so does $eF_n$ and the norm of $\qTr_n(\,\cdot\,e) = \Tr_n(\,\cdot\,eF_n)$
  equals $\|eF_n\|.$
\end{proof}
\begin{lemma}  \label{lem_global_bound}
  There exists a constant $M>0$, depending only on $q$, such that
  $\|z_n\| \leq M q^{-n}$ for all $n$.
\end{lemma}
\begin{proof}
  Let $\eta \in \ell^\infty(\GGamma)_*$. Let us decompose
  $A_n = \sum A_{n,ij}\otimes e_ie_j^*$, where $(e_i)_i$ is an ONB of $\Hil_n$. We
  have by definition
  \begin{align*}
    \eta(z_n) &= {\ts\sum} (\eta\otimes\omega_\infty\pi)
    (\Delta(\psi_{n,\infty}(A_{n,ij}))(1\otimes\psi_{n,\infty}(e_ie_j^*))) \\
    &= {\ts\sum}\lim_{t\to\infty} (\eta\otimes\qtr_t)
    (\Delta(\psi_{n,\infty}(A_{n,ij}))(1\otimes\psi_{n,t}(e_ie_j^*))) \\
    &= {\ts\sum}\lim_{t\to\infty} \qtr_t(B_{ij}^{nt}\psi_{n,t}(e_ie_j^*)),
    \end{align*}
    where $B_{ij}^{nt} = (\eta\otimes p_t)\Delta(\psi_{n,\infty}(A_{n,ij}))$
    satisfies $\|B_{ij}^{nt}\| \leq \|\eta\| \|A_{n,ij}\|$. Applying
    Lemma~\ref{lem_cones} and taking the limit we obtain
    $|\eta(z_n)| \leq q^n \|\eta\| \sum \|A_{n,ij}\| \|F_n e_j\|$.
    
    Now according to the explicit formula for $A_n$ (c.f. Eq. \eqref{expl_A}) we have
    $\|A_{n,ij}\| = \|e_j\| \|F_n^{-1}e_i\|$. Taking for $(e_i)_i$ an
    ONB of eigenvectors for $F_n$ we obtain
    $|\eta(z_n)| \leq q^n \|\eta\| \Tr(F_n^{-1}) \Tr(F_n)$. Since
    $\Tr(F_n^{-1}) = \Tr(F_n) = \qdim(n) \leq \sqrt M q^{-n}$ for some constant $M>0$ this
    yields the result.
\end{proof}

Now we need to prove that $q^n p_k z_n \to 0$, for fixed $k$, as
$n\to\infty$ which  requires a  much more technical argument. We start by rewriting $p_k z_n$ in a convenient way.

\begin{lemma}\label{lem_decomp}
  Let $\zeta$, $\xi$ be vectors in $\Hil_k$. Then we have
  \begin{displaymath}
    (\zeta \mid p_k z_n \xi) = {\textstyle\sum_{l=0}^k} \lim_{t\to\infty} 
    (\kappa_s^{k,t})^2 \qTr_{t-n}(w_{n,t}^{k,l}(\zeta)^*w_{n,t}^{\prime\,k,l}(\xi))
  \end{displaymath}
  where $s = t+k-2l$ and $w_{n,t}^{k,l}(\xi)$, $w_{n,t}^{\prime\,k,l}(\xi) \in B(\Hil_{t-n},\Hil_{s-n})$, for
  $\xi = \sum \xi^{(1)}\otimes\xi^{(2)}\in \Hil_{k-l}\otimes\Hil_l$, are given by:
  \begin{align*}
    w_{n,t}^{k,l}(\xi) = \qdim(t)^{-1/2} {\ts\sum} (\qTr_n\otimes\id)
    (P_s^+(\xi^{(1)}\bar\xi^{(2)*}\otimes \id_{t-l})P_t^+) \\
    w_{n,t}^{\prime\,k,l}(\xi) = \qdim(t)^{-1/2} {\ts\sum} (\qTr'_n\otimes\id)
    (P_s^+(\xi^{(1)}\bar\xi^{(2)*}\otimes \id_{t-l})P_t^+).
  \end{align*}
\end{lemma}

\begin{proof}
  For a linear form $\eta\in B(\Hil_k)^* \subset \ell^\infty(\GGamma)_*$ we have,
  as in the proof of Lemma~\ref{lem_global_bound}:
    \begin{equation*}
    \eta(z_n) = {\ts\sum}\lim_{t\to\infty} (\eta\otimes\qtr_t)
    (\Delta(\psi_{n,\infty}(A_{n(1)}))(1\otimes\psi_{n,t}(A_{n(2)}))).
  \end{equation*}
  We then compute the coproduct as
  $(p_k\otimes p_t)\Delta(x) = \sum_{l=0}^k V_{k,t}^s x V_{k,t}^{s*}$, where
  $s = t+k-2l$:
  \begin{align*}
    \eta(z_n) &= \lim_{t\to\infty} (\qdim t)^{-1}{\ts\sum\sum_{l=0}^k} 
                (\eta\otimes\qTr_t)
    [V_{k,t}^s P_s^+(A_{n(1)}\otimes \id_{s-n})P_s^+ V_{k,t}^{s*} \\
    & \hspace{5cm}
    (\id_k\otimes P_t^+)(\id_k\otimes A_{n(2)} \otimes\id_{t-n})
      (\id_k\otimes P_t^+)].
  \end{align*}
  For $\eta = (\zeta \mid \,\cdot\, \xi)$ this yields
  \begin{align*}
    (\zeta \mid p_kz_n\xi ) &= \lim_{t\to\infty} (\qdim t)^{-1}
    {\ts\sum\sum_{l=0}^k} (\qTr_n \otimes \qTr_{t-n})
    [(\zeta^*\otimes P_t^+)V_{k,t}^s P_s^+ \\
    & \hspace{5cm}   (A_{n(1)}\otimes \id_{s-n})P_s^+ V_{k,t}^{s*} 
    (\xi\otimes P_t^+)(A_{n(2)} \otimes\id_{t-n})].
  \end{align*} 
  Now we apply again the equality \eqref{lem_traces}, with $f$ and $g$ corresponding to the
  left legs of
  $(\zeta^*\otimes P_t^+)V_{k,t}^s P_s^+ \in B(H_n)\otimes B(H_{s-n}, H_{t-n})$ and
  $P_s^+ V_{k,t}^{s*} (\xi\otimes P_t^+) \in B(H_n)\otimes B(H_{t-n}, H_{s-n})$
  respectively:
  \begin{align*}
    (\zeta \mid p_kz_n\xi ) &= \lim_{t\to\infty} (\qdim t)^{-1}
    {\ts\sum_{l=0}^k} \qTr_{t-n}
    [(\qTr_n\otimes\id)((\zeta^*\otimes P_t^+)V_{k,t}^s P_s^+)  \\
    & \hspace{6cm}  (\qTr'_n\otimes\id)(P_s^+ V_{k,t}^{s*} (\xi\otimes P_t^+))].
  \end{align*}
  This has the required form if we put
  \begin{align*}
    w_{n,t}^{k,l}(\zeta) &= \qdim(t)^{-1/2}(\kappa_s^{k,t})^{-1}
    (\qTr_n\otimes\id)(P_s^+ V_{k,t}^{s*} (\zeta\otimes P_t^+)) \qquad \text{and} \\
    w^{\prime\,k,l}_{n,t}(\xi) &= \qdim(t)^{-1/2}(\kappa_s^{k,t})^{-1}
    (\qTr'_n\otimes\id)(P_s^+ V_{k,t}^{s*} (\xi\otimes P_t^+)).
  \end{align*}
  We have moreover, by definition of $\kappa_s^{k,t}$:
  \begin{align*}
    w_{n,t}^{k,l}(\zeta) &= \qdim(t)^{-1/2}
      (\qTr_n\otimes\id)(P_s^+ (\id_{k-l}\otimes t_l^*\otimes \id_{t-l}) 
      (\zeta\otimes P_t^+)) \\
    &= \qdim(t)^{-1/2} (\qTr_n\otimes\id)
      (P_s^+ (\zeta^{(1)}\bar\zeta^{(2)*} \otimes \id_{t-l}) P_t^+),
  \end{align*}
  and similarly for $w^{\prime\,k,l}_{n,t}(\xi)$.
  Here $\zeta^{(1)}\bar\zeta^{(2)*} \in B(\Hil_l,\Hil_{k-l})$ arises from the
  decomposition $\zeta = \sum \zeta^{(1)}\otimes\zeta^{(2)} \in \Hil_{k-l}\otimes \Hil_l$.
\end{proof}
Note that the element $w_{n,t}^{k,l}(\xi) \in B(\Hil_{t-n},\Hil_{s-n})$ is defined
for any vector $\xi\in \Hil_{k-l}\otimes \Hil_l$. In the next lemma we use twisted
Hilbert-Schmidt norms given for $f \in B(\Hil_p, \Hil_q)$ by
$\|f\|_\HS^2 = \qTr_p(f^*f)$. 

\begin{lemma} \label{lem_easy} ~
Let $ \xi 	\in \Hil_{k-l}\otimes\Hil_l$.
  \begin{itemize}
  \item We have
    $\qdim(t)^{1/2}\|w_{n,t}^{k,l}(\xi)\|_\HS = \qdim(s)^{1/2}\|w_{n,s}^{k,k-l}((F_l^{1/2}\otimes F_{k-l}^{1/2})\bar\xi)\|_\HS$ and
    $\|\xi\| = \|(F_l^{1/2}\otimes F_{k-l}^{1/2})\bar\xi\|$. 
  \item We have $\|w_{n,t}^{k,l}(\xi)\|_\HS \leq C q^{-n/2}q^{-l/2} \|\xi\|$ for
    a constant $C>0$ depending only on $q$.
  \end{itemize}
  The same properties hold for $w_{n,t}^{\prime\,k,l}(\xi)$.
\end{lemma}
\begin{proof}
  We clearly have $\qdim(t)^{1/2} w_{n,t}^{k,l}(\xi)^* = \qdim(s)^{1/2} w_{n,s}^{k,k-l}(\bar\xi)$, since the conjugate of $\xi  = \sum \xi^{(1)}\otimes \xi^{(2)}\in \Hil_{k-l}\otimes \Hil_l$ is given by $\bar\xi = \sum \bar\xi^{(2)}\otimes\bar\xi^{(1)} \in \Hil_l\otimes \Hil_{k-l}$. 
  On the other hand we have $\|f^*\|_\HS = \|F_q^{1/2}fF_p^{-1/2}\|_\HS$. Finally, since $F$-matrices
  commute with intertwiners, $\qTr_n(F_n^{1/2} \,\cdot\, F_n^{-1/2}) = \qTr_n$,
  and $F_l^{-1/2} j_l = j_l F_l^{1/2}$, it follows that
  \begin{displaymath}
    F_{s-n}^{1/2} w_{n,t}^{k,l}(\xi) F_{t-n}^{-1/2} =
    w_{n,t}^{k,l}((F_{k-l}\otimes F_l)^{1/2}\xi).
  \end{displaymath}
  
  For the norm estimate, first note that
  $\sum \|\xi^{(1)}\| \|\bar\xi^{(2)}\| = \sum \|\xi^{(1)}\|
  \|F_l^{1/2}\xi^{(2)}\| \leq \|\xi\| \qdim(l)^{1/2}$,
  by taking for $(\xi^{(2)})$ an ONB of eigenvectors of $F_l$ and using
  Cauchy-Schwarz. Hence we have
  \begin{align*}
    \|w_{n,t}^{k,l}(\xi)\|_\HS &\leq \|\qTr_{t-n}\|^{1/2} \|w_{n,t}^{k,l}(\xi)\| 
    \leq (\qdim t)^{-1/2} \|\qTr_{t-n}\|^{1/2} \|\qTr_n\| ~ 
    {\ts\sum} \|\xi^{(1)}\|  \|\bar\xi^{(2)}\| \\
    &\leq (\qdim t)^{-1/2} \qdim(t-n)^{1/2} \qdim(n) 
    \qdim (l)^{1/2} \|\xi\|
    \leq C q^{-n/2}q^{-l/2} \|\xi\|.
  \end{align*}
\end{proof}

The norm estimate above is of the same order as the one of
Lemma~\ref{lem_global_bound} as $n\to\infty$, which is not enough for our
purposes. Using the next lemma we will obtain a better estimate in the proof of
Theorem~\ref{thm_faithful} --- in the Kac case even a much better one since we will deduce that $\|w_{n,t}^{k,l}(\xi)\|_\HS$ is bounded with respect to $n$, for fixed $k$. In the proof we use the ``higher weight'' Wenzl recursion relation \cite[Lemma~3.4]{FreslonVergnioux}: for
  $\zeta\in H_{p+q}$  decomposed as $\zeta = \sum
  \zeta^{(1)}\otimes\zeta^{(2)}\in H_{p}\otimes H_{q}$ and $\zeta = \sum
  \zeta_{(1)}\otimes\zeta_{(2)}\in H_{q}\otimes H_{p}$ we have
  \begin{displaymath}
    \sum (\overline{\zeta}^{(1)*}\otimes
    \id_{n-p})P_{n}^+(\zeta^{(2)}\otimes\id_{n-q}) = \alpha_{p, q}^{n}\sum
    P_{n-p}^+(\zeta_{(1)}\overline{\zeta}_{(2)}^{*}\otimes
    \id_{n-p-q})P_{n-q}^+,
  \end{displaymath}
  where the coefficient $\alpha_{p, q}^{n}$ satisfies $|\alpha_{p, q}^{n}| \leq 1$. This formula is stated in \cite{FreslonVergnioux} only in the Kac case --- since the methods used in that article to study MASAs only hold in the tracial case. However a rapid inspection shows that its statement and proof hold without modification in the general, non-Kac case.

\begin{lemma}[The Ice-Swimming Lemma]  \label{lem_technical} 
  For any $1\leq\l\leq k\leq n\leq t$ and $\xi \in \Hil_{k-l}\otimes \Hil_l$ there are
  vectors $\xi'\in \Hil_{k-l}\otimes \Hil_l$,  
  $\xi''\in \Hil_{k-l-1}\otimes \Hil_{l-1}$ with $\|\xi'\|$,
  $\|\xi''\| \leq \|\xi\|$ such that
  \begin{align} \label{eq_technical}
    \|w_{n,t}^{k,l}(\xi)\|_\HS &\leq  \|F_l\| \Big({\ts\frac{\qdim(t-l)}{\qdim(t)}}\Big)^{1/2} \|w_{n-l,t-l}^{k,l}(\xi')\|_\HS + \\
    & \nonumber \qquad \qquad +
    C \|F_l\|^2 \|w_{n-1,t-1}^{k-2,l-1}(\xi'')\|_\HS  + 
    C q^{-3l/2} \|F_l\|^2 q^{n/2} \|\xi\|,
  \end{align}
   where $C$ is a constant depending only on $q$. In the case $k=l$ we can delete the term involving $\xi''$. In the case $k=2l$, if
  $t_l^*(\xi) = 0$ then we have also $t_{l-1}^*(\xi'')=0$.
\end{lemma}
\begin{proof}
  In this proof $C>0$ is a ``generic constant'', depending only $q$, that we
  will modify only a finite number of times.
  
   We first use the estimate
  $P_t^+ = (\id_l\otimes P_{t-l}^+)(P_n^+\otimes \id_{t-n}) + D$ from
  \cite[Lemma~A.4]{VaesVergnioux}, with ``generic'' error term $D\in B(\Hil_l\otimes \Hil_{n-l}\otimes \Hil_{t-n})$ controlled by
  $\|D\|\leq C q^{n-l}$, and we permute through the trace according to
  $\qTr_l(\zeta\bar \xi^{(2)*}) = (F_l\bar \xi^{(2)})^*\zeta$:
  \begin{align} \nonumber
    w_{n,t}^{k,l}(\xi) &\simeq 
    (\qdim t)^{-1/2} {\ts\sum} (\qTr_l\otimes \qTr_{n-l}\otimes\id)
    [P_s^+(\xi^{(1)}\otimes \id_{t-l})
    ((\bar\xi^{(2)*}\otimes\id_{n-l})P_n^+)  \otimes\id_{t-n})] 
    \\ \label{eq_technical_transformed} &= (\qdim t)^{-1/2}
    {\ts\sum} (\qTr_{n-l}\otimes\id) [((F_l\bar\xi^{(2)})^*\otimes\id_{s-l})
    P_s^+(\xi^{(1)}\otimes \id_{t-l})].
  \end{align}  
   To control the error term $E \in B(\Hil_{t-n},\Hil_{s-n})$ let us recall that
  $\sum \|\xi^{(1)}\| \|\bar\xi^{(2)}\| \leq \|\xi\| \qdim(l)^{1/2}$, see the proof of Lemma~\ref{lem_easy}. We shall also use a similar formula 
  \begin{equation}\label{sim_form}
       \sum \|\xi^{(1)}\| \|j^*\xi^{(2)}\| \leq \|\xi\| \qdim(l)^{1/2}
  \end{equation}
  which is proved as the previous one by noting that 
  $\|F_l\| = \|F_l^{-1}\|$. In particular 
  \[  {\ts\sum} \|\xi^{(1)}\|  \|F_l\bar\xi^{(2)}\| =  {\ts\sum} \|\xi^{(1)}\|  \|j^*\xi^{(2)}\|\leq\|\xi\| \qdim(l)^{1/2}. \]
 Hence we have
  \begin{align*}
    \|E\|_\HS &\leq \|\qTr_{t-n}\|^{1/2} \|E\| \leq C q^{n-l} (\qdim t)^{-1/2}
    \|\qTr_{t-n}\|^{1/2} \|\qTr_n\| ~ 
    {\ts\sum} \|\xi^{(1)}\|  \|F_l\bar\xi^{(2)}\| \\
    &\leq C q^{n-l} (\qdim t)^{-1/2} \qdim(t-n)^{1/2} \qdim(n) 
    \qdim(l)^{1/2} \|\xi\|
    \leq C q^{-3l/2} q^{n/2} \|\xi\|.
  \end{align*}
  This will go into the third term on the right-hand side of~\eqref{eq_technical}.
  
  Now we use \cite[Lemma~3.4]{FreslonVergnioux}. We want to apply it to
  $\tilde\sigma(\xi) := \sum F_l^{-1}\xi^{(2)}\otimes \xi^{(1)}$ which is {\em not}
  in $\Hil_k$ (even if $\xi$ was), so we first decompose:
  \begin{displaymath}
    \tilde\sigma(\xi) = \xi' + 
    (P_l^+\otimes P_{k-l}^+)(\id_{l-1}\otimes t_1\otimes\id_{k-l-1}) (\xi''_0)
  \end{displaymath}
  with $\xi'  = \sum \xi'^{(1)}\otimes \xi'^{(2)}\in \Hil_k\subset \Hil_{l}\otimes \Hil_{k-l}$ given by
  $\xi' = P_k^+\tilde\sigma(\xi)$ --- in particular $\|\xi'\| \leq \|F_l\| \|\xi\|$ ---, and with
  $\xi''_0 =\sum \xi_0''^{(1)}\otimes \xi_0''^{(2)} \in \Hil_{l-1}\otimes
  \Hil_{k-l-1}$ which we map into the orthogonal complement of $\Hil_{k}$ in $\Hil_l\otimes \Hil_{k-l}$ as explained in subsection \ref{subsec_free_QG}. 
  We denote by $W'$, $W''$ the corresponding terms of the right-hand side of~\eqref{eq_technical_transformed}.
  
  The statement of \cite[Lemma~3.4]{FreslonVergnioux} involves a rearrangement
  of the tensor product decomposition: $\xi' = \sum\xi'_{(1)}\otimes \xi'_{(2)}$
  with $\xi'_{(1)} \in \Hil_{k-l}$, $\xi'_{(2)}\in \Hil_l$. We can then write 
  \begin{align*}
    W' &= (\qdim t)^{-1/2}
    {\ts\sum} (\qTr_{n-l}\otimes\id) [\bar\xi'^{(1)*}\otimes\id_{s-l})
    P_s^+(\xi'^{(2)}\otimes \id_{t-l})]
    \\ &= (\qdim t)^{-1/2}
    \alpha^s_{l,k-l} {\ts\sum} (\qTr_{n-l}\otimes\id)
    [P_{s-l}^+ (\xi'_{(1)}\bar\xi'^*_{(2)}\otimes\id_{t-2l}) P_{t-l}^+] \\
    &= \Big({\ts\frac{\qdim(t-l)}{\qdim(t)}}\Big)^{1/2} \alpha^s_{l,k-l}  
    w_{n-l,t-l}^{k,l}(\xi').
  \end{align*}
    We recall that   $\alpha^s_{l,k-l} \leq
  1$. Replacing $\xi'$ with $\xi'/\|F_l\|$ we obtain the first term on the right-hand side of~\eqref{eq_technical}.
  
   Now we take $\xi_0''$ into account. Observe that we used
  $(P^+_l\otimes P^+_{k-l})(\id_{l-1}\otimes t_1\otimes\id_{k-l-1})$ to embed
  $\Hil_{l-1}\otimes \Hil_{k-l-1}$ into $\Hil_l\otimes \Hil_{k-l}$ as the orthogonal of
  $\Hil_k$, and on each irreducible subspace
  $P_m^{l-1,k-l-1} (\Hil_{l-1}\otimes \Hil_{k-l-1})$ equivalent to $\Hil_m$, with
  $m=k-2-2b$, this morphism has norm controlled as follows
  \cite[Proposition~2.3]{Vergnioux_Cayley}:
  \begin{align*}
    \|(P^+_l\otimes P^+_{k-l})(\id_{l-1}\otimes t_1\otimes\id_{k-l-1})
    P_m^{l-1,k-l-1}\|^2 &= 
    \ts\frac{\qdim(l)\qdim(k-l-1) - \qdim(l-b-1)\qdim(k-l-b-2)}
    {\qdim(l-1) \qdim(k-l-1)} \\
    &= \ts\frac {\qdim(b) \qdim(k-b-1)} {\qdim(l-1)\qdim(k-l-1)} \geq C > 0.
  \end{align*}
  In particular we have
  $\|\xi_0''\| \leq C^{-1} \|\tilde\sigma(\xi)\| \leq C\|F_l\| \|\xi\|.$

  Then we compute the term corresponding to $\xi_0''$ in $w_{n,t}^{k,l}(\xi)$, that is, we
  replace $F_l\bar\xi^{(2)}$ with
  $P_l^+(\bar e_i\otimes \bar\xi_0''^{(1)})$ and $\xi^{(1)}$ with
  $P_{k-l}^+(\bar e_i\otimes \xi_0''^{(2)})$ in~\eqref{eq_technical_transformed}, where
  $(e_i)_i$ is a ONB of $\Hil_1$. The projections $P_l^+$, $P_{k-l}^+$ are absorbed
  in $P_s^+$ and we moreover observe that we have, for $f\in B(\Hil_1)$:
  $\sum \bar e_i^* f \bar e_i = \qTr'_1(f)$. Since
  $\qdim(s-1) (\qTr'_1\otimes\id) (P_s^+) = \qdim(s) P_{s-1}^+$ we obtain
  \begin{align*}
    W'' & = 
    (\qdim t)^{-1/2} {\ts\sum} (\qTr_{n-l}\otimes\id)
    [(\bar\xi_0''^{(1)*}\otimes\id_{s-l}) (\qTr'_1\otimes\id)(P_s^+)
    (\xi_0''^{(2)}\otimes \id_{t-l})] \\
    &= (\qdim t)^{-1/2} {\ts\frac{\qdim(s)}{\qdim(s-1)}} 
    {\ts\sum} (\qTr_{n-l}\otimes\id)
    [(\bar\xi_0''^{(1)*}\otimes\id_{s-l}) P_{s-1}^+
    (\xi_0''^{(2)}\otimes \id_{t-l})].
  \end{align*} 
   Applying backwards the first step of this proof, with $\tilde\sigma(\xi)$
  replaced by $\xi_0''$, we recognize
  \begin{equation*}
    W'' \simeq 
    \Big(\frac{\qdim(t-1)}{\qdim(t)}\Big)^{1/2}\frac{\qdim(s)}{\qdim(s-1)} 
    w_{n-1,t-1}^{k-2,l-1}(\tilde\sigma^{-1}(\xi_0'')),
  \end{equation*}
  with a new error term $E$ with the same control as previously --- except that $\|\xi\|$ is replaced by $\|\tilde\sigma^{-1}(\xi''_0)\| \leq \|F_{l-1}\|\|\xi''_0\|\leq  C \|F_l\|^2 \|\xi\|$ --- and which goes into the third term on the right-hand side of~\eqref{eq_technical} as well.  We finally put
  $\xi'' = \tilde\sigma^{-1}(\xi_0'') / C \|F_l\|^2$ which satisfies
  $\|\xi''\| \leq \|F_{l-1}\| \|\xi_0''\| / C \|F_l\|^2 \leq  \|\xi\|$. Since $\qdim(t-1)/\qdim(t) \leq 1$ and
  $\qdim(s)/\qdim(s-1) \leq C$ this yields the second term on the right-hand side of~\eqref{eq_technical}.
  
  In the case $k=l$ we have already $\tilde\sigma(\xi) = F_k^{-1}\xi \in \Hil_k$ so that $\xi''=0$. In the case $k=2l$, we have  $\tilde\sigma^* t_l = \pm t_l$ (c.f. \eqref{usefulident}), in particular if $t_l^*(\xi)=0$ we get $t_l^*\tilde\sigma(\xi) = 0$. Since $t_l^*(\Hil_{2l}) = \{0\}$ this implies $t_{l-1}^*(\xi''_0)=0$ and finally $t_{l-1}^*(\xi'')$ vanishes because it is proportional to  $t_{l-1}^*\tilde\sigma^{-1}(\xi''_0) = t_{l-1}^*(\xi''_0) = 0$. This crucial property allows in the proof of the next Theorem to initialize the induction at $k=2$ and not $k=0$ (where the required estimate is false).
\end{proof}

We are ready to state and prove the second main result of this section: the proof will exploit the expression obtained in Lemma \ref{lem_decomp} and the induction based on Lemma~\ref{lem_technical}.

\begin{thm} \label{thm_faithful}
  For $N\geq 3$ the boundary action $\beta_\infty$ of $\FO_Q$ is faithful.
\end{thm}

\begin{proof}
  From Lemma \ref{lem_eval_unit} and  Lemma \ref{lem_global_bound} we know that the
  elements $q^nz_n$ of $\sN_\infty$ are uniformly bounded and satisfy
  $p_0 q^nz_n = p_0$. It remains to prove that $p_k q^nz_n \to 0$ for $k>0$
  fixed as $n\to\infty$. For this we use Lemma~\ref{lem_technical}. We put
  $\rho = \|F_1\|$, so that $\|F_l\| = \rho^l$, and we recall that $\rho q<1$.

  More precisely we perform an induction over $k$, with induction assumption
  $(H_k)$: ``there exists a constant $M_k > 0$ such that, for any
  $0\leq l\leq k$, $t\geq n\geq k$ and for any $\xi \in \Hil_{k-l}\otimes \Hil_l$ such that $t_l^*(\xi) = 0$ (when $k=2l$), we have
  $\|w_{n,t}^{k,l}(\xi)\|_\HS \leq M_k \rho^{n/2} \|\xi\|$''. We prove that $(H_{k-2})$ implies $(H_k)$, and at the same time initialize the induction at $k=1$ and $k=2$. 
  
  First of all we ensure that $l\geq 1$ by replacing if necessary
  $w_{n,t}^{k,l}(\xi)$ with
  $w_{n,t}^{k,k-l}((F_l^{1/2}\otimes F_{k-l}^{1/2})\bar\xi)$, c.f.
  Lemma~\ref{lem_easy}. Since $s \subset t\otimes k$ the resulting factor
  $\qdim(s)^{1/2}/\qdim(t)^{1/2}$ is controlled from above by $\qdim(k)^{1/2}$,
  which we incorporate into $M_k$. Then we apply Lemma~\ref{lem_technical}. When
  $k=2l$ we assume that $t_l^*(\xi) = 0$, hence we have also
  $t_{l-1}^*(\xi'') = 0$ and we can apply $(H_{k-2})$ to
  $w_{n-1,t-1}^{k-2,l-1}(\xi'')$, obtaining for all $t\geq n\geq k$:
  \begin{align*}
    \|w_{n,t}^{k,l}(\xi)\|_\HS &\leq  \rho^l \Big({\ts\frac{\qdim(t-l)}{\qdim(t)}}\Big)^{1/2} \|w_{n-l,t-l}^{k,l}(\xi')\|_\HS + 
    C \rho^{2l} M_{k-2} \rho^{n/2}\|\xi''\|  + C q^{-3l/2} \rho^{2l} q^{n/2} \|\xi\| \\
    &\leq \rho^l \Big({\ts\frac{\qdim(t-l)}{\qdim(t)}}\Big)^{1/2} \|w_{n-l,t-l}^{k,l}(\xi')\|_\HS + 
       M'_k \rho^{n/2}\|\xi\|,
  \end{align*}
  where $M'_k = C\rho^{2k} M_{k-2} + C \rho^{2k}q^{-3k/2}$. Note that if $k=1$ we are in the case $k=l=1$ hence $\xi''=0$. The same applies if $k=2$ and $l=2$. If $k=2$ and $l=1$, we have $\xi''\in \Hil_0\otimes \Hil_0 = \CC$ and $t_0^*(\xi'') = 0$ hence $\xi''=0$ as well. As a result we do not need to use $(H_0)$ (which is false by Lemma~\ref{lem_eval_unit}).
  
  Then we apply the above inequality repeatedly until $n < l$. The index $n$ takes the values $n - pl$ with $p$ ranging
  from $0$ to $r := \lfloor n/l\rfloor$. Using the inequality $\qdim(t-pl)/\qdim(t) \leq C q^{pl}$ we obtain
  \begin{align*}
    \|w_{n,t}^{k,l}(\xi)\|_\HS &\leq {\ts\sum_{p=0}^r} ~ \rho^{pl} \Big({\ts\frac{\qdim(t-pl)}{\qdim(t)}}\Big)^{1/2} 
    M'_k \rho^{(n-pl)/2}\|\xi\| + \rho^{l(r+1)}q^{l(r+1)/2}\|w_{n-rl,t-rl}^{k,l}(\xi)\|_\HS \\
    &\leq C M'_k \Big({\ts\sum_{p=0}^\infty} (\rho q)^{pl/2}\Big) \rho^{n/2}  
    \|\xi\| + C \rho^{n/2}\rho^{k/2} q^{-k} \|\xi\|.
\end{align*}
For the second term we used the inequalities
$\rho q<1 $, $lr\leq n$ and the ``easy'' upper bound from Lemma~\ref{lem_easy}:
$\|w_{n-ql,t-ql}^{k,l}(\xi)\|_\HS \leq C q^{-(n-ql)/2}q^{-l/2}\|\xi\| \leq C q^{-k}
\|\xi\|$. We have now proved $(H_k)$.

Now we can come back to $p_kz_n$, for $k\geq 1$. By Cauchy-Schwarz we have:
\begin{displaymath}
|\qTr_{t-n}(w_{n,t}^{k,l}(\zeta)^*w_{n,t}^{\prime\,k,l}(\xi))| \leq \|w_{n,t}^{k,l}(\zeta)\|_\HS \|w_{n,t}^{\prime\,k,l}(\xi)\|_\HS
\leq CM_k q^{-k/2} \rho^{n/2} q^{-n/2} \|\zeta\| \|\xi\|.
\end{displaymath}
We used Lemma~\ref{lem_easy} for the term in $w'$, and the property $(H_k)$ for the term in $w$. Now we apply Lemma~\ref{lem_decomp} and recall that $\kappa_s^{k,t} \leq D_2 q^{l/2} \leq D_2$.
Taking the limit $t\to\infty$, the sum over $0\leq l\leq k$ and the $\sup$ over
$\zeta$, $\xi\in H_k$ of norm $1$ this yields
$\|p_k q^nz_n\| \leq CD_2^2(k+1)M_k q^{-k/2} \rho^{n/2} q^{n/2}$. Since $\rho q < 1$ we have proved that $p_k q^nz_n \to 0$ as $n\to \infty$.
\end{proof}

Theorem \ref{thm:faith-Fur-bnd-->unq-trc} now yields immediately the following corollary. In the unimodular case it already appeared in \cite{VaesVergnioux}, whereas the non-unimodular case of \cite[Theorem~7.2]{VaesVergnioux} deals with the modular group of the Haar state instead of the scaling group. Recall that when $Q$ is normalized by $Q \bar Q = \pm I_N$, $\FO_Q$ is unimodular {\bf iff} $Q$ is unitary.

\begin{cor}
Let $N\geq 3$ and let $Q \in M_N(\CC)$ be such that $Q\bar Q = \pm I_N$. If $Q$ is unitary the $\C^*$-algebra $\C_r^*(\FO_Q) = C(O^+_Q)$ admits a unique trace, and otherwise  it does not admit any KMS state for the scaling automorphism group at the inverse temperature $1$.
\end{cor}

\section{Open questions}

In this short section we gather and discuss some of the open question arising from our work. We begin with some questions related to Section 3, and concepts
related to relative amenability/coamenability. 

\begin{quest}
Suppose that $\G$ is a compact quantum group with a closed quantum subgroup $\QH$. Is $\QG/\QH$ coamenable if and only if $\ell^\infty(\hh \QH)$ is relatively amenable in $\ell^\infty(\hh \QG)$?
\end{quest}

We know the equivalence above holds for \emph{normal} quantum subgroups (so for example if $\QG$ is a dual of a classical discrete group) -- see Theorem \ref{thm:rel_amen_coamen}. The positive answer would generalize the duality between (non-relative) amenability and coamenability from \cite{Tomatsu06} and would provide another  description of the co-kernel of the Furstenberg boundary action for any discrete quantum group.

\begin{quest}
Does every compact quantum group $\QG$ admits a smallest closed quantum subgroup $\QH$ such that $\QG/\QH$ is coamenable?
\end{quest}
One possible approach to this question is the following:
let $\QG$ be a compact quantum group and $\QH_i\subset \QG$, $i = 1,2$ be a pair of quantum subgroups such that $\pi_i\in\Mor(C^u(\QG),C^u(\QH_i))$ admit the reduced versions
$\tilde\pi_i\in\Mor(\C(\QG),\C(\QH_i))$. Denote by $\tilde I_i = \ker\tilde\pi_i$, $i=1,2$. Let $I$ be the  ideal in $\C(\QG)$ generated by $\tilde{I}_1$ and $\tilde{I}_2$. Is it then true that $\I\notin I$? If the answer is positive, so is the answer to the last displayed question, and we would have a well-defined notion of the `co-amenable co-radical' of a compact quantum group.

\begin{quest}
Must the co-kernel of the Furstenberg boundary action of a discrete quantum group $\GG$ be a normal coideal in $\ell^\infty(\hh \GG)$? 
\end{quest}

A positive answer to the above question would yield a simple description of the co-kernel of the Furstenberg boundary action for any discrete quantum group and would in particular imply that the Furstenberg boundary action of $\GG$ is faithful if and only if the amenable radical of $\GG$ is trivial.

We now pass to a series of questions regarding the  operator algebras associated to discrete quantum groups. Note that for classical discrete groups all the answers below are positive, but some have been obtained only recently.

\begin{quest}\label{Q:C*-simple-subgrp}
	Let $\G$ be a discrete quantum group and let $\LLambda\subset \G$ be a normal quantum subgroup. Is there any relation between $\cst$-simplicity or the unique trace property (in the unimodular case) of $\G$ and $\LLambda$?
\end{quest}

\begin{quest}\label{Q:unq-trc-->fthful-bnd?}
	Does the converse of Theorem \ref{thm:faith-Fur-bnd-->unq-trc} hold, at least in the unimodular case? That is, does the unique trace property of a unimodular discrete quantum group $\G$ imply that the action $\G\act \C(\fb)$ is faithful?
\end{quest}

\begin{quest}\label{Q:C*-simple-->unq-tr?}
	Suppose that $\G$ is unimodular. Does  $\cst$-simplicity of $\G$ imply that it has the unique trace property?
\end{quest}

The first two questions are naturally related to the following one, which ends our list.

\begin{quest}\label{Q:Furs-bnd-ext-norm-subgrp}
	Let $\G$ be a discrete quantum group and let $\LLambda\subset \G$ be a normal quantum subgroup. Does the action $\LLambda\act C(\partial_F\LLambda)$ extends to an action $\G\act \C(\partial_F\LLambda)$? (In that case it would be automatically a $\G$-boundary by Proposition \ref{prop:subgrp-bnd-->bnd}.)
\end{quest}

\subsection*{Acknowledgements} 
M.K. was partially supported by the NSF Grant DMS-1700259. 
A.S. was partially supported by the National Science Center (NCN) grant no.~2014/14/E/ST1/00525. 
R.V. was partially supported by the ANR project ANR-19-CE40-0002 and the CEFIPRA project 6101-1. R.V. thanks Holger Reich and the Algebraic Topology group at the Freie Universit\"at Berlin for their kind hospitality in the academic year 2019--2020.
P.K.\, and A.S.\, are grateful to the University of Houston and the Brazos Analysis Seminar for hospitality during the visits in 2017 and 2018 when some of the work on the paper took place.
M.K.\, is grateful to the IMPAN for hospitality during his visit to Warsaw in 2018 when some of the work on the paper took place.
 We all thank the referee for a careful reading of our manuscript and several useful comments improving the presentation.


\begin{thebibliography}{6666666666}
	\bib{BV}{article}{
		AUTHOR = {Baaj, S.},
		Author = {Vaes, S.},
		TITLE = {Double crossed products of locally compact quantum groups},
		JOURNAL = {J. Inst. Math. Jussieu},
		VOLUME = {4},
		YEAR = {2005},
		NUMBER = {1},
		PAGES = {135--173}
	}
	
	
	\bib{BSV}{article}{
		author={Baaj, S.},
		author={Skandalis, G.},
		author={Vaes, S.},
		title={Non-semi-regular quantum groups coming from number theory},
		journal={Comm. Math. Phys.},
		volume={235},
		date={2003},
		pages={139--167}
	}
	
	\bib{BrannanCollins}{article}{
		author={Brannan, M.},
		author={Collins, B.},
		title={Highly entangled, non-random subspaces of tensor products from
			quantum groups},
		journal={Comm. Math. Phys.},
		volume={358},
		date={2018},
		number={3},
		pages={1007--1025},
	}
	
	
	\bib{BrannanVergnioux}{article}{
		AUTHOR = {Brannan, M.},
		author = { Vergnioux, R.},
		TITLE = {Orthogonal free quantum group factors are strongly 1-bounded},
		JOURNAL = {Adv. Math.},
		FJOURNAL = {Advances in Mathematics},
		VOLUME = {329},
		YEAR = {2018},
		PAGES = {133--156},
	}
	
	\bib{BKKO}{article}{
		author={Breuillard, E.},
		author={Kalantar, M.},
		author={Kennedy, M.},
		author={Ozawa, N.},
		title={$C^*$-simplicity and the unique trace property for discrete groups},
		journal={Publ. Math. Inst. Hautes \'{E}tudes Sci.},
		volume={126},
		date={2017},
		pages={35--71}
	}
	
	
	
	\bib{IsomorphismTheorems}{article}{
		author={Chirvasitu, A.},
		author={Hoche, S.O.},
		author={Kasprzak, P. },
		title={Fundamental isomorphism theorems for quantum groups},
		journal={Expo. Math.},
		volume={35},
		date={2017},
		number={4},
		pages={390--442},
	}
	
	
	
	\bibitem[Cra17]{Jason} J.\,Crann, On hereditary properties of quantum group amenability, \emph{Proc.\,AMS} \textbf{145} (2017), 627--635.
	
	
	
	
	
	
	\bib{DKSS}{article}{
		author={Daws, M.},
		author={Kasprzak, P.},
		author={Skalski, A.},
		author={So\l tan, P.M.},
		title={Closed quantum subgroups of locally compact quantum groups},
		journal={Adv. Math.},
		volume={231},
		date={2012},
		number={6},
		pages={3473--3501},
	}
	
	
	\bib{Day}{article}{
		author={Day, M.},
		title={Amenable semigroups},
		journal={Illinois Journal of Mathematics},
		volume={1},
		date={1957},
		pages={509-544}
	}
	
	\bib{Kenny}{article}{
		author={De Commer, K.},
		author={Kasprzak, P. },
		author={Skalski, A.},
		author={So\l tan, P.M.},
		title={Quantum actions on discrete quantum spaces and a generalization of
			Clifford's theory of representations},
		journal={Israel J. Math.},
		volume={226},
		date={2018},
		number={1},
		pages={475--503},
	}
	
	
	\bib{EffrosRuan}{book}{
		AUTHOR = {Effros, E. G.}, 
		author={ Ruan, Z.-J.},
		TITLE = {Operator spaces},
		SERIES = {London Mathematical Society Monographs. New Series},
		VOLUME = {23},
		PUBLISHER = {The Clarendon Press, Oxford University Press, New York},
		YEAR = {2000},
		PAGES = {xvi+363},
		ISBN = {0-19-853482-5},
		MRCLASS = {46L07 (46B28 47L25)},
		MRNUMBER = {1793753},
		MRREVIEWER = {Christian Le Merdy},
	}
	
	\bib{EnockSchwartz}{book}{
		AUTHOR = {Enock, M.}
		author = {Schwartz, J.-M.},
		TITLE = {Kac algebras and duality of locally compact groups},
		NOTE = {With a preface by Alain Connes,
			With a postface by Adrian Ocneanu},
		PUBLISHER = {Springer-Verlag, Berlin},
		YEAR = {1992},
		PAGES = {x+257},
	}
	
	\bib{FreslonVergnioux}{article}{
		author={Freslon, A.},
		author={Vergnioux, R.},
		title={The radial MASA in free orthogonal quantum groups},
		journal={J. Funct. Anal.},
		volume={271},
		date={2016},
		number={10},
		pages={2776--2807},
	}
	
	
	
	
	\bib{Furman}{article}{
		author={Furman, A.},
		title={On minimal strongly proximal actions of locally compact groups},
		journal={Israel Journal of Mathematics},
		volume={136},
		date={2003},
		pages={173-187}
	}
	
	
	
\bib{Ham79b}{article}{
author={Hamana, M.},
title={Injective envelopes of operator systems},
journal={Publications of the Research Institute for Mathematical Sciences},
volume={15},
date={1979},
pages={773-785}
}

\bib{Ham85}{article}{
author={Hamana, M.},
title={Injective envelopes of C*-dynamical systems},
journal={Tohoku Mathematical Journal},
volume={37},
date={1985},
pages={463-487}
}
	
	
	
	
	\bibitem[HK]{YairMehrdad} Y.\,Hartman and M.\,Kalantar, \emph{Stationary $C^*$-dynamical systems}, J.\ of the EMS, to appear.
	
	
	\bib{Izumi} {article}{
		AUTHOR = {Izumi, M.},
		TITLE = {Non-commutative {P}oisson boundaries and compact quantum group
			actions},
		JOURNAL = {Adv. Math.},
		FJOURNAL = {Advances in Mathematics},
		VOLUME = {169},
		YEAR = {2002},
		NUMBER = {1},
		PAGES = {1--57},
	}
	
	
	

	
	\bib{Induction}{article}{
		author={Kalantar, M.},
		author={Kasprzak, P.},
		author={Skalski, A.},
		author={So\l tan, P.},
		title={Induction for locally compact quantum groups revisited},
		journal={Proceedings of the Royal Society of Edinburgh Section A: Mathematics},
			volume={150},
		date={2020},
		number={2},
		pages={1071--1093},
	}
	
	\bib{KalKen}{article}{
		author={Kalantar, M.},
		author={Kennedy, M.},
		title={Boundaries of reduced $C^*$-algebras of discrete groups},
		journal={J. Reine Angew. Math. (Crelle's Journal)},
		volume={727},
		date={2017},
		pages={247-267}
	}
	
	
	
	\bib{PawelPiotr}{article}{
		author={Kasprzak, P.},
		author={So\l tan, P.M.},
		title={Embeddable quantum homogeneous spaces},
		journal={J. Math. Anal. Appl.},
		volume={411},
		date={2014},
		number={2},
		pages={574--591},
	}
	
	\bib{KSprojections}{article}{
		author={Kasprzak, P.},
		author={So\l tan, P.M.},
		title={Quantum groups with projection and extensions of locally compact quantum groups},
		journal = { J. Noncommut. Geom.},
		volume={14},
		date={2020},
		number={1},
		pages={105--123},
	}
	

	\bib{KauffmannLins}{book}{
		author={Kauffman, L.H.},
		author={Lins, S.},
		title={Temperley-Lieb Recoupling Theory and Invariants of $3$-Manifolds},
		series={Annals of Mathematics Studies},
		volume={134},
		publisher={Princeton University Press},
		address={Princeton, NJ},
		date={1994}
	}
	
	
	\bib{Johanuniv}{article}{
		author={Kustermans, J.},
		title={Locally compact quantum groups in the universal setting},
		journal={Internat. J. Math.},
		volume={12},
		date={2001},
		number={3},
		pages={289--338},
	}
	
	
	\bib{KusNotes}{incollection}{
		AUTHOR = {Kustermans, J.},
		TITLE = {Locally compact quantum groups},
		BOOKTITLE = {Quantum independent increment processes. {I}},
		SERIES = {Lecture Notes in Math.},
		VOLUME = {1865},
		PAGES = {99--180},
		PUBLISHER = {Springer, Berlin},
		YEAR = {2005}
	}
	
	

	\bib{KV}{article}{
		author={Kustermans, J.},
		author={Vaes, S.},
		title={Locally compact quantum groups},
		journal={Ann. Sci. \'Ec. Norm. Sup\'er.},
		volume={33},
		date={2000},
		number={6},
		pages={837--934},
	}
	
	\bib{LB}{article} {
		AUTHOR = {Le Boudec, A.},
		TITLE = {{$C^*$}-simplicity and the amenable radical},
		JOURNAL = {Invent. Math.},
		FJOURNAL = {Inventiones Mathematicae},
		VOLUME = {209},
		YEAR = {2017},
		NUMBER = {1},
		PAGES = {159--174},
	}
	
	
	\bib{SLW12}{article}{
		author={Meyer, R. },
		author={Sutanu, R.},
		author={Woronowicz, S.L.},
		title={Homomorphisms of quantum groups},
		journal={M\"unster J. of Math.},
		volume={5},
		date={2012},
		pages={1--24},
	}
	
	
	
	\bibitem[NeY14]{NeshYam} S.\,Neshveyev and M.\,Yamashita, Categorical duality for Yetter-Drinfeld algebras, \emph{Doc. Math.}, \textbf{19}
	(2014), 1105--1139.
	
	
	
	
	\bib{Ozawa}{article}{
		AUTHOR = {Ozawa, N.},
		TITLE = {Boundaries of reduced free group {$C^\ast$}-algebras},
		JOURNAL = {Bull. Lond. Math. Soc.},
		FJOURNAL = {Bulletin of the London Mathematical Society},
		VOLUME = {39},
		YEAR = {2007},
		NUMBER = {1},
		PAGES = {35--38},
	}
	
	\bib{Paulsen}{article}{
		AUTHOR = {Paulsen, V.I.},
		TITLE = {Weak expectations and the injective envelope},
		JOURNAL = {Trans. Amer. Math. Soc.},
		FJOURNAL = {Transactions of the American Mathematical Society},
		VOLUME = {363},
		YEAR = {2011},
		NUMBER = {9},
		PAGES = {4735--4755},
		}
	
	
	\bibitem[Rau19]{Raum} S.\,Raum, $C^*$-simplicity
	[after Breuillard, Haagerup, Kalantar, Kennedy and Ozawa], \emph{S\'eminaire Bourbaki} \textbf{1158}, 2019.
	
	
	
	\bib{Rieffel}{article}{
		AUTHOR = {Rieffel, M.A.},
		TITLE = {Induced representations of {$C^{\ast} $}-algebras},
		JOURNAL = {Advances in Math.},
		FJOURNAL = {Advances in Mathematics},
		VOLUME = {13},
		YEAR = {1974},
		PAGES = {176--257},
	}
	
	
	\bib{Tomatsu06}{article}{
		author={Tomatsu, R.},
		title={Amenable discrete quantum groups},
		journal={J. Math. Soc. Japan},
		volume={58},
		date={2006},
		number={4},
		pages={949--964},
	}
	
	
	
	\bib{Tomquot} {article}{
		AUTHOR = {Tomatsu, R.},
		TITLE = {A characterization of right coideals of quotient type and its
			application to classification of {P}oisson boundaries},
		JOURNAL = {Comm. Math. Phys.},
		FJOURNAL = {Communications in Mathematical Physics},
		VOLUME = {275},
		YEAR = {2007},
		NUMBER = {1},
		PAGES = {271--296},
		
	}
	
	
	\bib{Tomatsu} {article}{
		AUTHOR = {Tomatsu, R.},
		TITLE = {Poisson boundaries of discrete quantum groups},
		BOOKTITLE = {Noncommutative harmonic analysis with applications to
			probability {II}},
		SERIES = {Banach Center Publ.},
		VOLUME = {89},
		PAGES = {297--312},
		PUBLISHER = {Polish Acad. Sci. Inst. Math., Warsaw},
		YEAR = {2010},
	}
	
	
	\bib{Tomiyama}{article}{
		AUTHOR = {Tomiyama, J.},
		TITLE = {Applications of {F}ubini type theorem to the tensor products
			of {$C^{\ast} $}-algebras},
		JOURNAL = {Tohoku Math. J. (2)},
		FJOURNAL = {The Tohoku Mathematical Journal. Second Series},
		VOLUME = {19},
		YEAR = {1967},
		PAGES = {213--226},
		}
	
	
	\bib{VainVaes}{article}{
		author={Vaes, S.},
		author={Vainerman, L.},
		title={Extensions of locally compact quantum groups and the bicrossed product construction},
		journal={Adv. Math.},
		volume={175},
		date={2006},
		number={4},
		pages={1--101},
	}
	
	
	\bib{VaesVDW}{article}{
		AUTHOR = {Vaes, S.},
		author = {Vander Vennet, N.},
		TITLE = {Poisson boundary of the discrete quantum group
			{$\widehat{A_u(F)}$}},
		JOURNAL = {Compos. Math.},
		FJOURNAL = {Compositio Mathematica},
		VOLUME = {146},
		YEAR = {2010},
		NUMBER = {4},
		PAGES = {1073--1095},
	}
	
	
	
	
	\bib{VaesVergnioux}{article}{
		author={Vaes, S.},
		author={Vergnioux, R.},
		title={The boundary of universal discrete quantum groups, exactness, and
			factoriality},
		journal={Duke Math. J.},
		volume={140},
		date={2007},
		number={1},
		pages={35--84},
	}
	
	\bib{VDW}{article}{
		AUTHOR = {Van Daele, A.},
		author = {Wang, S.},
		TITLE = {Universal quantum groups},
		JOURNAL = {Internat. J. Math.},
		FJOURNAL = {International Journal of Mathematics},
		VOLUME = {7},
		YEAR = {1996},
		NUMBER = {2},
		PAGES = {255--263},
	}
	
	
	\bib{Vergnioux_Cayley}{article}{
		AUTHOR = {Vergnioux, R.},
		TITLE = {Orientation of quantum {C}ayley trees and applications},
		JOURNAL = {J. Reine Angew. Math.},
		FJOURNAL = {Journal f\"{u}r die Reine und Angewandte Mathematik. [Crelle's
			Journal]},
		VOLUME = {580},
		YEAR = {2005},
		PAGES = {101--138},
		
	}
	
	
	\bib{Vergnioux_Decay}{article}{
		author={Vergnioux, R.},
		title={The property of rapid decay for discrete quantum groups},
		journal={J. Operator Theory},
		volume={57},
		date={2007},
		number={2},
		pages={303--324},
	}
	
		
	\bib{Vergnioux_Amalgamated}{article}{
    AUTHOR = {Vergnioux, R.},
     TITLE = {{$K$}-amenability for amalgamated free products of amenable
              discrete quantum groups},
   JOURNAL = {J. Funct. Anal.},
  FJOURNAL = {Journal of Functional Analysis},
    VOLUME = {212},
      YEAR = {2004},
    NUMBER = {1},
     PAGES = {206--221},
	}
	
	

	
	
	
	
	
	\bibitem[Wan09]{Wang} S.\,Wang, Simple compact quantum groups I, \emph{J.\,Funct.\,Anal.} \textbf{256} (2009), no.\,10, 3313--3341. 
	
	
	\bib{Woronowicz_Matrix}{article}{
   author={Woronowicz, S.L.},
   title={Compact matrix pseudogroups},
   journal={Comm. Math. Phys.},
   volume={111},
   date={1987},
   number={4},
   pages={613--665},
   issn={0010-3616},
}

\bib{Wor}{article}{
   author={Woronowicz, S.L.},
   title={Compact quantum groups},
   conference={
      title={Sym\'{e}tries quantiques},
      address={Les Houches},
      date={1995},
   },
   book={
      publisher={North-Holland, Amsterdam},
   },
   date={1998},
   pages={845--884},
}


\end{thebibliography}
\end{document}